\documentclass[12pt,a4paper]{amsart} 
\usepackage[utf8]{inputenc}
\usepackage[english]{babel}

\usepackage{amsmath}
\usepackage{amsfonts}
\usepackage{amssymb}
\usepackage{color}

\usepackage{amsthm}

\newtheorem{theorem}{Theorem} [section]

\newtheorem{corollary}[theorem]{Corollary} 
\newtheorem{lemma}[theorem]{Lemma}
\newtheorem{conjecture}[theorem]{Conjecture}
\newtheorem{open}{Open Question}

\newcounter{defno}

\usepackage{graphics}
\usepackage{epsfig}

\usepackage{url}
\usepackage{hyperref}

\usepackage[a4paper,top=2.5cm,bottom=2.8cm,
left=3.2cm,right=3.2cm]{geometry}
\markleft{\hfill \textsc{The Shape of Central Quadrilaterals} \hfill}
\newcommand{\degrees}{^\circ}
\long\def\void#1{}
\begin{document}
% ===================
International Journal of  Computer Discovered Mathematics (IJCDM) \\
ISSN 2367-7775 \copyright IJCDM \\
Volume 7, 2022, pp. xx--yy  \\
Received xx January 2022. Published on-line xx mmm 2022 \\ 
web: \url{http://www.journal-1.eu/} \\

\copyright The Author(s) This article is published 
with open access.\footnote{This article is distributed under the terms of the Creative Commons Attribution License which permits any use, distribution, and reproduction in any medium, provided the original author(s) and the source are credited.} \\
% ===========================   
\bigskip
\bigskip

\begin{center}
	{\Large \textbf{The Shape of Central Quadrilaterals}} \\
	\medskip
	\bigskip
        \bigskip

	\textsc{Stanley Rabinowitz$^a$ and Ercole Suppa$^b$} \\

	$^a$ 545 Elm St Unit 1,  Milford, New Hampshire 03055, USA \\
	e-mail: \href{mailto:stan.rabinowitz@comcast.net}{stan.rabinowitz@comcast.net}\footnote{Corresponding author} \\
	web: \url{http://www.StanleyRabinowitz.com/} \\
	
	$^b$ Via B. Croce 54, 64100 Teramo, Italia \\
	e-mail: \href{mailto:ercolesuppa@gmail.com}{ercolesuppa@gmail.com} \\
	web: \url{https://www.esuppa.it} \\

\bigskip
%draft revised: Jan. 12, 2022

\end{center}
\bigskip
\bigskip

% ==============================
\textbf{Abstract.} 
The diagonals of a quadrilateral form four component triangles (in two ways).
For each of various shaped quadrilaterals, we examine 1000 triangle centers
located in these four component triangles.
Using a computer, we determine when the four
centers form a special quadrilateral, such as a rhombus or a cyclic quadrilateral.
A typical result is the following.
The diagonals of an equidiagonal quadrilateral divide the quadrilateral
into four nonoverlapping triangles. Then the Nagel points
of these four triangles form an orthodiagonal quadrilateral.



\def\T{^{\rm T}}

%\font\bigf=cmb10 at 16pt

\bigskip
\bigskip
% ================================
% 1 Introduction 
% ================================
%
\section{Introduction}
\label{section:introduction}

\vspace{-4pt}
Consider the following two results.
\vspace{-3pt}
\begin{center}
\fbox{
\begin{minipage}{2.55in}{
Result 1
\begin{center}
{\includegraphics[width=0.8\linewidth]{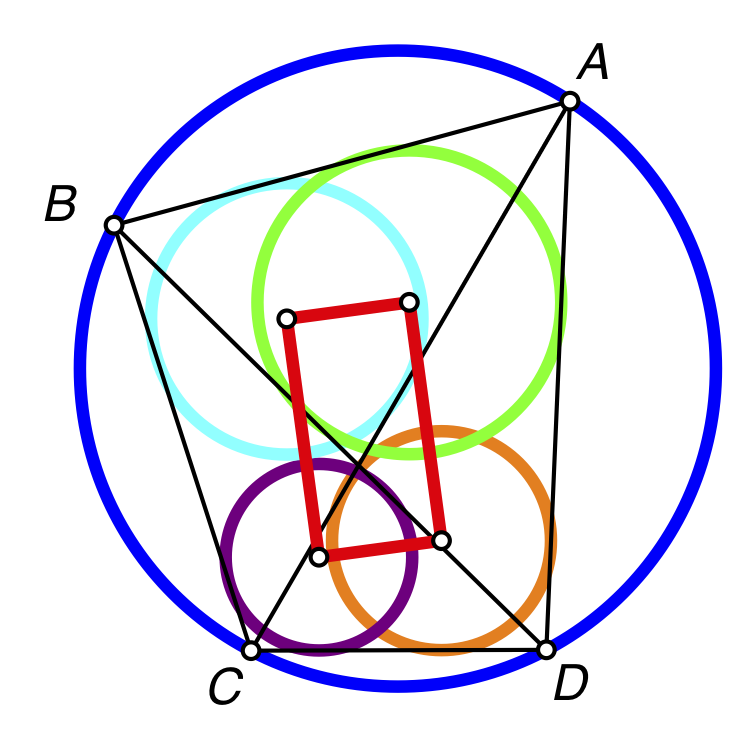}}\\
\end{center}
\hrule
\vspace*{4pt}
Let $ABCD$ be a cyclic quadrilateral.\\
Then the incenters of $\triangle BCD$, $\triangle ACD$, $\triangle ABD$, and $\triangle ABC$
form a rectangle.
}
\end{minipage}
}%
\quad
\fbox{
\begin{minipage}{2.55in}{
Result 2
\begin{center}
{\includegraphics[width=0.782\linewidth]{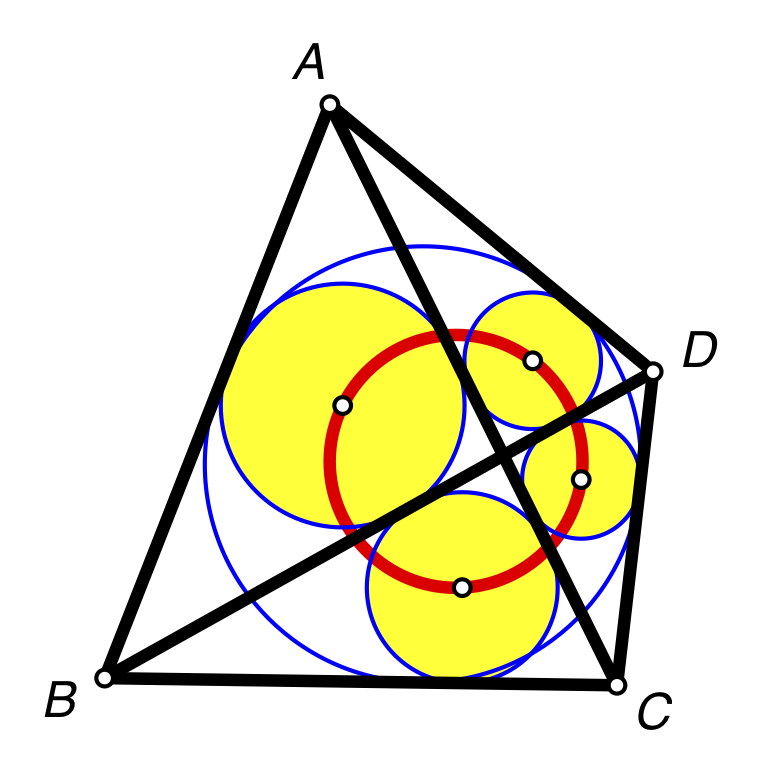}}\\
\end{center}
\hrule
\vspace*{4pt}
Let $ABCD$ be a tangential quadrilateral.
Then the incenters of the 4 nonoverlapping triangles formed by
the diagonals lie on a circle.
}
\end{minipage}
}
\end{center}
\vspace{-3pt}

At first glance, these two results seem unrelated. However, the following
table points out their similarities.

\begin{center}
\begin{tabular}{|c|l|l|l|l|}
\hline
&&component&type of&central\\
Result&quadrilateral&triangles&center&quadrilateral\\
\hline
1&cyclic&half triangles&incenter&rectangle\\
\hline
2&tangential&quarter triangles&incenter&cyclic quadrilateral\\
\hline
\end{tabular}
\end{center}

The following definitions are used.
A \emph{cyclic quadrilateral} is a quadrilateral with a circumcircle.
A \emph{tangential quadrilateral} is a quadrilateral with an incircle.
The diagonals of a convex quadrilateral divide it into four overlapping triangles
each bounded by two sides of the quadrilateral and a diagonal.
These are called \emph{half triangles}.
The diagonals of a convex quadrilateral divide it into four nonoverlapping triangles
each bounded by a side of the quadrilateral and parts of the two diagonals.
These are called \emph{quarter triangles}.
The \emph{incenter} of a triangle is the center of the circle inscribed in the triangle.
The four incenters form a quadrilateral called the \emph{central quadrilateral}.
The original quadrilateral is called the \emph{reference quadrilateral}.
Information about result 1 can be found in \cite[p.~133]{Altshiller-Court}.
Information about result 2 can be found in \cite{Josefsson2011}.

Looking at the table, we note the similarities. In both cases, we started with
a quadrilateral with a special shape. We then considered four component triangles
associated with the quadrilateral (half triangles or quarter triangles).
Within each component triangle, we located a center, namely the incenter.
We concluded that the four centers form the vertices of a quadrilateral
that also have a special shape (rectangle or cyclic quadrilateral).

In this paper, we use a computer to discover similar results. In each case,
we will start with a reference quadrilateral that has a special shape. We then form
four component triangles related to that quadrilateral. Then we locate
a triangle center in each component triangle. These four centers form a quadrilateral
known as the \emph{central quadrilateral}. Finally, we check to see
if the central quadrilateral has a special shape.
\goodbreak
\section{Types of Quadrilaterals Studied}
\label{section:quadrilaterals}

We are only interested in quadrilaterals that have a certain amount of symmetry.
For example, we excluded bilateral quadrilaterals (those with two equal sides),
bisect-diagonal quadrilaterals (where one diagonal bisects another), right kites,
right trapezoids, and golden rectangles.
The types of quadrilaterals we studied are shown in the following table.
The sides of the quadrilateral, in order, have lengths $a$, $b$, $c$, and $d$.
The measures of the angles of the quadrilateral, in order, are $A$, $B$, $C$, and $D$.

\begin{center}
\footnotesize
\begin{tabular}{|l|l|l|}\hline
\multicolumn{3}{|c|}{\small\textbf{\large \strut Types of Quadrilaterals Considered}}\\ \hline
Quadrilateral Type&Geometric Definition&Algebraic Condition\\ \hline
general&convex&none\\ \hline
cyclic&has a circumcircle&$A+C=B+D$\\ \hline
tangential&has an incircle&$a+c=b+d$\\ \hline
extangential&has an excircle&$a+b=c+d$\\ \hline
parallelogram&opposite sides parallel&$a=c$, $b=d$\\ \hline
equalProdOpp&product of opposite sides equal&$ac=bd$\\ \hline
equalProdAdj&product of adjacent sides equal&$ab=cd$\\ \hline
orthodiagonal&diagonals are perpendicular&$a^2+c^2=b^2+d^2$\\ \hline
equidiagonal&diagonals have the same length&\\ \hline
Pythagorean&equal sum of squares, adjacent sides&$a^2+b^2=c^2+d^2$\\ \hline
%harmonic&cyclic, equal-product&$ac=bd$ plus cyclic\\ \hline
kite&two pair adjacent equal sides&$a=b$, $c=d$\\ \hline
trapezoid&one pair of opposite sides parallel&$A+B=C+D$\\ \hline
rhombus&equilateral&$a=b=c=d$\\ \hline
rectangle&equiangular&$A=B=C=D$\\ \hline
Hjelmslev&two opposite right angles&$A=C=90^\circ$\\ \hline
%right trapezoid&trapezoid with a right angle&$A=B=90^\circ$\\ \hline
isosceles trapezoid&trapezoid with two equal sides&$A=B$, $C=D$\\ \hline
%trilateral&three equal sides&$a=b=c$\\ \hline
%triangular&three equal angles&$A=B=C$\\ \hline
%equiRecipSum&sum of reciprocals of opp sides equal&$1/a+1/c=1/b+1/d$\\ \hline
%cyclicEquiRecipSum&cyclic and equiRecipSum&above plus cyclic\\ \hline
APquad&sides in arithmetic progression&$d-c=c-b=b-a$\\ \hline
%bicentric&cyclic and tangential&$a+c=b+d$ plus cyclic\\ \hline
%exbicentric&cyclic and extangential&$a+b=c+d$ plus cyclic\\ \hline
\end{tabular}
\end{center}

The following combinations of entries in the above list were also considered:
bicentric quadrilaterals (cyclic and tangential), exbicentric quadrilaterals (cyclic and extangential),
bicentric trapezoids, cyclic orthodiagonal quadrilaterals, equidiagonal orthodiagonal kites,
equidiagonal orthodiagonal quadrilaterals, equidiagonal orthodiagonal trapezoids,
harmonic quadrilaterals (cyclic and equalProdOpp), orthodiagonal trapezoids, tangential trapezoids,
and squares (equiangular rhombi).

So, in addition to the general convex quadrilateral, a total of 27 other types of quadrilaterals
were considered in this study. When checking to see if a central quadrilateral has one of these
shapes, we also considered the following two degenerate quadrilaterals: a line segment (the four vertices
of the quadrilateral are collinear) and a point (the four vertices coincide).

A graph of the types of quadrilaterals considered is shown in Figure \ref{fig:quadShapes}.
An arrow from A to B means that any quadrilateral of type B is also of type A.
For example: all squares are rectangles and all kites are orthodiagonal.

\begin{figure}[h!t]
\centering
\scalebox{1}[1.5]{\includegraphics[width=1\linewidth]{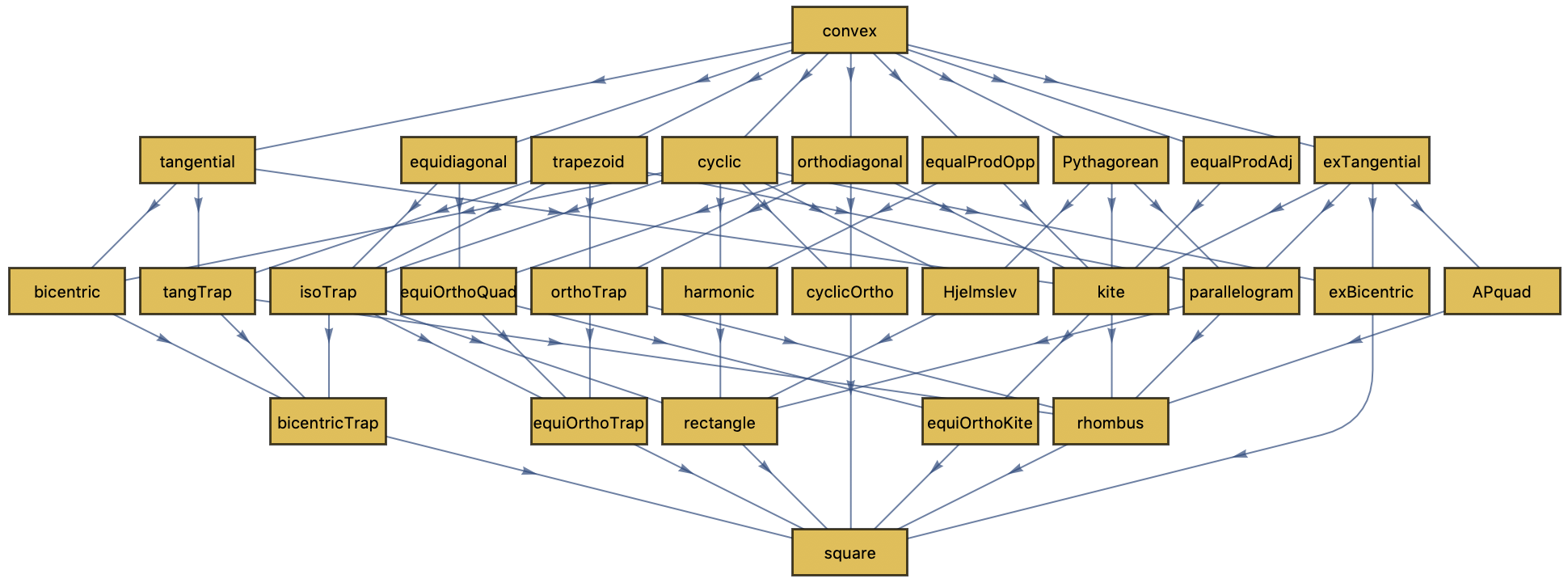}}
\caption{Quadrilateral Shapes}
\label{fig:quadShapes}
\end{figure}

\section{Centers}
\label{section:centers}

In this study, we will place triangle centers in the four component
triangles. We use Clark Kimberling's definition of a triangle center \cite{KimberlingA}.

A \emph{center function} is a nonzero function $f(a,b,c)$
homogeneous in $a$, $b$, and $c$ and symmetric in $b$ and $c$.
Homogeneous in $a$, $b$, and $c$ means that
$$f(ta,tb,tc)=t^nf(a,b,c)$$
for some nonnegative integer $n$,
all $t>0$, and all positive real numbers $(a,b,c)$ satisfying $a<b+c$, $b<c+a$, and $c<a+b$.
Symmetric in $b$ and $c$ means that
$$f(a,c,b)=f(a,b,c)$$
for all $a$, $b$, and $c$.

A \emph{triangle center} is an equivalence class $x:y:z$ of ordered triples $(x,y,z)$
given by
$$x=f(a,b,c),\quad y=f(b,c,a),\quad z=f(c,a,b).$$

Tens of thousands of interesting triangle centers have been cataloged in the Encyclopedia of Triangle Centers \cite{ETC}. We use $X_n$ to denote the nth named center in this encyclopedia.

\void{
The diagonals of a quadrilateral (the reference quadrilateral) form four small triangles. Each contains two adjacent
vertices of the quadrilateral plus the point where the diagonals intersect.
We will call these component triangles of the quadrilateral. We will construct triangle
centers of various types in each of these component triangles and then examine
the four points obtained. These four points determine another quadrilateral (the central quadrilateral).
}

\section{Methodology}
\label{section:methodology}

We used a computer program called GeometricExplorer to determine the shape
of the central quadrilateral. Starting with each type of quadrilateral listed in
Figure~\ref{fig:quadShapes} for
the reference quadrilateral, and for each $n$ from 1 to 1000, we placed center $X_n$
in each of the component triangles of the reference quadrilateral.
The program then analyzes the central quadrilateral formed by these four centers
and reports if the central quadrilateral has a special shape.
GeometricExplorer uses numerical coordinates (to 15 digits of precision) for locating
all the points. This does not constitute a proof that the result is correct,
but gives us compelling evidence for the validity of the result.

We then examine the center functions associated with each group of results and use
that to guess at a pattern. Other center functions satisfying that pattern are then checked
and if they too form the same shape central quadrilateral, then we list the result
as a theorem in this paper.
We then used exact symbolic computation using Mathematica to try and
give a formal computer proof of the result.
If a proof is found, we include it with the supplementary material
associated with this paper.

For example, using an arbitrary quadrilateral as the reference quadrilateral,
when we place the center $X_n$ in each of the quarter triangles formed by the diagonals,
the resulting central quadrilateral appears to be a parallelogram for
$n=$2, 3, 4, 5, 20, 140, 376, 382, 546, 547, 548, 549, 550, 631, and 632.

We then compiled a list of the center functions corresponding to each of these centers.
There are multiple possibilities for each center function, since the center function
is not unique.
It can be scaled or expressed in various ways using either algebraic or trigonometric terms.
The list is shown in Table \ref{table:patterns}.

Examining the entries in Table \ref{table:patterns}, we can make several conjectures.
For example, we can conjecture that whenever the center function is of the form
$$\cos B\cos C+k\cos A$$
for some constant $k$, then the central quadrilateral will be a parallelogram.
We test this conjecture by testing the center function $\cos B\cos C+\frac79\cos A$ and
GeometricExplorer returns the fact that the central quadrilateral is a parallelogram.
We now have strong evidence that the result is true.

In this case, a formal computer proof was found
and we list the result later on in this paper as Theorem \ref{thm:general}.

\begin{table}[ht!]
\caption{}
\label{table:patterns}
\begin{center}
\begin{tabular}{|r|l|}
\hline
$n$&center function\\
\hline
2&$\cos A+\cos (B-C)$\\
\hline
3&$\cos A$\\
\hline
4&$\cos A-\cos (B-C)$\\
\hline
5&$\cos (B-C)$\\
\hline
20&$\cos A-\cos B\cos C$\\
\hline
140&$2\cos A+\cos (B-C)$\\
\hline
376&$5\cos A+\cos (B-C)$\\
\hline
381&$2\cos (B-C)-\cos A$\\
\hline
382&$\cos A-4\cos B\cos C$\\
\hline
546&$3\cos B\cos C-2\cos A$\\
\hline
547&$5\cos B\cos C+2\cos A$\\
\hline
548&$6\cos A-\cos (B-C)$\\
\hline
549&$\cos(B-C)+4\cos A$\\
\hline
550&-$\cos(B-C)+4\cos A$\\
\hline
631&$2\cos A+\cos B\cos C$\\
\hline
632&$7\cos A+6\cos B\cos C$\\
\hline
\end{tabular}
\end{center}
\end{table}

If a theorem in this paper is accompanied by a figure, this means that the
figure was drawn using either Geometer's Sketchpad or GeoGebra.
In either case, we used the drawing program to dynamically vary the points
in the figure. Noticing that the result remains true as the points vary offers
further evidence that the theorem is true.

%**************************************
%    Quarter Triangles
%**************************************

\section{Results for Quarter Triangles}
\label{section:quarterTriangles}

In this configuration, the reference quadrilateral is named $ABCD$.
The two diagonals of the quadrilateral divide that quadrilateral
into four small nonoverlapping triangles called quarter triangles.
The point of intersection of the diagonals
will be called $E$. The four triangles (numbered 1 to 4) are shown
in Figure \ref{fig:quarterTriangles}.

\begin{figure}[h!t]
\centering
\includegraphics[width=0.4\linewidth]{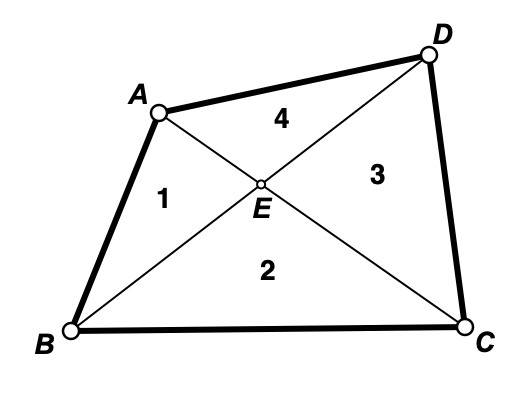}
\caption{Quarter Triangles}
\label{fig:quarterTriangles}
\end{figure}

The triangles have been numbered in a counterclockwise order starting with side $AB$:
$\triangle ABE$, $\triangle BCE$, $\triangle CDE$, $\triangle DAE$.
Triangle centers are selected in each triangle.
In order, their names are $F$, $G$, $H$, and $I$.

The raw data collected can be
found in Appendix \ref{appendix:diagonalPointData}.
Looking for patterns in the raw data, we found a number of theorems and
conjectures which are presented below.

%\newpage
\subsection{General Quadrilaterals}

\void{
\begin{theorem}
\label{thm:general}
If the reference quadrilateral is a general quadrilateral, the central quadrilateral is a parallelogram whenever the chosen center lies on the Euler line.
\end{theorem}

The \emph{Euler line} of a triangle is the line joining the centroid and the circumcenter.
According to \cite{ETC}, a center function for the centroid is $\cos B\cos C+\cos A$
and a center function for the circumcenter is $\cos A$.

If $P$ and $Q$ are two distinct points, then any point on the line joining $P$ and $Q$
is of the form $P+k(Q-P)$ where $k$ is a real parameter. Letting
$P=\cos A$ and $Q=\cos B\cos C+\cos A$, we see that any point on the Euler line
of a triangle is of the form $\cos A+k\cos B\cos C$ for some constant $k$.

In other words, every point on the Euler line is given by the trilinear coordinates
$$\bigl(\cos A+k\cos B\cos C:\cos B+k\cos C\cos A:\cos C+k\cos A\cos B\bigr)$$
for some constant $k$.

Thus, Theorem \ref{thm:general} is equivalent to the following theorem.
}

\begin{theorem}
\label{thm:general}
If the reference quadrilateral is a general quadrilateral, the central quadrilateral is a parallelogram whenever the chosen center has center function of the form
$$\cos B\cos C+k\cos A$$
\index{$\cos B\cos C+k\cos A$}
for some constant $k$.
\end{theorem}

\begin{proof}

Let $ABCD$ be the quadrilateral. Let $E$ be the intersection of the diagonals.
Without loss of generality, we can place $B$ at the origin
and $C$ at the point with Cartesian coordinates $(1,0)$. Then let the coordinates of $A$, $D$, and $E$ be as shown in Figure \ref{fig:dpCoordinates}.

\begin{figure}[h!t]
\centering
\includegraphics[width=0.4\linewidth]{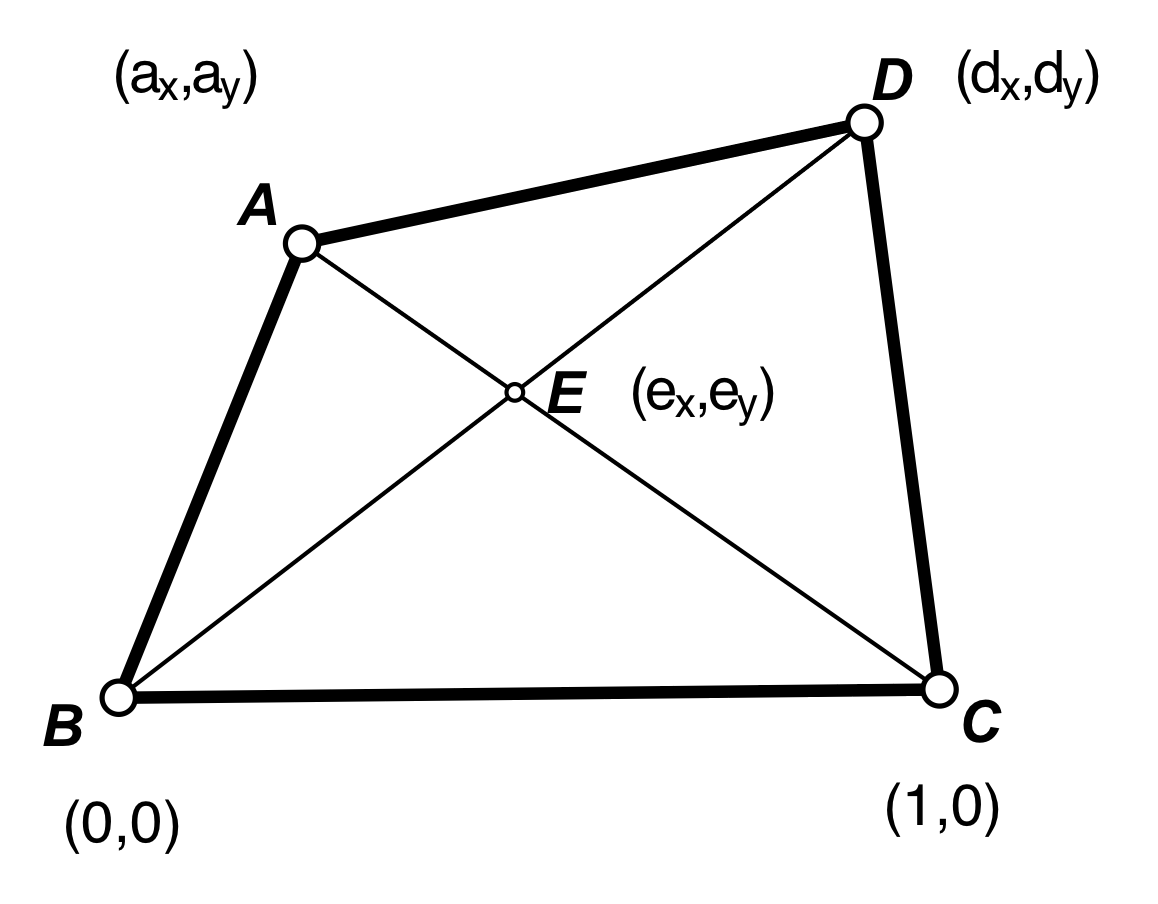}
\caption{Cartesian coordinates for a general quadrilateral}
\label{fig:dpCoordinates}
\end{figure}

Using the equations for lines $AC$ and $BD$, we can find the coordinates of their intersection point, $E$.
We get
$$E=\left(\frac{a_yd_x}{a_yd_x+d_y-a_xd_y},\frac{a_yd_y}{a_yd_x+d_y-a_xd_y}\right).$$

Using the Law of Cosines, we can replace the cosines in the center function $\cos B\cos C+k\cos A$ to get an equivalent function in terms of $a$, $b$, and $c$.
This gives us the 1st trilinear coordinate for the triangle center
determined by this center function in terms of $a$, $b$, and $c$. Multiplying by $a$ gives the 1st barycentric
coordinate. We can clear fractions by multiplying the barycentric coordinates by the cyclic function $4abc$.
We then find that the 1st barycentric coordinate for a point given by this center function is
$$a^4 (1-2 k)+2 a^2 k \left(b^2+c^2\right)-\left(b^2-c^2\right)^2.$$

Using the formula for the distance between two points in Cartesian coordinates, we can find the lengths of
$AB$, $BC$, $CD$, $DA$, $EA$, $EB$, $EC$, and $ED$ in terms of $a_x$, $a_y$, $d_x$, and $d_y$.
For example,
$$DE=\frac{\bigl(a_y(d_x-1)+d_y-a_xd_y\bigr)\sqrt{d_x^2+d_y^2}}{a_yd_x+d_y-a_xd_y}.$$
If the lengths of the sides of a triangle are $L_1$, $L_2$, and $L_3$, and the coordinates of the
vertices are $A=(a_x,a_y)$, $B=(b_x,b_y)$, and $C=(c_x,c_y)$, then the point with barycentric coordinates
$(t_1:t_2:t_3)$ has Cartesian coordinates
$$(t_1A+t_2B+t_3C)/(t_1+t_2+t_3)$$
where $tA=(ta_x,ta_y)$ and similarly for $B$ and $C$ (the change of coordinates formula).

We can thus find the Cartesian coordinates for point $F$, the center of $\triangle ABE$ corresponding to
the given center function. We find that the coordinates of $F$ are
$$\Bigl(\frac{-a_x^2 d_y k+a_x (a_yd_x+d_yk)+a_y(d_xk+a_yd_y(1-k))}{(2k+1) (a_yd_x+d_y-a_xd_y)},$$
$$\frac{a_x^2 d_x (k-1)+a_x (-a_yd_y-d_xk+d_x)+a_y (a_yd_xk+d_y(k+1))}{(2k+1) (a_yd_x+d_y-a_xd_y)}\Bigr).$$

In the same manner, we can find the coordinates for points $G$, $H$, and $I$, the centers of
triangles $\triangle BCE$, $\triangle CDE$, and $\triangle DAE$.

Using the point slope formula, we can find the slopes of $FG$, $GH$, $HI$, and $IF$.
A little algebra shows that the slope of $FG$ is the same as the slope of $HI$
and the slope of $GH$ is the same as the slope of $IF$.
Thus $FG\parallel HI$ and $GH\parallel IJ$. Hence, the central quadrilateral $FGHI$ is a parallelogram.
\end{proof}

These computations are straightforward, but very tedious, so best left to Mathematica.
In the remainder of this paper, if a theorem is stated without a proof,
we refer the reader to the Mathematica notebooks
distributed as supplementary material to this paper.
These notebooks contain the full proofs for many of the theorems in this paper.

\textbf{Special Cases.}

Common centers of the form given in Theorem \ref{thm:general} include the centroid ($X_2$),
the circumcenter ($X_3$), the orthocenter ($X_4$), the nine-point center ($X_5$), and the de~Longchamps point ($X_{20}$). For some of these, simple geometric proofs can be found.
\index{$X_2$}
\index{$X_3$}
\index{$X_4$}
\index{$X_5$}
\index{$X_{20}$}

\begin{corollary}
\label{thm:genCentroids}
Let $ABCD$ be a convex quadrilateral with diagonal point $E$.
Let $F$, $G$, $H$, and $I$ be the centroids of triangles $\triangle ABE$, $\triangle BCE$, $\triangle CDE$,
and $\triangle DAE$, respectively.
Then $FGHI$ is a parallelogram  (Figure \ref{fig:genCentroids}).
\end{corollary}

\begin{figure}[h!t]
\centering
\includegraphics[width=0.5\linewidth]{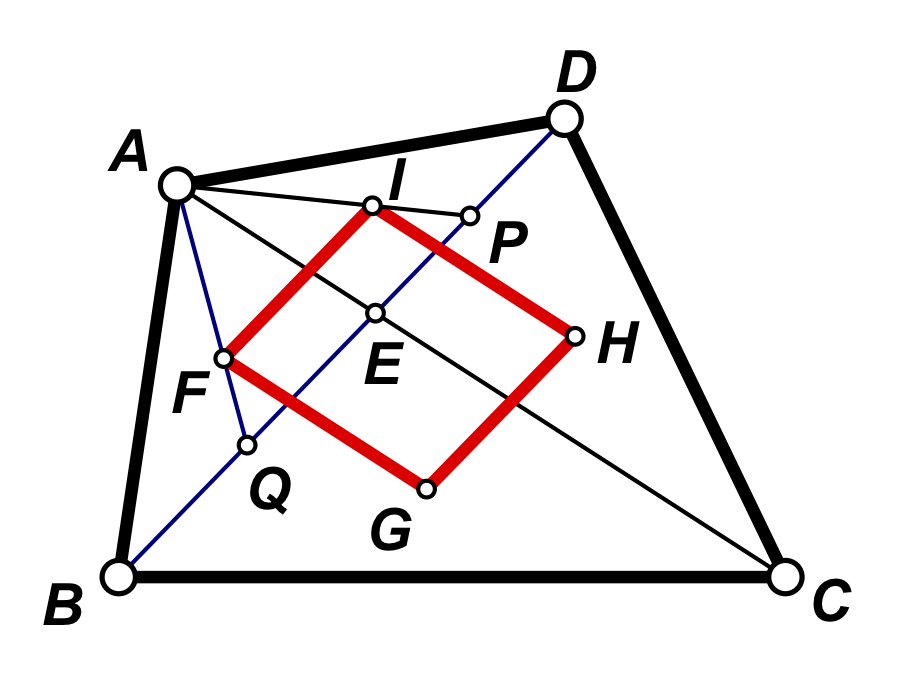}
\caption{General quadrilateral: centroids $\implies$ parallelogram}
\label{fig:genCentroids}
\end{figure}

\begin{proof}
Let $AI$ meet $BD$ at $P$ and let $AF$ meet $BD$ at $Q$.
Then $AP$ is a median of $\triangle ADE$ and $AQ$ is a median of $\triangle ABE$.
Thus $AI=\frac23AP$ and $AF=\frac23AQ$ which implies $FI\parallel BD$. Similarly, $GH\parallel BD$.
So $FI\parallel GH$. In the same way, $FG\parallel IH$. Hence, $FGHI$ is a parallelogram.
\end{proof}

\begin{corollary}
\label{thm:genCircumcenters}
Let $ABCD$ be a convex quadrilateral with diagonal point $E$.
Let $F$, $G$, $H$, and $I$ be the circumcenters of triangles $\triangle ABE$, $\triangle BCE$, $\triangle CDE$,
and $\triangle DEA$, respectively.
Then $FGHI$ is a parallelogram (Figure \ref{fig:genCircumcenters}).
\end{corollary}

\begin{figure}[h!t]
\centering
\includegraphics[width=0.4\linewidth]{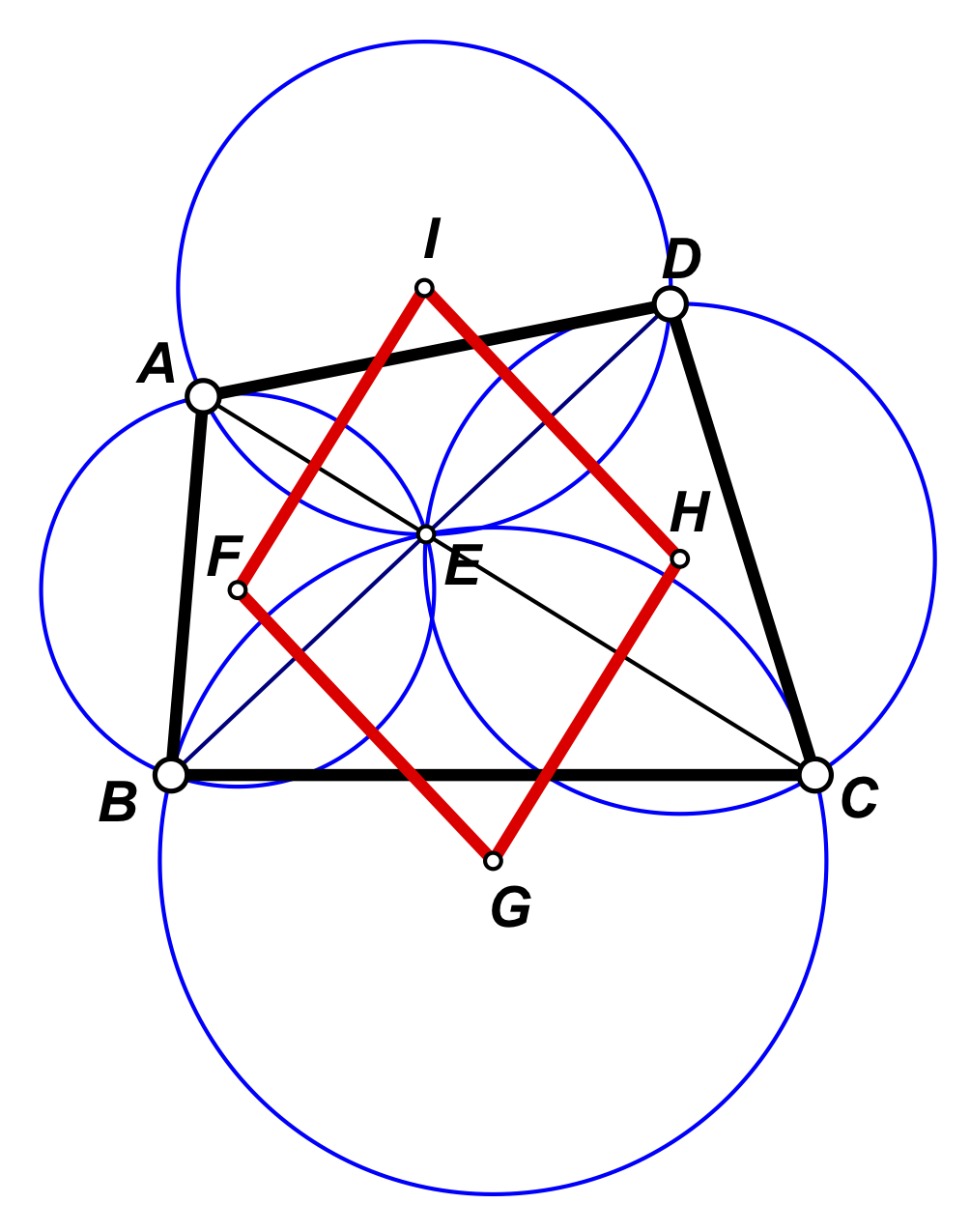}
\caption{General quadrilateral: circumcenters $\implies$ parallelogram}
\label{fig:genCircumcenters}
\end{figure}

\begin{proof}
The line of centers of two intersecting circles is perpendicular to the common chord.
Thus $FG\perp BE$ and $HI\perp ED$. Thus $FG\parallel HI$.
In the same manner, $GH\parallel FI$.
Hence, $FGHI$ is a parallelogram.
\end{proof}

\begin{corollary}
Let $ABCD$ be a convex quadrilateral with diagonal point $E$.
Let $F$, $G$, $H$, and $I$ be the orthocenters of triangles $\triangle ABE$, $\triangle BCE$, $\triangle CDE$,
and $\triangle DAE$, respectively.
Then $FGHI$ is a parallelogram (Figure \ref{fig:genOrthocenters}).
\end{corollary}

\begin{figure}[h!t]
\centering
\includegraphics[width=0.4\linewidth]{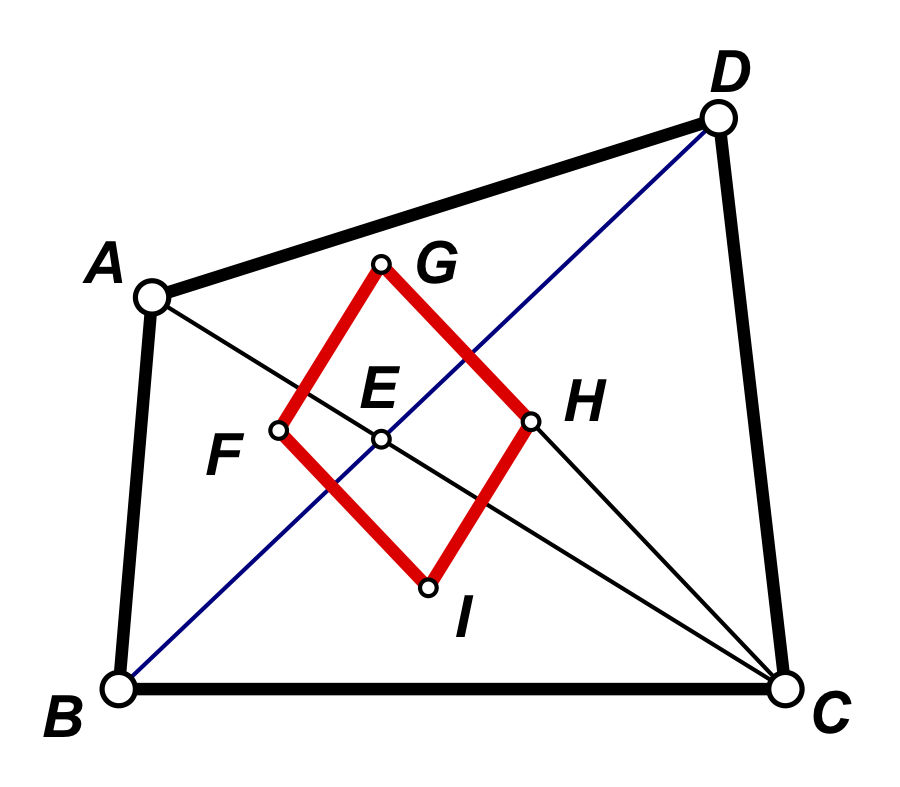}
\caption{General quadrilateral: orthocenters $\implies$ parallelogram}
\label{fig:genOrthocenters}
\end{figure}

\begin{proof}
Since $CH$ is an altitude of $\triangle CED$, we must have $CH\perp ED$.
Since $CG$ is an altitude of $\triangle BEC$, we must have $CH\perp BE$.
Thus, $GH\perp BD$. Similarly, $FI\perp BD$. Hence $GH\parallel FI$.
In the same manner, $FG\parallel IH$.
Therefore, $FGHI$ is a parallelogram.
\end{proof}

\begin{theorem}
\label{thm:EulerLine}
All the points with center function of the form 
$$\cos B\cos C+k\cos A$$
lie on the Euler line.
\end{theorem}

\begin{proof}
The \emph{Euler line} of a triangle is the line joining the centroid and the circumcenter.
According to \cite{ETC}, a center function for the centroid $G$ is $\cos B\cos C+\cos A$
and a center function for the circumcenter $O$ is $\cos A$.

From \cite[p.~28]{KimberlingB},
the equation for the line joining $(x_1:y_1:z_1)$ and $(x_2:y_2:z_2)$ is
\begin{equation}
\left|
\begin{array}{ccc}
x&y&z\\
x_1&y_1&z_1\\
x_2&y_2&z_2
\end{array}
\right|=0.
\end{equation}

If $P$ is the point with center function $\cos B\cos C+k\cos A$, then $P$ lies on line $OG$
because it is easy to verify that
$$
\left|
\begin{array}{ccc}
\cos B\cos C+k\cos A&\cos C\cos A+k\cos B&\cos A\cos B+k\cos C\\
\cos B\cos C+\cos A&\cos C\cos A+\cos B&\cos A\cos B+\cos C\\
\cos A&\cos B&\cos C
\end{array}
\right|=0
$$
because row 1 of the determinant is a linear combination of rows 2 and 3.
\end{proof}

Using the same reasoning, we get the following result.

\begin{theorem}
\label{thm:lines}
The center with center function $p$ lies on the line joining the centers
with center functions $q$ and $r$ if $p$ is a linear combination of $q$ and $r$.
\end{theorem}

\begin{theorem}
\label{thm:general2}
If the reference quadrilateral is a general quadrilateral, the central quadrilateral is a parallelogram whenever the chosen center has center function of the form
$$\cos(B-C)+k\cos A$$
for some constant $k$.
\end{theorem}

\begin{proof}
From the trigonometric identity
$$2\cos B\cos C=\cos(B-C)+\cos(B+C),$$
we see that
$$
\begin{aligned}
\cos(B-C)+k\cos A&=2\cos B\cos C-\cos(B+C)+k\cos A\\
&=2\cos B\cos C-\cos(180\degrees-A)+k\cos A\\
&=2\cos B\cos C+\cos A+k\cos A\\
&=2(\cos B\cos C+k'\cos A)\\
\end{aligned}
$$
where $k'=(k+1)/2$. The result then follows from Theorem \ref{thm:general}.
\end{proof}

A function $f(a,b,c)$ is a \emph{cyclic function} in $a$, $b$, and $c$ if
$$f(a,b,c)=f(b,c,a)=f(c,a,b)$$
for all $a$, $b$, and $c$.\label{cyclicSymmetric}
Two center functions are said to be \emph{equivalent} if their ratio is a cyclic function in $a$, $b$, and $c$.
In particular, two center functions are equivalent if their ratio is a constant.
The triangle centers corresponding to equivalent center functions coincide.

The proof of Theorem \ref{thm:general2} shows that the center functions
$\cos B\cos C+\frac{k+1}{2}\cos A$ and $\cos(B-C)+k\cos A$ are equivalent because
$$\frac{\cos B\cos C+\frac{k+1}{2}\cos A}{\cos(B-C)+k\cos A}=\frac12$$
is identically true when $A$, $B$, and $C$ are the angles of a triangle.

\begin{conjecture}
\label{conjecture:general}
If the reference quadrilateral is a general quadrilateral, the central quadrilateral is a parallelogram if and only if the center function is equivalent to $\cos B\cos C+k\cos A$ for some constant $k$.
\end{conjecture}

\begin{theorem}
\label{thm:general1}
If the reference quadrilateral is a general quadrilateral, the central quadrilateral is orthodiagonal
when the center is $X_1$. See Figure \ref{fig:diagPtIncenters}.

\end{theorem}

\begin{figure}[h!t]
\centering
\includegraphics[width=0.5\linewidth]{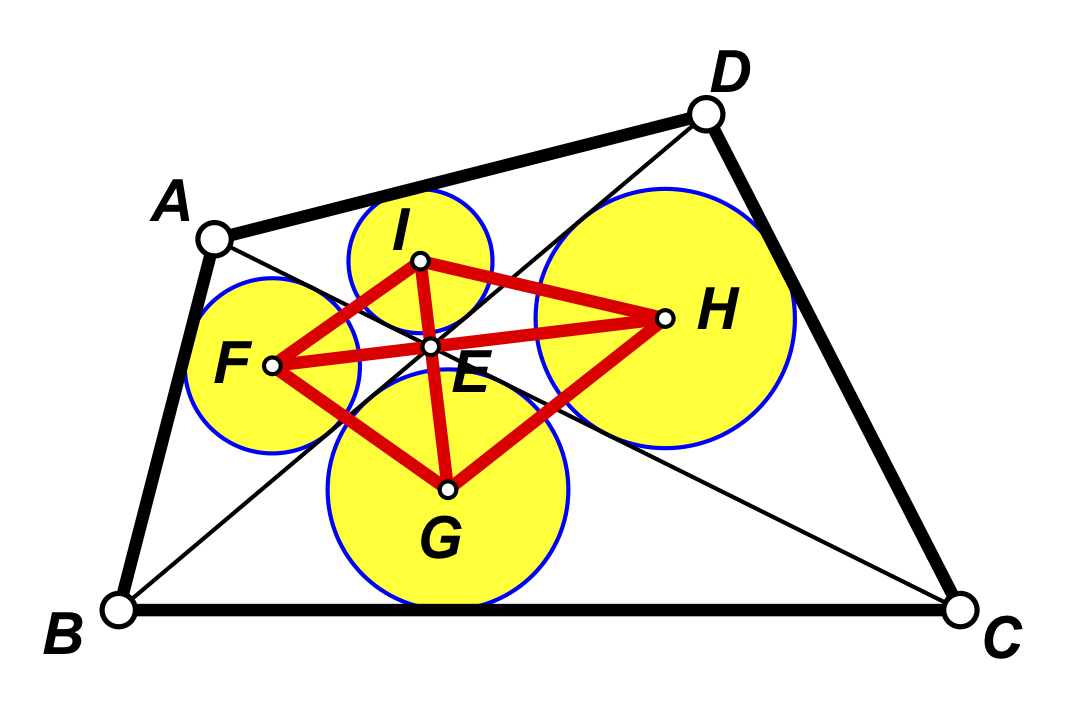}
\caption{Central Quadrilateral is orthodiagonal}
\label{fig:diagPtIncenters}
\end{figure}

\begin{proof}
Since $F$ is the incenter of $\triangle AEB$, this means that $EF$ bisects $\angle AEB$.
Similarly, $EH$ bisects $\angle DEC$. Thus, $FEH$ is a straight line.
Similarly, $IEG$ is a straight line. But $\angle FEG=\angle FEB+\angle BEG=\frac12\angle AEB
+\frac12\angle BEC=\frac12(180\degrees)=90\degrees$ and so $FH\perp GI$.
\end{proof}

\begin{conjecture}
\label{conjecture:general1}
If the reference quadrilateral is a general quadrilateral, the central quadrilateral is orthodiagonal
if and only if the center function is 1, that is, the center is $X_1$.
\end{conjecture}

\subsection{Cyclic Quadrilaterals}

\begin{theorem}
\label{thm:cyclic110}
If the reference quadrilateral is cyclic,
then the central quadrilateral is an isosceles trapezoid if the center is $X_{110}$
(Figure \ref{fig:dpCyclicX110}).
\end{theorem}

\begin{figure}[h!t]
\centering
\includegraphics[width=0.4\linewidth]{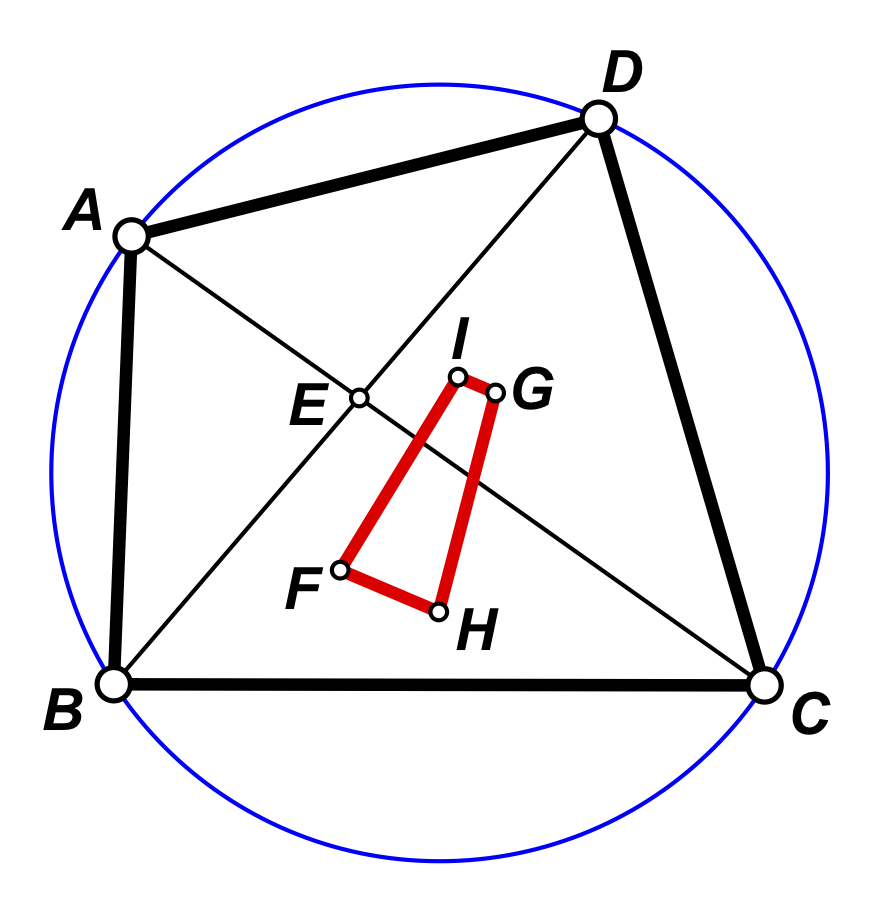}
\caption{$X_{110}$ points $\implies$ isosceles trapezoid}
\label{fig:dpCyclicX110}
\end{figure}

Note that in this case, the central quadrilateral is $FHGI$ as opposed to $FGHI$.

\begin{theorem}
\label{thm:cyclic485}
If the reference quadrilateral is cyclic,
then $FH\perp GI$ when the center is one of the Vecten points,
$X_{485}$ or $X_{486}$
(Figures \ref{fig:dpCyclicX485} and \ref{fig:dpCyclicX486}).
\end{theorem}

\textbf{Note.} Figure \ref{fig:dpCyclicX485} suggests that $FH$ and $GI$ pass through $E$.
This is not the case.

\begin{figure}[h!t]
\centering
\includegraphics[width=0.4\linewidth]{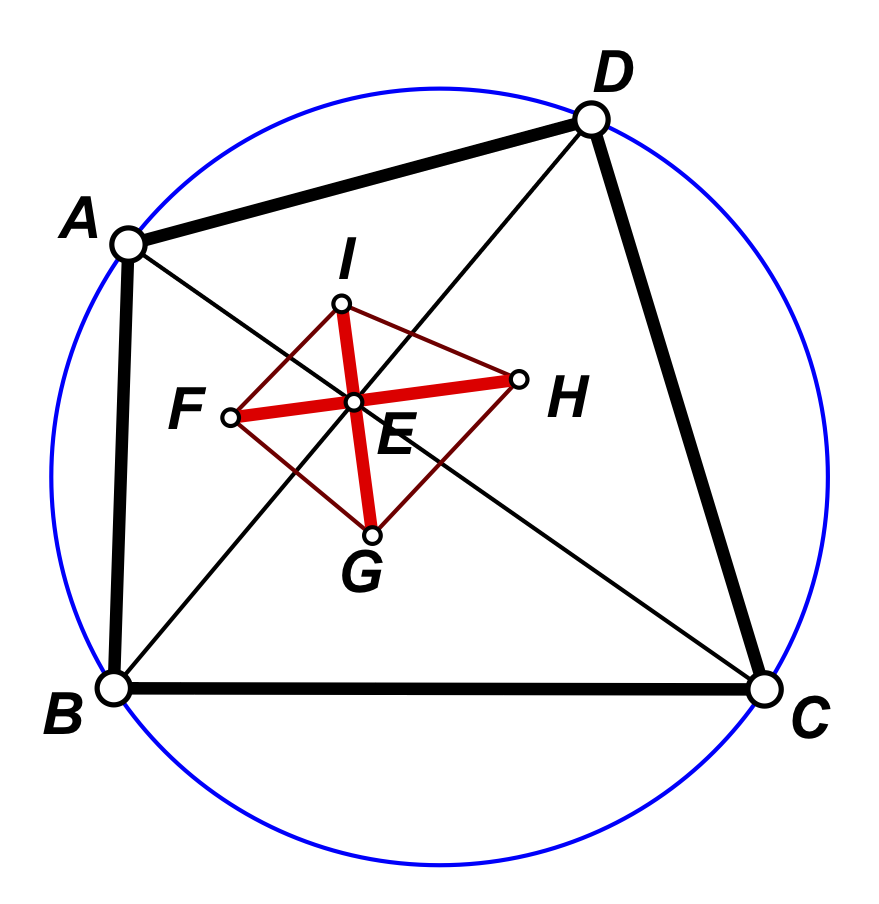}
\caption{$X_{485}$ points $\implies$ $FH\perp GI$}
\label{fig:dpCyclicX485}
\end{figure}

\begin{figure}[h!t]
\centering
\includegraphics[width=0.4\linewidth]{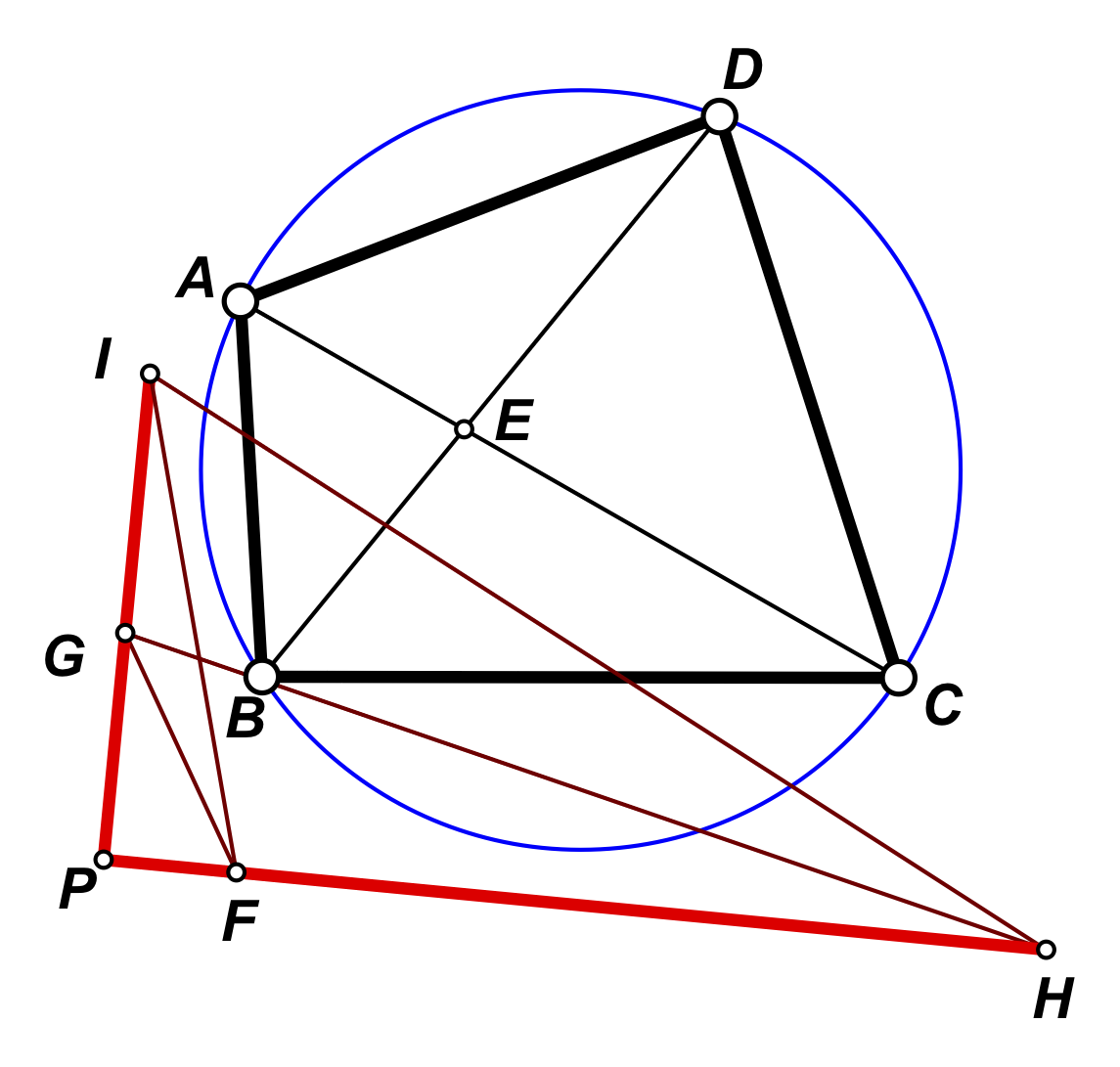}
\caption{$X_{486}$ points $\implies$ $FH\perp GI$}
\label{fig:dpCyclicX486}
\end{figure}

\begin{proof}
When dealing with cyclic quadrilaterals, it is more convenient to use barycentric
coordinates (rather than Cartesian coordinates) because the equation for the circumcircle
of a triangle in barycentric coordinates \cite[p.~63]{Yiu2012} is very simple, namely
\begin{equation}
\label{eq:circle}
a^2yz+b^2xz+c^2xy=0.
\end{equation}
We set up a barycentric coordinate system using $\triangle ABC$ as the reference triangle, so that $A=(1:0:0)$, $B=(0:1:0)$ and $C=(0:0:1)$.
Let $D$ have coordinates $(u:v:w)$. To make $ABCD$ convex, we will assume $u>0$, $v<0$, and $w>0$.
Line $AC$ has equation $y=0$ and line $BD$ has equation $wx=uz$.
Using the formula for the intersection of two lines, we can find the coordinates for point $E$,
the intersection of the diagonals. See \cite{Grozdev} for formulas involving barycentric coordinates.
We get $$E=(u:0:w).$$ 
The coordinate setup is shown in Figure \ref{fig:baryCoordinates}.

\begin{figure}[h!t]
\centering
\includegraphics[width=0.45\linewidth]{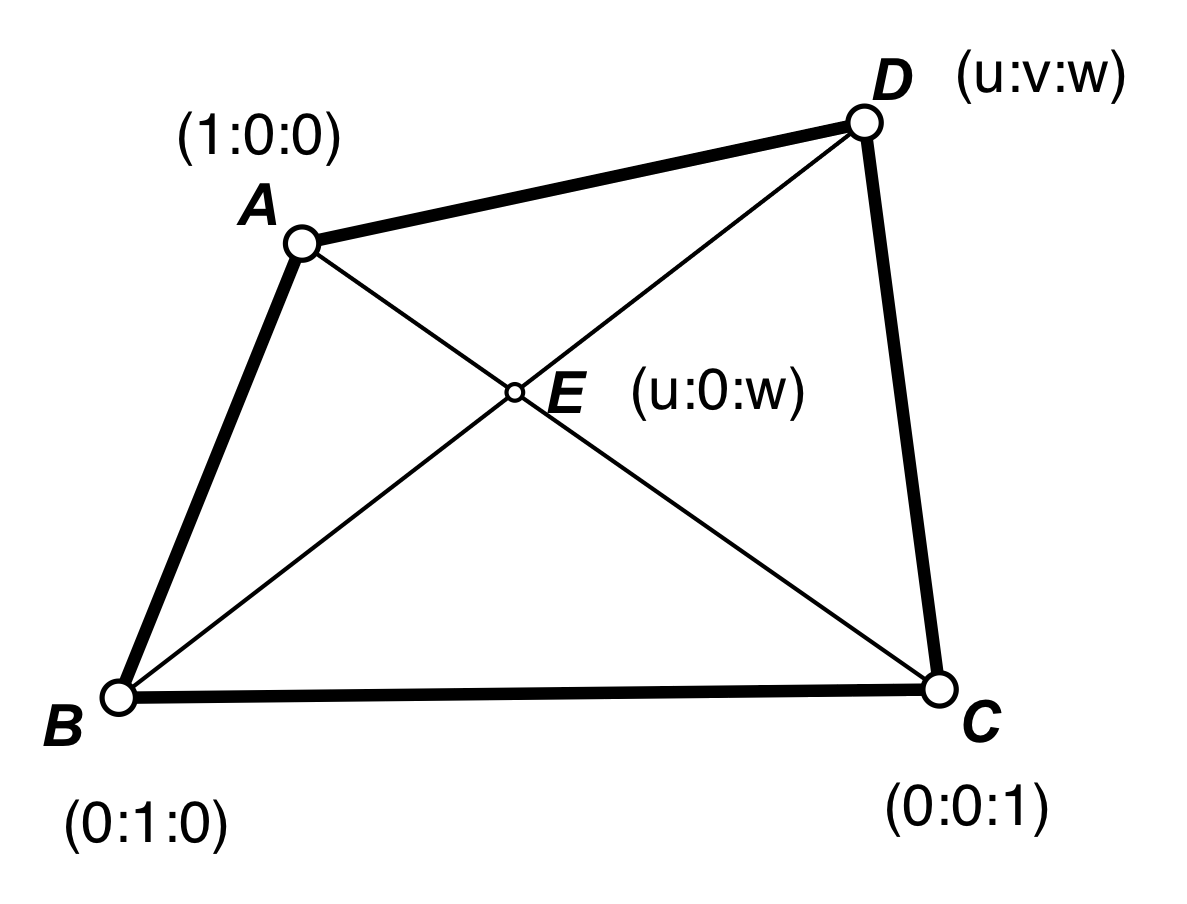}
\caption{Barycentric coordinates for a general quadrilateral}
\label{fig:baryCoordinates}
\end{figure}

When $D$ lies on the circumcircle of $\triangle ABC$, parametric equations for $D$ are
$$\begin{aligned}
u&=c^2+ta^2\\
v&=-b^2t\\
w&=t(c^2+ta^2)
\end{aligned}
$$
where $0<t<1$. This can be verified by substituting these values into Equation~
\ref{eq:circle}.

This gives us the coordinate setup as shown in Figure \ref{fig:baryCyclicCoordinates}.

\begin{figure}[h!t]
\centering
\includegraphics[width=0.5\linewidth]{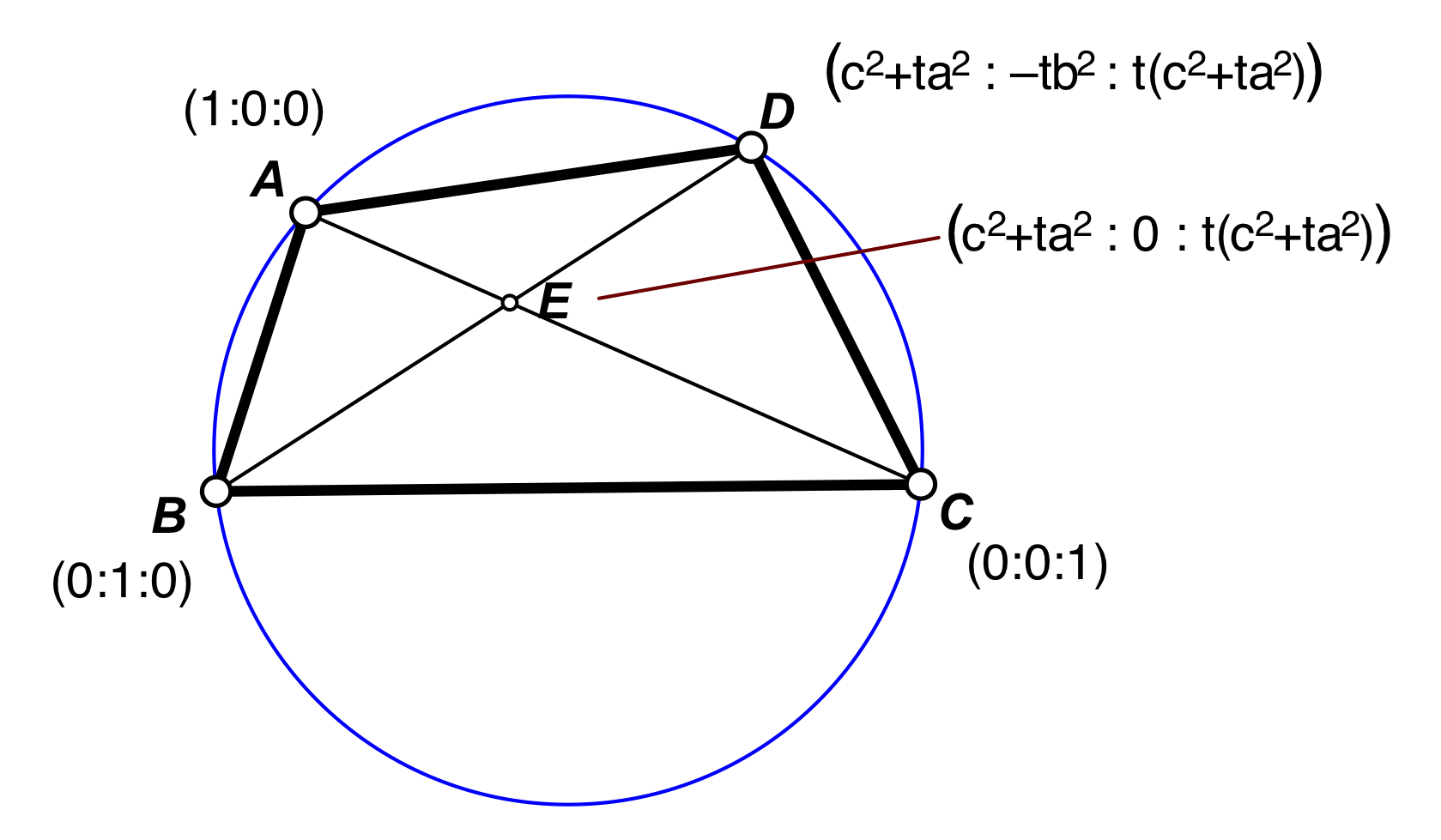}
\caption{Barycentric coordinates for a cyclic quadrilateral}
\label{fig:baryCyclicCoordinates}
\end{figure}

Now we find the coordinates for $F$, $G$, $H$, and $I$.
Instead of using the change of coordinates formula, we use the geometric
definition of the outer Vecten point, since this only requires us to do some
rotations through $90\degrees$ and some intersections of lines.
This avoids introducing square roots into our formulas.
We get
$$F=\Bigl(-c^2 \left(a^2 t (t+3)+b^2 t (2 t+3)+4 S (t+1)\right)+$$
$$t \left(a^4-a^2 \left(b^2(t+2)
+2 S (t+1)\right)+b^4 (t+1)+2 b^2 S\right)+c^4 t (t+2):$$
$$t \left(a^4 (t+1)-a^2
   \left(b^2 (t+2)+2 c^2 (t+1)\right)+b^4-b^2 t \left(c^2+2 S\right)+c^4 (t+1)\right):$$
$$t
   \left(a^4 t+a^2 \left(-2 b^2 t-c^2 (t-1)\right)+b^4 t-b^2 c^2 (t+1)-c^2 \left(c^2+2 S
   (t+1)\right)\right)\Bigr)$$
and similar expressions for $G$, $H$, and $I$, where $S$ is twice the area of $\triangle ABC$.
Finally, we use the formula that gives the condition for two lines to be perpendicular,
and we find that $FH\perp GI$.
These calculations are best done using Mathematica.
A similar calculation provides a proof for the inner Vecten point.
\end{proof}

\subsection{Tangential Quadrilaterals}

\begin{theorem}
If the reference quadrilateral is a tangential quadrilateral,
then the central quadrilateral is cyclic if the chosen center has
center function 1, i.e. the center is $X_1$.
(Figure \ref{fig:tangential-incircles}).
\end{theorem}

\begin{figure}[h!t]
\centering
\includegraphics[width=0.3\linewidth]{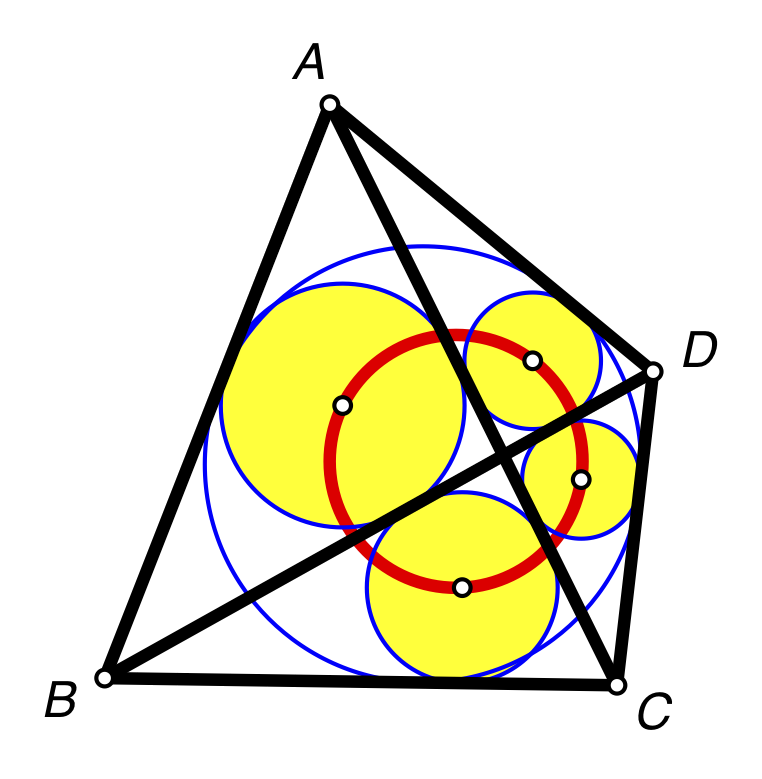}
\caption{Tangential quadrilateral: incenters $\implies$ cyclic}
\label{fig:tangential-incircles}
\end{figure}

\begin{proof}
This is Result 2 in the Introduction.
For a geometric proof, see Theorem 20 of \cite{Grinberg}. 
\end{proof}

\begin{conjecture}
If the reference quadrilateral is a tangential quadrilateral,
then the central quadrilateral is cyclic if and only if the center is $X_1$.
\end{conjecture}

\subsection{Equidiagonal Quadrilaterals}

\begin{theorem}
\label{thm:equi-rhombus}
If the reference quadrilateral is equidiagonal, the central quadrilateral is a rhombus if the chosen center has center function of the form
$$\cos B\cos C+k\cos A$$
for some constant $k$.
\end{theorem}

Theorem \ref{thm:equi-rhombus} is illustrated in Figure \ref{fig:equi-orthocenters} when the chosen center is the orthocenter.

\begin{figure}[h!t]
\centering
\includegraphics[width=0.3\linewidth]{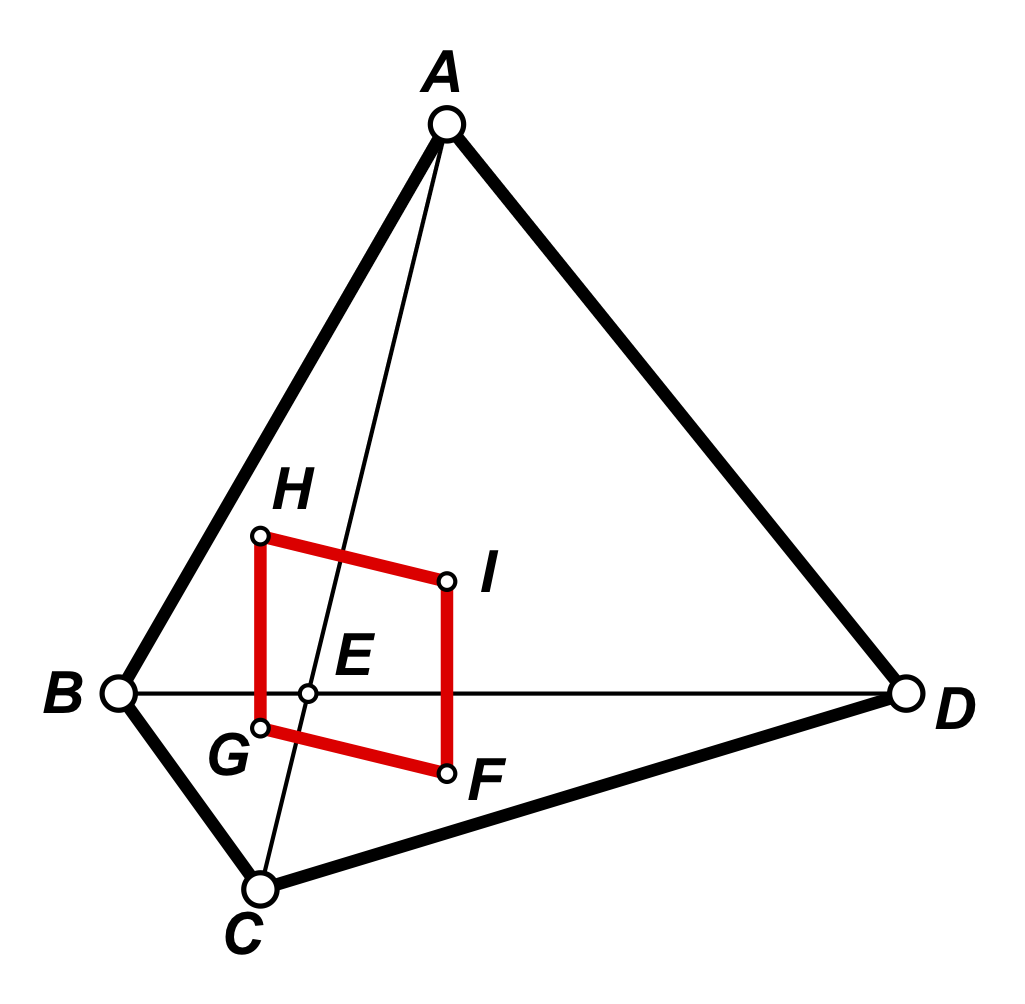}
\caption{$AC=BD$, orthocenters $\implies$ rhombus}
\label{fig:equi-orthocenters}
\end{figure}

\begin{corollary}
\label{cor:equi-rhombusH}
Let $ABCD$ be an equidiagonal quadrilateral with diagonal point $E$.
Let $F$, $G$, $H$, and $I$ be the orthocenters of triangles $\triangle ABE$, $\triangle BCE$, $\triangle CDE$,
and $\triangle DAE$, respectively.
Then $FGHI$ is a rhombus (Figure \ref{fig:equi-orthocenters}).
\end{corollary}

\begin{corollary}
\label{cor:equi-rhombusM}
Let $ABCD$ be an equidiagonal quadrilateral with diagonal point $E$.
Let $F$, $G$, $H$, and $I$ be the centroids of triangles $\triangle ABE$, $\triangle BCE$, $\triangle CDE$,
and $\triangle DAE$, respectively.
Then $FGHI$ is a rhombus.
\end{corollary}

For this corollary, a simple geometric proof can be given.

\begin{proof}
From the proof of Corollary \ref{thm:genCentroids} (Figure \ref{fig:genCentroids}), it is
easily seen that $IP=\frac13 AP$ and $FQ=\frac13 AQ$, so that $FI=\frac13BD$.
Similarly, $FG=\frac13AC$. But since $BD=AC$, this implies $FI=FG$
and parallelogram $FGHI$ becomes a rhombus.
\end{proof}

\begin{corollary}
\label{cor:equi-rhombusO}
Let $ABCD$ be an equidiagonal quadrilateral with diagonal point $E$.
Let $F$, $G$, $H$, and $I$ be the circumcenters of triangles $\triangle ABE$, $\triangle BCE$, $\triangle CDE$,
and $\triangle DAE$, respectively.
Then $FGHI$ is a rhombus (Figure \ref{fig:equi-circumcenters}).
\end{corollary}

\begin{figure}[h!t]
\centering
\includegraphics[width=0.3\linewidth]{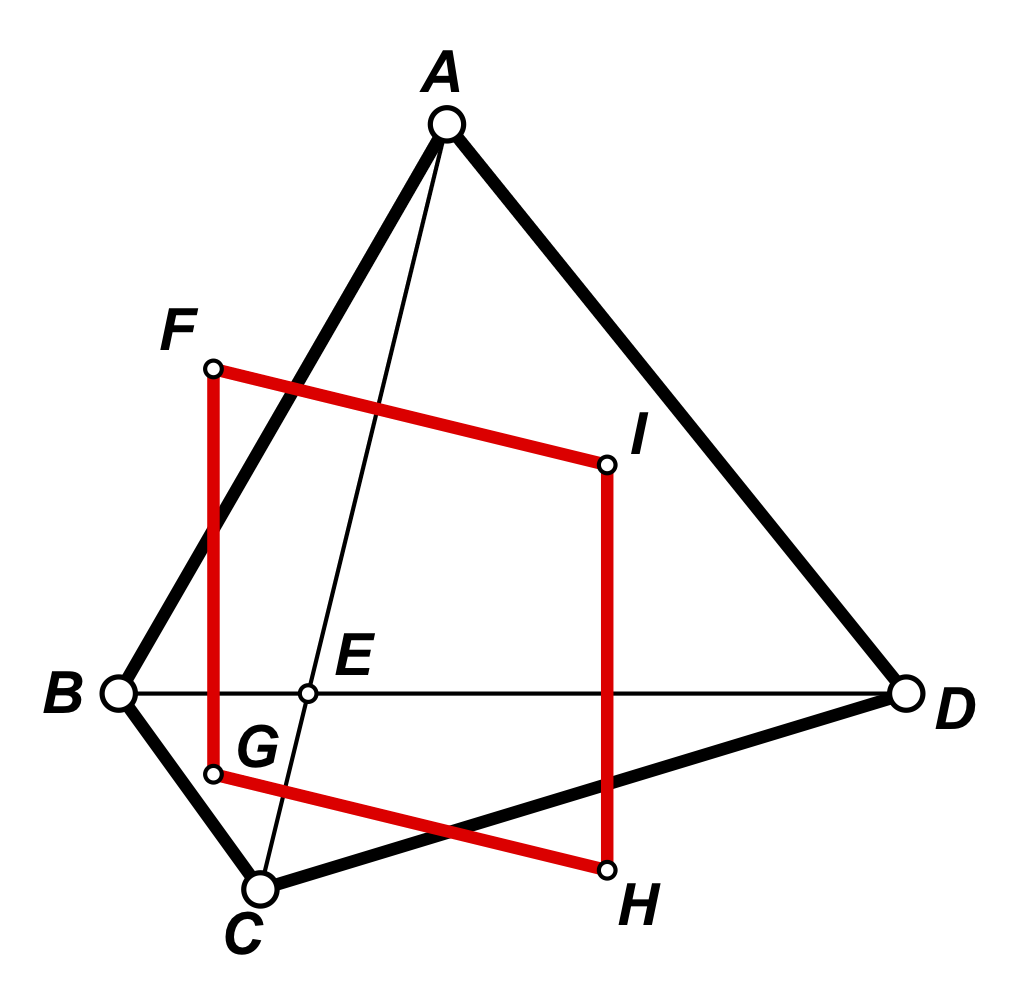}
\caption{$AC=BD$, circumcenters $\implies$ rhombus}
\label{fig:equi-circumcenters}
\end{figure}

\void{
\begin{theorem}
\label{thm:equi}
If the reference quadrilateral is equidiagonal, the central quadrilateral is orthodiagonal whenever the chosen center lies on the Nagel line.
\end{theorem}

The \emph{Nagel line} of a triangle is the line joining the incenter, the centroid,
the Nagel point, and the Spieker center of a triangle.
From \cite{ETC} we find that a center function for the Nagel point
is $P=(b+c)/a-1$ and a center function for the Spieker center is $Q=(b+c)/a$.
Therefore, any point on the Nagel line has first trilinear coordinate of the form
$$P+t(Q-P)=\frac{b+c}{a}-1+t.$$

Letting $t=k+1$ shows that Theorem \ref{thm:equi} is equivalent to the
following theorem.
}

\begin{theorem}
\label{thm:equi}
If the reference quadrilateral is equidiagonal, the central quadrilateral is orthodiagonal whenever the chosen center has center function of the form
$$\frac{b+c}{a}+k.$$
Also, the diagonals of the central quadrilateral are
parallel to the bimedians of the reference quadrilateral.
\end{theorem}

\textbf{Note.} A \emph{bimedian} of a quadrilateral is the line joining the
midpoints of two opposite sides.

\begin{proof}
When dealing with equidiagonal quadrilaterals, it is convenient to set up a
coordinate system where the diagonal point of the quadrilateral is at the origin
and one of the diagonals ($AC$) of the quadrilateral lies along the $x$-axis.
We can scale the figure so that $C$ lies at $(1,0)$.
Figure \ref{fig:equicoordinates} shows the coordinates we use, where
$a_x<0$, $b_x\leq 0$, $b_y<0$, $d_x\geq 0$, and $d_y>0$.

\begin{figure}[h!t]
\centering
\includegraphics[width=0.6\linewidth]{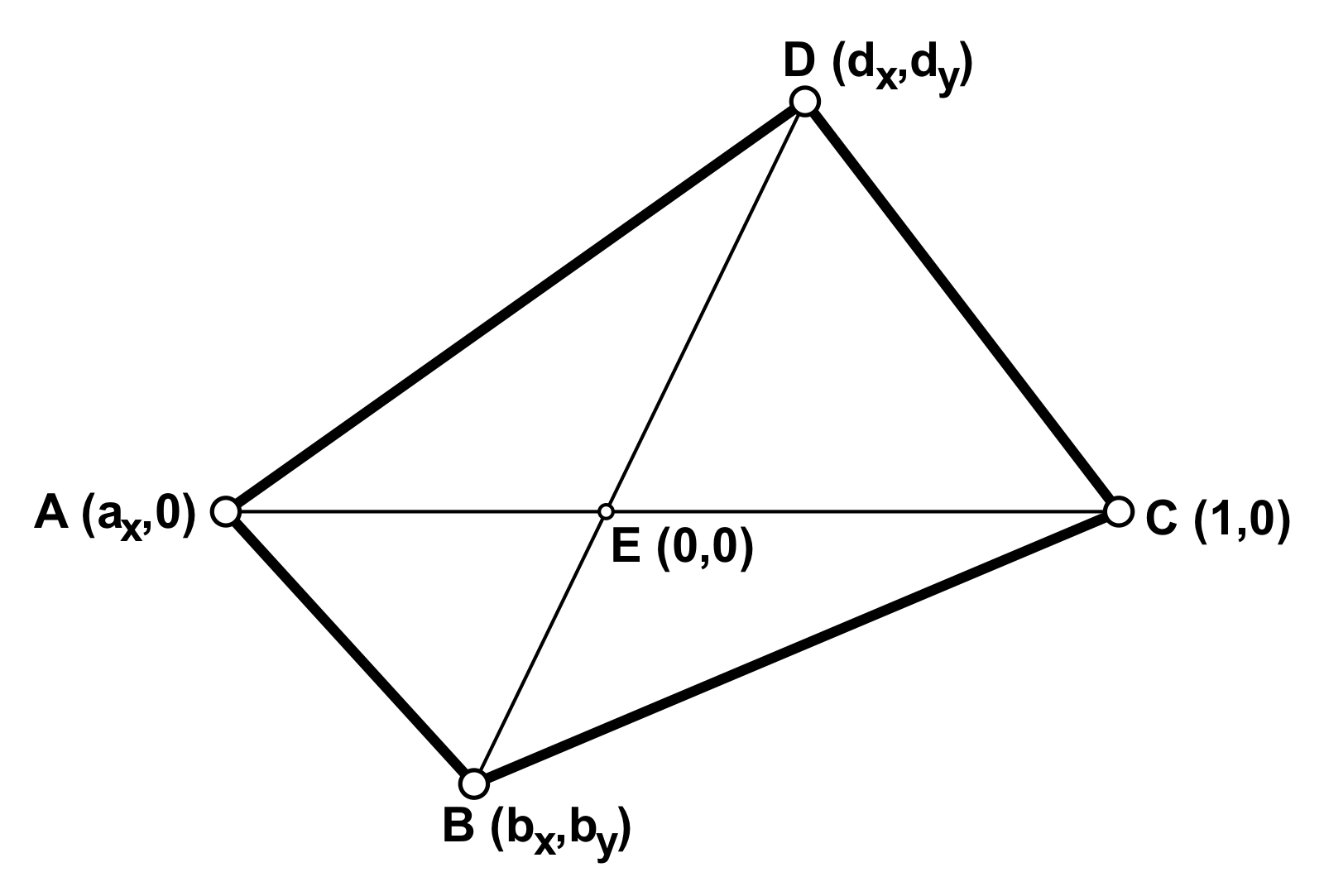}
\caption{Cartesian Coordinates for an equidiagonal quadrilateral}
\label{fig:equicoordinates}
\end{figure}

In order for $E$ to be at the origin, we need the slopes of $BE$ and $ED$
to be the same. We can ensure that this is the case by setting
$$d_y=\frac{d_x\cdot b_y}{b_x}.$$

In order for the quadrilateral to be equidiagonal, we need $AC=BD$
or equivalently $(AC)^2=(BD)^2$. This gives the equation
$$(1-a_x)^2=\frac{(b_x^2+b_y^2)(d_x-b_x)^2}{b_x^2}.$$
Solving for $a_x$ and taking the solution that makes $a_x$ negative,
we find that we can ensure that the quadrilateral is equidiagonal by setting
$$a_x=1+\frac{(d_x-b_x)\sqrt{b_x^2+b_y^2}}{b_x}.$$

Using $(b+c)/a+k$ as the center function,
we can compute the coordinates for $F$, $G$, $H$, and $I$.
%These turn out to be complicated expressions.
Using the formula for the slope of a line through two points,
we find the slopes of $FH$ and $GI$.
Then (using Mathematica) we find and simplify the product of these slopes.
The result is $-1$, proving that $FGHI$ is orthodiagonal.
\end{proof}

\void{
\begin{theorem}
\label{thm:equi}
If the reference quadrilateral is equidiagonal, the central quadrilateral is orthodiagonal whenever the chosen center has center function equivalent to one of the following forms, for some constants $k$, $p$, and $q$.
$$\frac{b+c}{a}+k$$
$$\frac{b+c+k}{a}+1$$
$$\cos B+\cos C-1-k\cos A$$
$$\cos B+\cos C-1+p\cos A+q\cos B\cos C$$
$$(a+b+c)\sin B\sin C+k\frac{bc}{a}\sin A$$
$$\frac{r}{R}+k\cos A$$
$$\frac{r}{R}+k\sin B\sin C$$
Furthermore, the diagonals of the central quadrilateral are
parallel to the bimedians of the reference quadrilateral.
\end{theorem}
}

Since the Spieker center has center function $(b+c)/a$, it has the form
referenced by Theorem \ref{thm:equi} (with $k=0$),
so we have the following corollary.

\begin{corollary}
\label{thm:Spieker}
Let $ABCD$ be an equidiagonal quadrilateral with diagonal point $E$.
Let $F$, $G$, $H$, and $I$ be the Spieker centers of $\triangle ABE$, 
$\triangle BCE$, $\triangle CDE$, and $\triangle DAE$, respectively.
Then quadrilateral $FGHI$ is orthodiagonal. (Figure \ref{fig:diagPtSpiekerCenters})
\end{corollary}

\begin{figure}[h!t]
\centering
\includegraphics[width=0.6\linewidth]{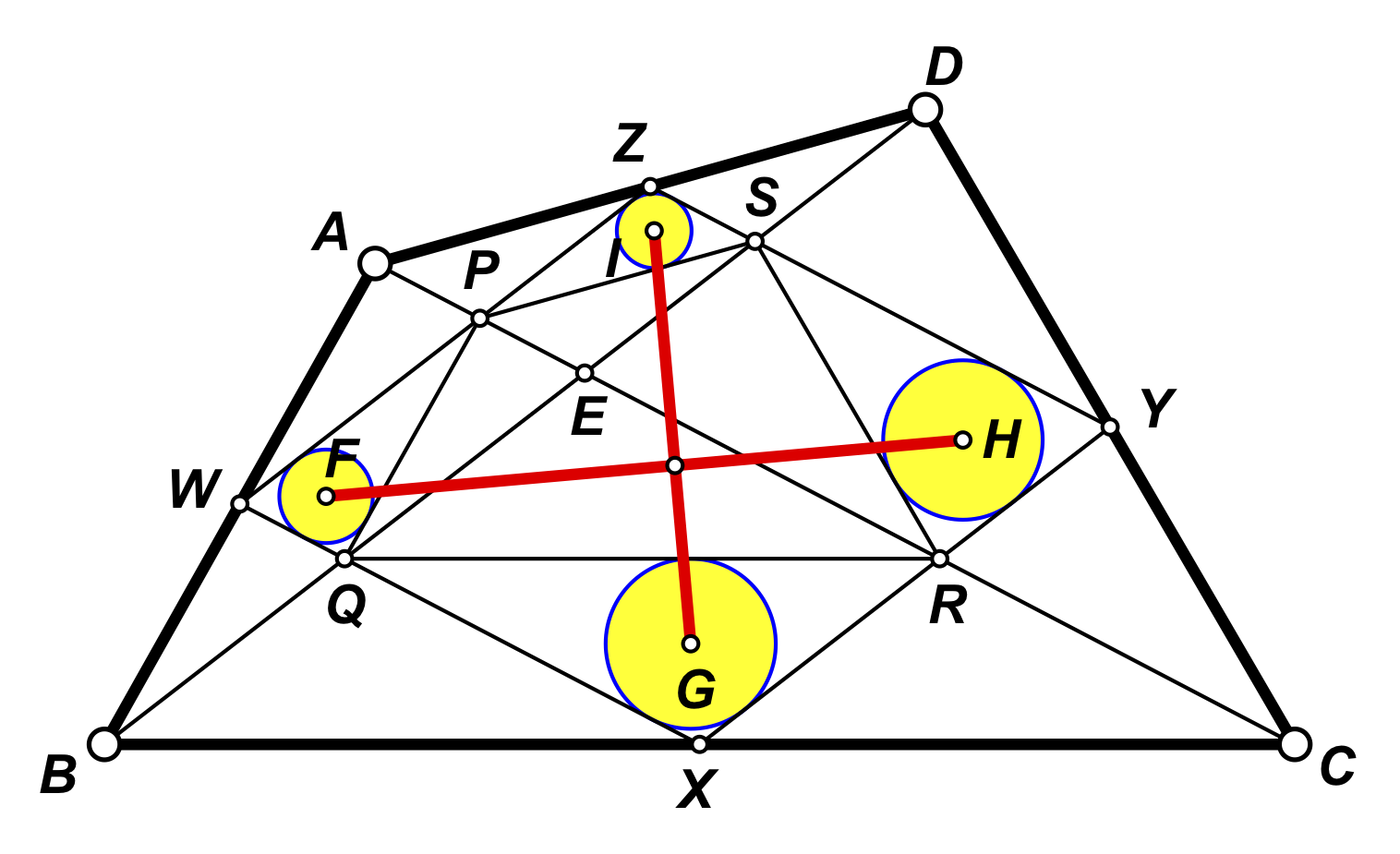}
\caption{Equidiagonal quadrilateral: Spieker centers $\implies$ orthodiagonal}
\label{fig:diagPtSpiekerCenters}
\end{figure}

In this case, we have a simple purely geometric proof.

\begin{proof}
Let the midpoints of the sides of the quadrilateral be $W$, $X$, $Y$, and $Z$ as shown.
Let the Spieker centers be $F$, $G$, $H$, and $I$. Let $P$, $Q$, $R$, and $S$ be the
midpoints of $AE$, $BE$, $CE$, and $DE$, respectively.
The Spieker center of a triangle is the incenter of its medial triangle.
Thus, $F$, $G$, $H$, and $I$ are the centers of the yellow circles in Figure \ref{fig:diagPtSpiekerCenters}.
Since $Y$, $S$, and $Z$ are midpoints, $YSZ$ is a straight line segment and $YZ=\frac12 AC$.
Similarly, $ZW=\frac12 BD$, $WX=\frac12 AC$, and $XY=\frac12 BD$.

Since quadrilateral $ABCD$ is equidiagonal, quadrilateral $WXYZ$ is a rhombus. The diagonals of a rhombus are perpendicular, so $WY\perp XZ$. The diagonals of a rhombus also bisect the angles at the vertices.
Since $H$ is the incenter of $\triangle SYR$, $YH$ is the angle bisector of $\angle SYR$.
Since $YW$ is also the angle bisector of $\angle SYR$, we can conclude that $WY$ passes through $H$ and $F$.
In the same manner, $ZX$ passes through $I$ and $G$. Since $WY\perp XZ$, we can conclude that
$GI\perp FH$ or that quadrilateral $FGHI$ is orthodiagonal.

Note that in this case, the diagonals of $FGHI$ coincide with the bimedians of $ABCD$.
\end{proof}

Since the Nagel point has center function $(b+c)/a-1$, we have the following corollary.

\begin{corollary}
\label{thm:Nagel}
Let $ABCD$ be an equidiagonal quadrilateral with diagonal point $E$.
Let $F$, $G$, $H$, and $I$ be the Nagel points of $\triangle ABE$, 
$\triangle BCE$, $\triangle CDE$, and $\triangle DAE$, respectively.
Then $FH\perp GI$. In other words, quadrilateral $FGHI$ is orthodiagonal.
Furthermore, $FH$ and $GI$ are parallel to the bimedians of quadrilateral $ABCD$.
(Figure \ref{fig:diagPtNagelPoints})
\end{corollary}

\begin{figure}[h!t]
\centering
\includegraphics[width=0.4\linewidth]{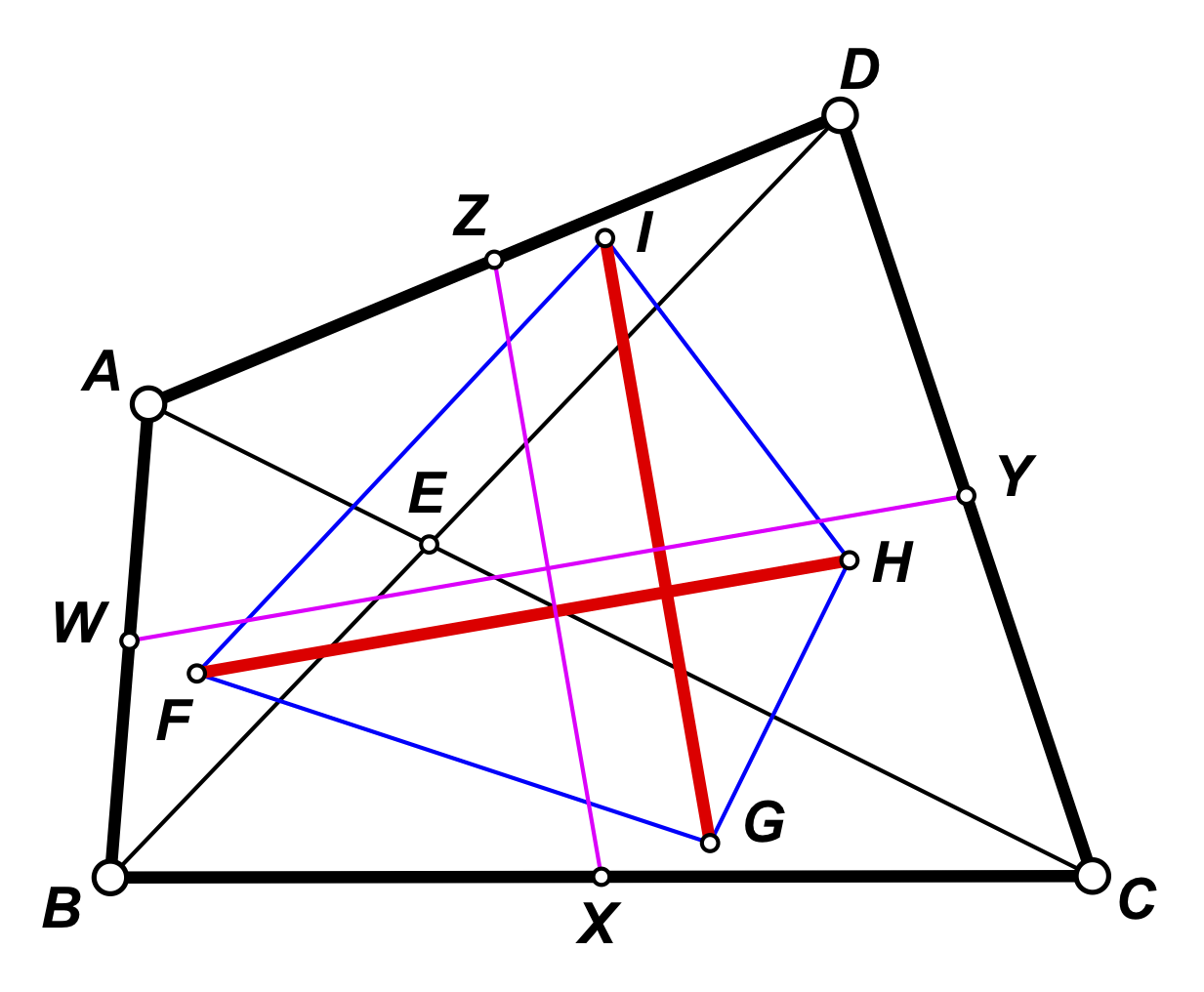}
\caption{Equidiagonal quadrilateral: Nagel points $\implies$ orthodiagonal}
\label{fig:diagPtNagelPoints}
\end{figure}

\begin{open}
Is there a simple geometric proof for the special case of Theorem \ref{thm:equi} when the
chosen center is the Nagel point? (Figure \ref{fig:diagPtNagelPoints})
Is there a geometrical proof that in Figure \ref{fig:diagPtNagelPoints}, $FH\parallel WY$?
\end{open}

\begin{theorem}
\label{thm:NagelLine}
All the points with center function of the form 
$$\frac{b+c}{a}+k$$
lie on the Nagel line.
\end{theorem}

\begin{proof}
The \emph{Nagel line} of a triangle is the line joining the incenter, the centroid,
the Nagel point, and the Spieker center of a triangle.
From \cite{ETC}, we find that a center function for the incenter
is $I=1$ and a center function for the Spieker center is $S=(b+c)/a$.
Let $P=(b+c)/a+k$.
Since $P=S+kI$, $P$ is a linear combination of $I$ and $S$ and
our result follows from Theorem~\ref{thm:lines}.
\end{proof}

\begin{theorem}
\label{thm:equiI}
Let $ABCD$ be an equidiagonal quadrilateral with diagonal point $E$.
Let $F$, $G$, $H$, and $I$ be the incenters of $\triangle ABE$, 
$\triangle BCE$, $\triangle CDE$, and $\triangle DAE$, respectively.
Then $FH$ and $GI$ are parallel to the bimedians of quadrilateral $ABCD$
(Figure \ref{fig:equi-incenters}).
\end{theorem}

\begin{figure}[h!t]
\centering
\includegraphics[width=0.45\linewidth]{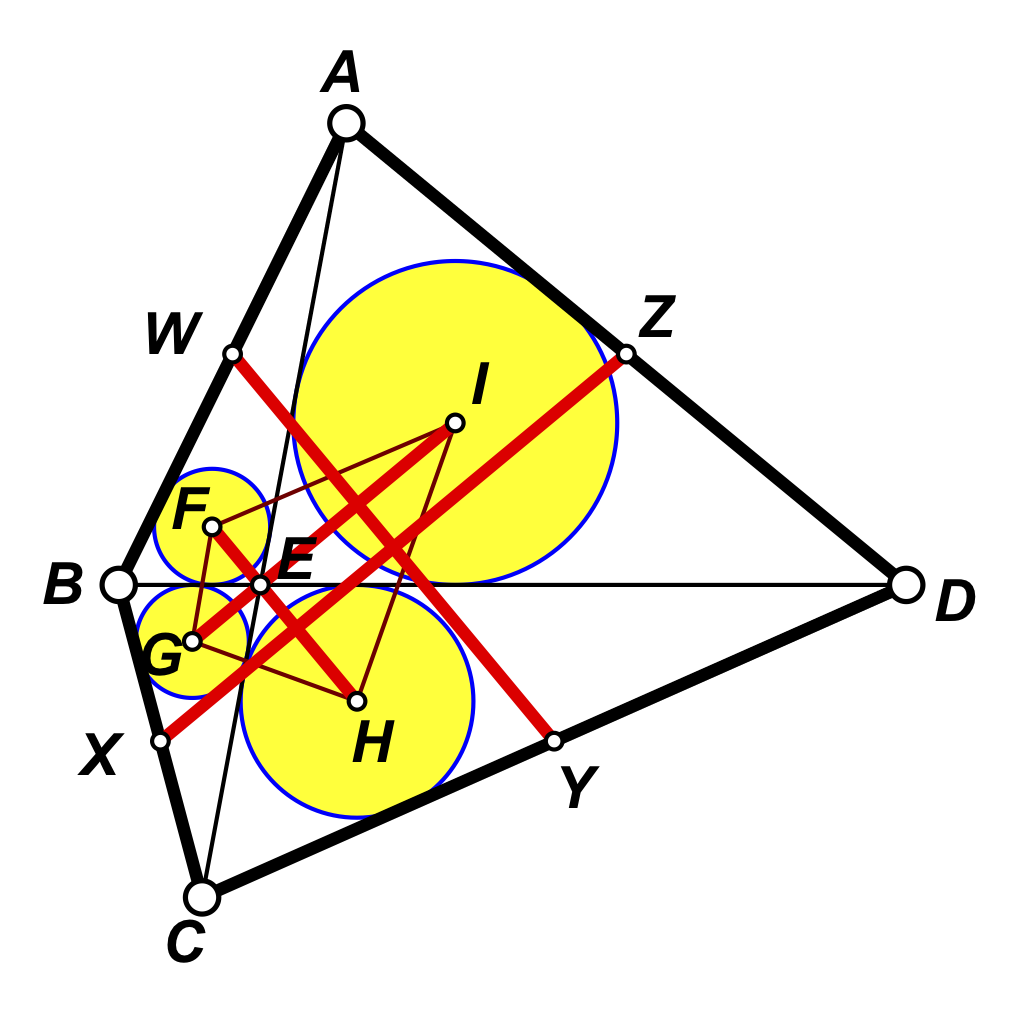}
\caption{$AC=BD$, incenters, bimedians $\implies$ $FH\parallel WY$}
\label{fig:equi-incenters}
\end{figure}

\begin{theorem}
\label{thm:equiy}
If the reference quadrilateral is equidiagonal, the central quadrilateral is orthodiagonal whenever the chosen center has center function of the form
$$\cos B+\cos C-1-k\cos A.$$
\end{theorem}

Since the Bevan point has center function $\cos B+\cos C-1-\cos A$, we have the following corollary.

\begin{corollary}
\label{cor:equi-X40}
Let $ABCD$ be an equidiagonal quadrilateral with diagonal point $E$.
Let $F$, $G$, $H$, and $I$ be the Bevan points of $\triangle ABE$, 
$\triangle BCE$, $\triangle CDE$, and $\triangle DAE$, respectively.
Then $FH\perp GI$ (Figure \ref{fig:equi-X40}).
\end{corollary}

\begin{figure}[h!t]
\centering
\includegraphics[width=0.4\linewidth]{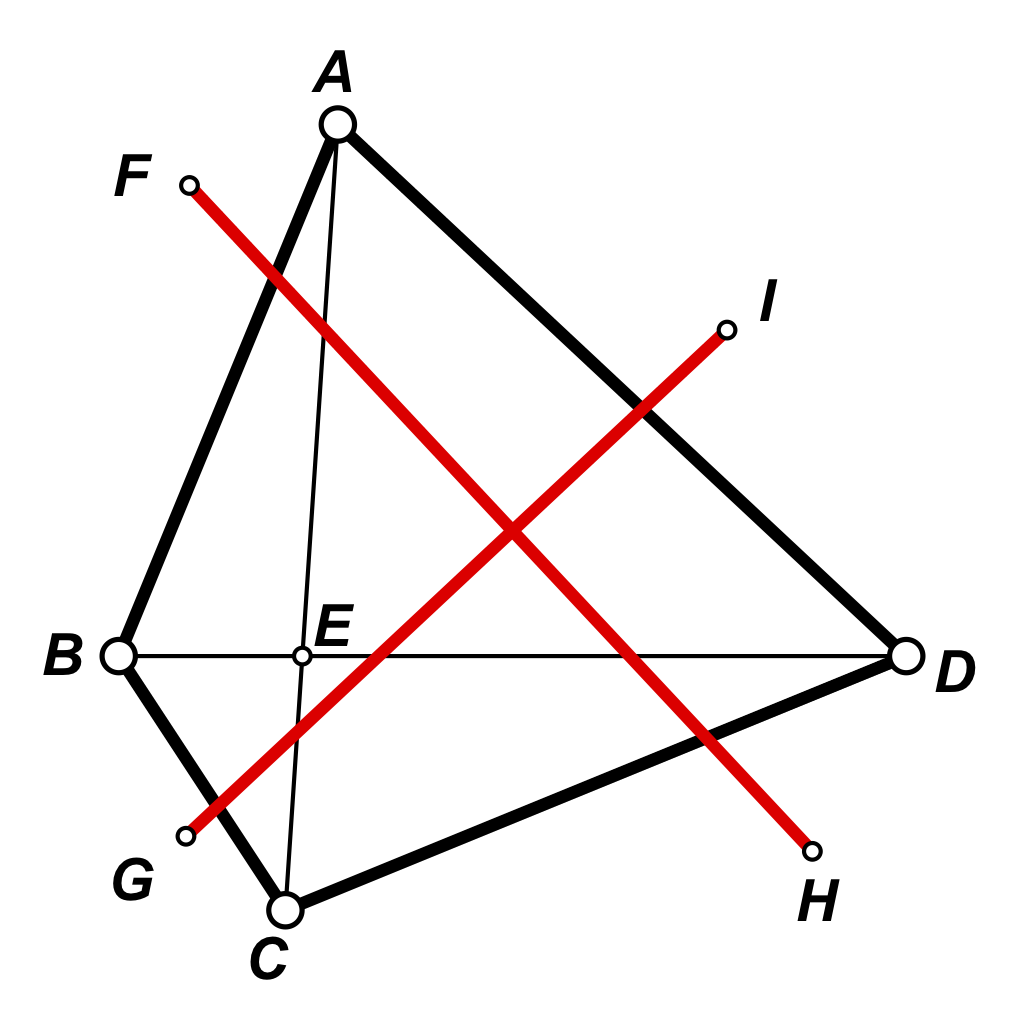}
\caption{$AC=BD$, Bevan points $\implies$ $FH\perp GI$}
\label{fig:equi-X40}
\end{figure}

\begin{theorem}
\label{thm:BevanLine}
All the points with center function of the form 
$$\cos B+\cos C-1-k\cos A$$
lie on the Bevan line.
\end{theorem}

\begin{proof}
The \emph{Bevan line} of a triangle (also known as the $OI$ line) is the line joining the incenter,
the circumcenter, and the Bevan point.
From \cite{ETC}, we find that a center function for the circumcenter
is $O=\cos A$ and a center function for the Bevan point is $B=\cos B+\cos C-1-\cos A$.
Let $P=\cos B+\cos C-1-k\cos A$.
Since $P=B-(k-1)O$, $P$ is a linear combination of $O$ and $P$ and
our result follows from Theorem~\ref{thm:lines}.
\end{proof}

\newpage
\begin{theorem}
\label{thm:equirR}
If the reference quadrilateral is equidiagonal, the central quadrilateral is orthodiagonal whenever the chosen center has center function of the form
$$\frac{r}{R}+k\cos A.$$
\end{theorem}

\begin{proof}
Using the rules
\begin{equation}
\label{rules}
\begin{aligned}
\cos A&=\frac{b^2 + c^2 - a^2}{2bc}\\
\cos B&=\frac{c^2 + a^2 - b^2}{2ca}\\
\cos C&=\frac{a^2 + b^2 - c^2}{2ab}\\
R&=\frac{abc}{4K}\\
r&=\frac{K}{s}\\
K&=\sqrt{s(s-a)(s-b)(s-c)}\\
s&=\frac{a+b+c}{2}
\end{aligned}
\end{equation}
we can remove $A$, $B$, $C$, $r$, and $R$ from the expression
$$\frac{\cos B+\cos C-1-k\cos A}{\frac{r}{R}+t\cos A}.$$
When $t=-(k+1)$, the expression simplifies to 1.
This means that $$\cos B+\cos C-1-k\cos A\quad \hbox{and}\quad\frac{r}{R}+t\cos A$$ are equivalent center functions.
Thus, this theorem follows from Theorem \ref{thm:equiy}.
\end{proof}

\begin{theorem}
\label{thm:equirR2}
If the reference quadrilateral is equidiagonal, the central quadrilateral is orthodiagonal whenever the chosen center has center function of the form
$$\frac{r}{R}+k\sin B\sin C.$$
\end{theorem}

%\newpage
\subsection{Orthodiagonal Quadrilaterals}

\begin{lemma}
\label{lemma:angleBisector}
The condition for a triangle center with center function $f(x,y,z)$ to lie
on the angle bisector at vertex $A$ in right triangle $ABC$ (with right angle at $A$) is
$$f(x,y,z)=f(y,x,z)$$
for all $x$, $y$, and $z$ satisfying $x^2+y^2=z^2$.
\end{lemma}

\begin{proof}
Let $P$ be a point on the angle bisector of $\angle A$ in right triangle $ABC$
as shown in Figure \ref{fig:angleBisectorCondition}.
Suppose a center function for $P$ is $f(x,y,z)$ so that the trilinear coordinates of $P$ are
$$P=\bigl(f(a,b,c):f(b,c,a):f(c,a,b)\bigr).$$

\begin{figure}[h!t]
\centering
\includegraphics[width=0.6\linewidth]{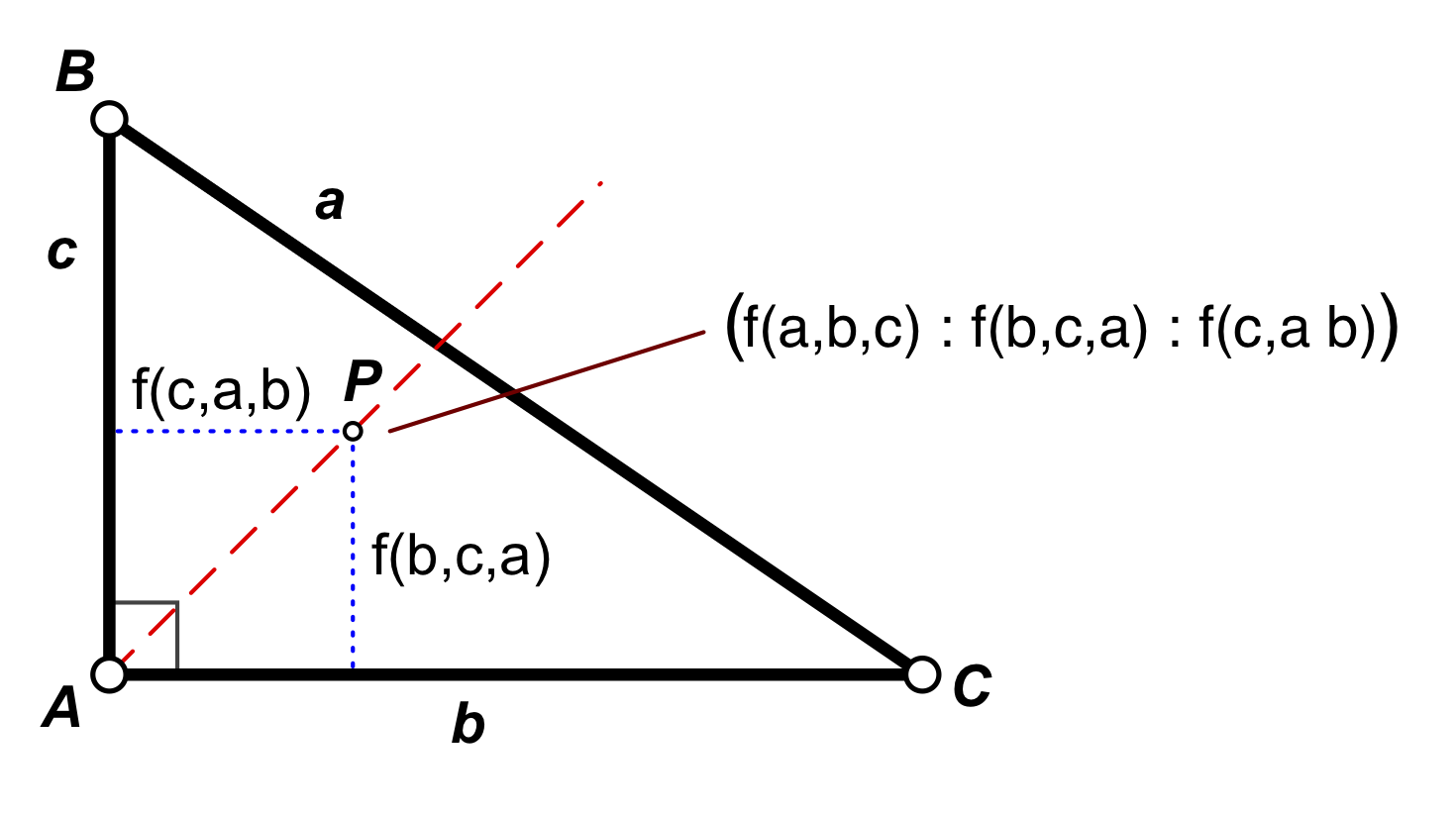}
\caption{$P$ lies on bisector of $\angle A$ in right triangle $ABC$}
\label{fig:angleBisectorCondition}
\end{figure}

The second trilinear coordinate is proportional to the distance from $P$ to $AC$
and the third trilinear coordinate is proportional to the distance from $P$ to $AB$.
Since $P$ is on the angle bisector of $\angle A$, it is equidistant from $AC$ and $AB$.
Thus, $$f(b,c,a)=f(c,a,b)$$ for all $a$, $b$, and $c$ satisfying $a^2=b^2+c^2$.
But, by definition, a center function is symmetric in its 2nd and 3rd arguments, so
$$f(c,a,b)=f(c,b,a)$$ for all $a$, $b$ and $c$. Therefore,
$$f(b,c,a)=f(c,b,a)$$
for all $a$, $b$ and $c$ satisfying $a^2=b^2+c^2$.
Switching notation from $a,b,c$ to $x,y,z$ gives that the condition is that
$f(x,y,z)=f(y,x,z)$
for all $x$, $y$, and $z$ satisfying $x^2+y^2=z^2$.
\end{proof}

\begin{theorem}
\label{thm:normal}
Let $f(x,y,z)$ be a polynomial center function such that
\begin{equation}
\label{cond:perp}
f(a,b,c)-f(b,a,c)\quad\hbox{has}\quad a^2+b^2-c^2\quad\hbox{as a factor}.
\end{equation}
If the reference quadrilateral is orthodiagonal, then the central quadrilateral
is orthodiagonal whenever the chosen center has center function $f$.
\end{theorem}

\begin{proof}
If $Q$ is the center corresponding to $\triangle AEB$, then by Lemma \ref{lemma:angleBisector},
$EQ$ bisects $\angle AEB$. Similarly, if $P$ is the center corresponding to $\triangle AED$,
then $EP$ bisects $\angle AED$ (Figure \ref{fig:perpAngleBisectors}).

\begin{figure}[h!t]
\centering
\includegraphics[width=0.4\linewidth]{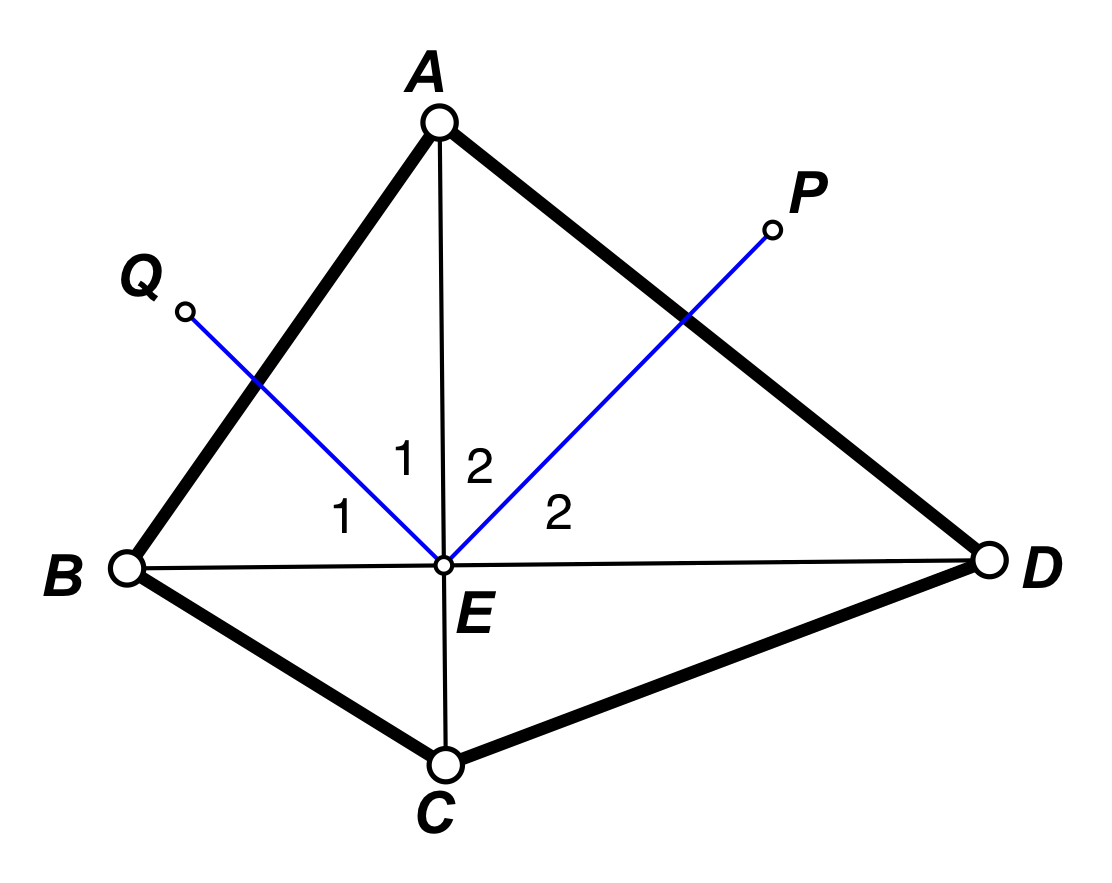}
\caption{orthodiagonal quadrilateral with two angle bisectors}
\label{fig:perpAngleBisectors}
\end{figure}

Then we see that angles 1 and 2 both have measure $45\degrees$, so $EQ\perp EP$.
Thus, the central quadrilateral is orthodiagonal and has the same diagonal point
as the reference quadrilateral.
\end{proof}

Let us call a function $f$ that has the form specified by Theorem \ref{thm:normal}
a \emph{normal center function}. We may also say that $f$ is \emph{normal}.

The following lemma should be obvious.

\begin{lemma}
\label{lemma:normal}
If $f$ is normal, then so is $f+k$, where $k$ is a constant.
If $f$ is normal, then so is $kf$, where $k$ is a constant.
If $f$ is normal, then so is $f^n$, where $n$ is a positive integer.
If $f$ and $g$ are normal, then so is $f+g$.
\end{lemma}

\begin{proof}
Let $g=f+k$. Then $g(a,b,c)-g(b,a,c)=f(a,b,c)-f(b,a,c)$.
Let $E=kf$. Then $E(a,b,c)-E(b,a,c)=k(f(a,b,c)-f(b,a,c))$.
If $h=f^n$, note that $h(a,b,c)-h(b,a,c)=f(a,b,c)^n-f(b,a,c)^n$ is 
divisible by $f(a,b,c)-f(b,a,c)$ which is divisible by $a^2+b^2-c^2$.
If $F=f+g$, note that $F(a,b,c)-F(b,a,c)=f(a,b,c)-f(b,a,c)+g(a,b,c)-g(b,a,c)$
which is divisible by $a^2+b^2-c^2$.
\end{proof}

\begin{theorem}
\label{thm:normal1}
If the reference quadrilateral is orthodiagonal, the central quadrilateral
is orthodiagonal whenever the chosen center has center function of the form
$$\cos B\cos C+k$$
for some constant $k$.
\end{theorem}

\begin{proof}
By Lemma \ref{lemma:normal}, it suffices to prove this for center function $\cos B\cos C$.
Applying transformation set (\ref{rules}), we find that 
$\cos B\cos C$ is equivalent to the center function
$$\frac{a^4-\left(b^2-c^2\right)^2}{4 a^2 b c}.$$
Multiplying this by the cyclic function $4a^2b^2c^2$ shows this to be equivalent to the 
center function
$$f(a,b,c)=bc(a^4-\left(b^2-c^2\right)^2).$$
A straightfoward calculation shows that
$$f(a,b,c)-f(b,a,c)=c (a-b) (a+b-c) (a+b+c) \left(a^2+b^2-c^2\right),$$
which has the factor $a^2+b^2-c^2$.
Thus, $f$ is normal and the central quadrilateral is orthodiagonal by Theorem \ref{thm:normal}.
\end{proof}

\void{
Some centers that have center functions of this form are $X_4$ ($\cos A\cos B$),
$X_{388}$ ($\cos A\cos B+1$), and $X_{497}$ ($\cos A\cos B-1$).
}

\begin{theorem}
\label{thm:normalPlus}
Let $f(x,y,z)$ be a polynomial center function such that
\begin{equation}
\label{cond:perp2}
f(a,b,c)+f(b,a,c)\quad\hbox{has}\quad a^2+b^2-c^2\quad\hbox{as a factor}.
\end{equation}
If the reference quadrilateral is orthodiagonal, then the central quadrilateral
is orthodiagonal whenever the chosen center has center function $f$.
\end{theorem}

\begin{proof}
The proof is the same as the proof of Theorem \ref{thm:normal} except that a variant of
Lemma \ref{lemma:angleBisector} is used to show when a point lies on the external angle bisector of
the right angle in a right triangle. The details are omitted.
\end{proof}

Using the same proof technique that was used in the proof of Theorem \ref{thm:normal1},
combined with either Theorem \ref{thm:normal} or Theorem \ref{thm:normalPlus},
we can prove the following theorem.

\begin{theorem}
\label{thm:normal2}
If the reference quadrilateral is orthodiagonal, the central quadrilateral
is orthodiagonal whenever the chosen center has center function of one of the 
following forms.
$$\sec A$$
$$\sec B$$
$$\tan A$$
$$\sin 2A$$
$$\cos B\cos C$$
$$\tan B+\tan C$$
$$\sec B+\sec C$$
$$\sec B-\sec C$$
$$\cos 2A$$
$$\cos B+\cos C-\cos A$$
$$1/(\cos B+\cos C-\cos A)$$
$$\sec 2A$$
$$\sin 4A$$
$$\tan B$$
$$\tan B-\tan C$$
\end{theorem}

Using these forms combined with Lemma \ref{lemma:normal} shows that
the following centers produce orthodiagonal central quadrilaterals
when the given quadrilateral is orthodiagonal.

\begin{center}
\begin{tabular}{|l|l|}\hline
\multicolumn{2}{|c|}{\textbf{A selection of centers that make the}}\\
\multicolumn{2}{|c|}{\textbf{central quadrilateral orthodiagonal}}\\ \hline
\textcolor{blue}{center}&\textcolor{blue}{center function}\\ \hline
%$X_{1}$&$1$\\ \hline
%$X_{4}$&$\sec A$\\ \hline
$X_{46}$&$\cos B+\cos C-\cos A$\\ \hline
$X_{47}$&$\cos 2A$\\ \hline
$X_{48}$&$\tan B+\tan C$\\ \hline
$X_{73}$&$\sec B+\sec C$\\ \hline
$X_{90}$&$1/(\cos B+\cos C-\cos A)$\\ \hline
$X_{91}$&$\sec 2A$\\ \hline
$X_{223}$&$\sec B+\sec C-\sec A-1$\\ \hline
$X_{388}$&$\cos B\cos C+1$\\ \hline
$X_{497}$&$\cos B\cos C-1$\\ \hline
$X_{563}$&$\sin 4A$\\ \hline
$X_{610}$&$\tan B+\tan C-\tan A$\\ \hline
$X_{652}$&$\sec B-\sec C$\\ \hline
$X_{656}$&$\tan B-\tan C$\\ \hline
$X_{822}$&$(\sec B+\sec C)(\sec B-\sec C)$\\ \hline
\end{tabular}
\end{center}

\begin{theorem}
\label{thm:ortho151}
If the reference quadrilateral is orthodiagonal, the central quadrilateral is equidiagonal if the chosen center is $X_{151}$.
\end{theorem}

\begin{theorem}
\label{thm:ortho6}
If the reference quadrilateral is orthodiagonal, the central quadrilateral is cyclic if the chosen center is the symmedian point ($X_{6})$.
\end{theorem}

\begin{proof}
When dealing with orthodiagonal quadrilaterals, it is convenient to set up a
coordinate system where the diagonal point of the quadrilateral is at the origin
and the diagonals of the quadrilateral lie along the $x$- and $y$-axes.
We can scale the figure so that vertex $D$ lies at $(1,0)$.
Figure \ref{fig:orthoCoordinates} shows the coordinates we use, where $a_y>0$,
$b_x<0$, and $c_y<0$.

\begin{figure}[h!t]
\centering
\includegraphics[width=0.5\linewidth]{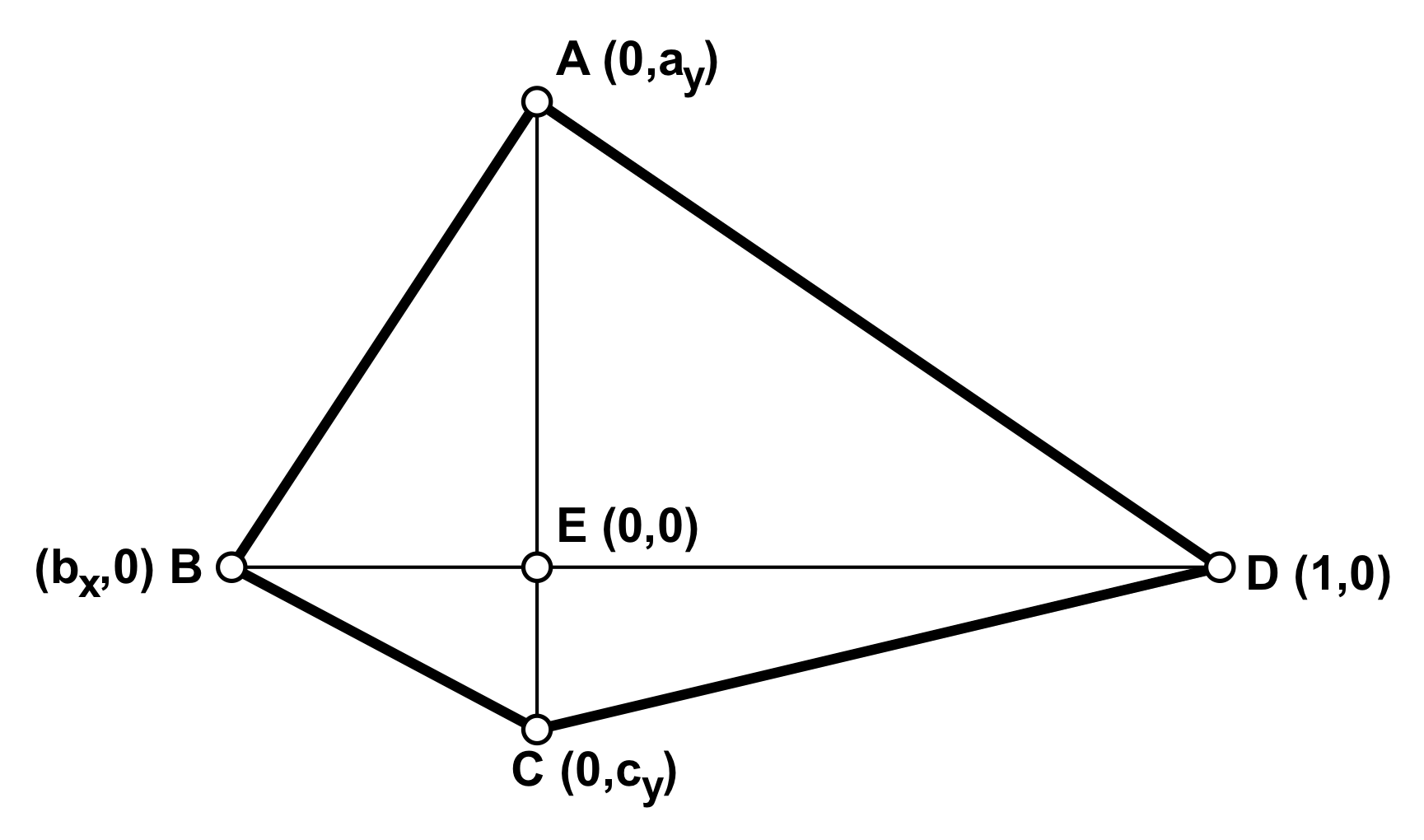}
\caption{Cartesian coordinates for an orthodiagonal quadrilateral}
\label{fig:orthoCoordinates}
\end{figure}

Using $a$ as the center function for $X_6$, we can compute the coordinates for
$F$, $G$, $H$, and $I$. We find that the coordinates for $F$ are
$$F=\left(\frac{b_xa_y^2}{2(a_y^2+2b_x^2)},\frac{a_yab_x^2}{2(a_y^2+2b_x^2)}\right).$$
Similar expressions are found for the coordinates for $G$, $H$, and $I$.
Then we find the coordinates for $P$, the intersection of $FH$ and $GI$.
We find
$$P=\left(\frac{b_xa_y^2c_y^2(b_x+1)}{2\left(b_x^2c_y^2+a_y^2c_y^2+a_y^2b_x^2(1+c_y^2)\right)},
\frac{a_yb_x^2c_y(a_y+c_y)}
{2\left(b_x^2c_y^2+a_y^2c_y^2+a_y^2b_x^2(1+c_y^2)\right)}\right).$$
Using the distance formula, we find the lengths of $PF$, $PG$, $PH$, and $PI$.
The points $F$, $G$, $H$, and $I$ will lie on a circle if $PF\cdot PH=PG\cdot PI$
(the converse of the intersecting chords theorem for a circle).
A straightforward computation (using Mathematica) shows this to be true.
Hence $FGHI$ is a cyclic quadrilateral.
\end{proof}

\begin{open}
\label{conj:symmedian}
\ \\Is there a purely geometric proof for Theorem \ref{thm:ortho6}?
(Figure~\ref{fig:dpOrthodiagonalSymmedian})
\end{open}

\begin{figure}[h!t]
\centering
\includegraphics[width=0.3\linewidth]{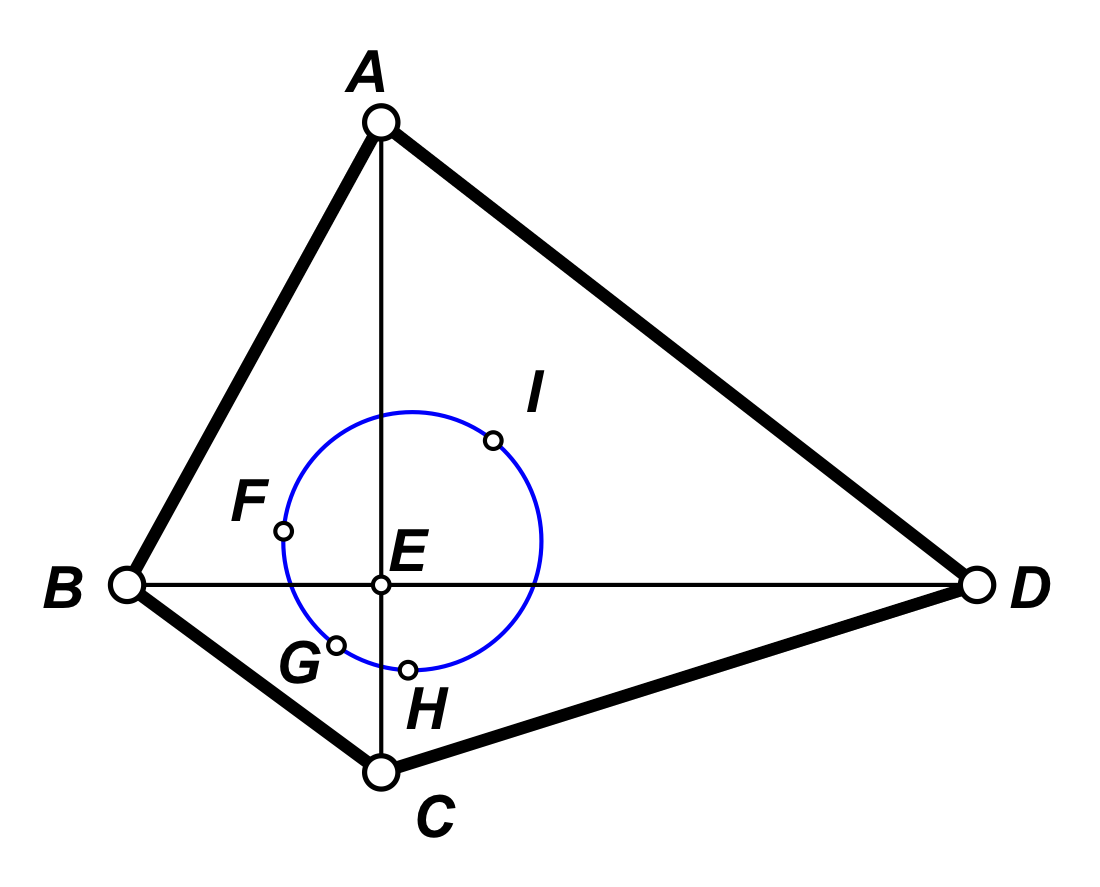}
\caption{Orthodiagonal quadrilateral: symmedian points $\implies$ cyclic}
\label{fig:dpOrthodiagonalSymmedian}
\end{figure}

\void{
\begin{theorem}
\label{thm:orthoSymmetric}
If the reference quadrilateral is orthodiagonal, the central quadrilateral is cyclic whenever the chosen center has center function of the form
$$f(a,b,c)\sin A$$
where $f$ is a cyclic function, $f(a,b,c)=f(b,c,a)=f(c,a,b)$.
\end{theorem}

\begin{proof}
We noted earlier (page \pageref{cyclicSymmetric}) that equivalent center functions
resulted in identical centers.

Since $f(a,b,c)$ is a cyclic function,
the center function $f(a,b,c)\sin A$ is equivalent to the center function $\sin A$.
It therefore suffices to prove that this theorem is true when the center function is
$\sin A$. We have already done that (Theorem \ref{thm:ortho6}) since according to \cite{ETC},
the symmedian point ($X_6$) has center function $\sin A$.
\end{proof}
}

\begin{theorem}
\label{thm:orthoCyclicTan}
If the reference quadrilateral is orthodiagonal, the central quadrilateral is cyclic whenever the chosen center has center function of the form
$$k\cos A+\sin A(\tan A+\tan B+\tan C)$$
for some constant $k$.
\end{theorem}

\textbf{Note.} Examples of centers that have center functions of this form are $X_{389}$ and $X_{578}$.

\void{
\begin{theorem}
\label{thm:orthoOrthoF}
If the reference quadrilateral is orthodiagonal, the central quadrilateral is orthodiagonal whenever the chosen center has center function of the form
$$\bigl(f(a)+\tan B+\tan C\bigr)^n$$
where $n$ is a positive integer and $f$ is any function.
\end{theorem}

\begin{theorem}
\label{thm:orthoOrtho}
If the reference quadrilateral is orthodiagonal, the central quadrilateral is orthodiagonal whenever the chosen center has center function of the form
$$\cos B\cos C+k$$
where $k$ is a constant.
\end{theorem}

\begin{theorem}
\label{thm:orthoOrtho2}
If the reference quadrilateral is orthodiagonal, the central quadrilateral is orthodiagonal whenever the chosen center has center function of the form
$$(\cos B+\cos C-\cos A)^n$$
where $n$ is a positive integer.
\end{theorem}

\begin{theorem}
\label{thm:orthoOrtho3}
If the reference quadrilateral is orthodiagonal, the central quadrilateral is orthodiagonal whenever the chosen center has center function of the form
$$(\sec B+\sec C-\sec A)^n+k$$
where $k$ is a constant and $n$ is a positive integer.
\end{theorem}

\begin{theorem}
\label{thm:orthoOrtho4}
If the reference quadrilateral is orthodiagonal, the central quadrilateral is orthodiagonal whenever the chosen center has center function of the form
$$\cos^n (2A)$$
where $n$ is a positive integer.
\end{theorem}
}

\begin{theorem}
\label{thm:orthoRect}
If the reference quadrilateral is orthodiagonal, the central quadrilateral is a rectangle whenever the chosen center has center function of the form
$$\cos A+k\cos B\cos C$$
where $k$ is a constant.
\end{theorem}

For some special cases of this theorem, a simpler geometrical proof is possible.

The centroid has center function $\cos A+\cos B\cos C$.
This is of the form referenced by Theorem \ref{thm:orthoRect} when $k=1$.
We thus have the following corollary.

\begin{corollary}
\label{cor:orthoCentroid}
If the reference quadrilateral is orthodiagonal, the central quadrilateral is a rectangle
when the chosen center is the centroid
(Figure \ref{fig:dpOrthoCentroid}).
\end{corollary}

\begin{figure}[h!t]
\centering
\includegraphics[width=0.3\linewidth]{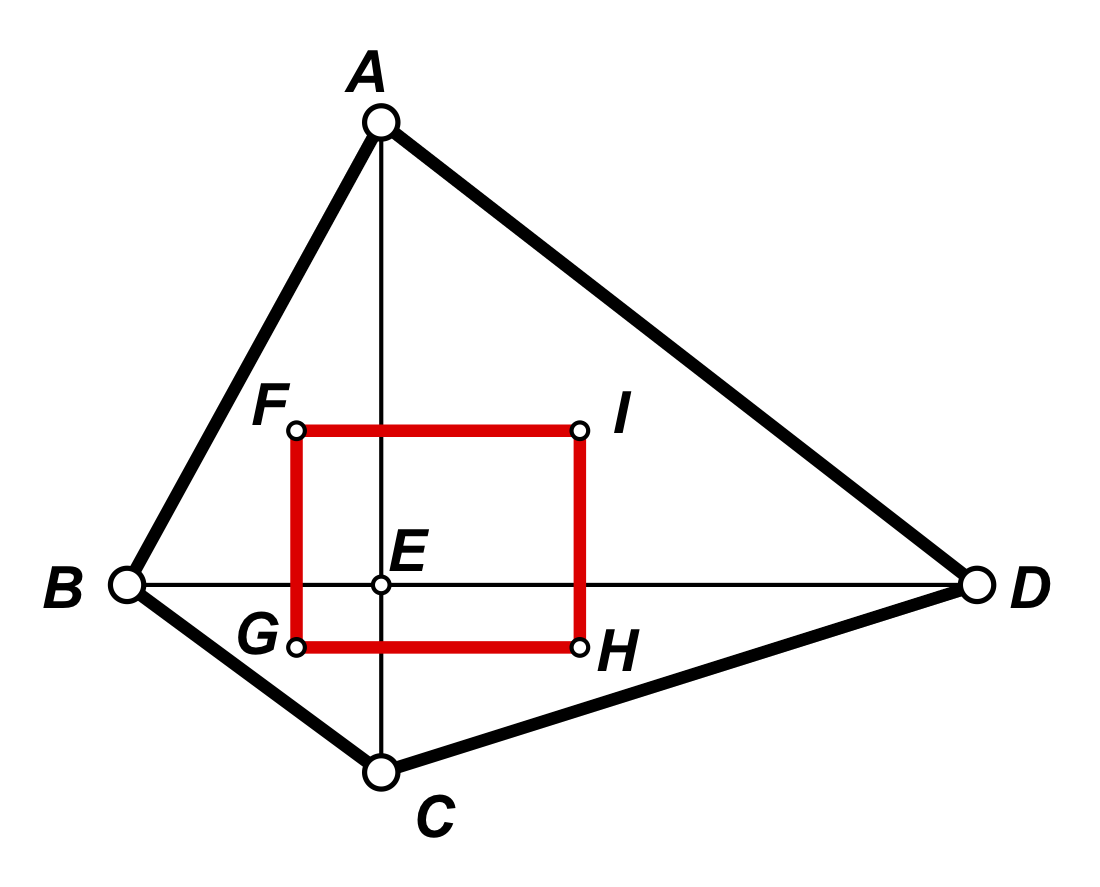}
\caption{Orthodiagonal quadrilateral: centroids $\implies$ rectangle}
\label{fig:dpOrthoCentroid}
\end{figure}

\begin{proof}
By Theorem \ref{thm:genCentroids}, $FGHI$ is a parallelogram.
From the proof of Theorem \ref{thm:genCentroids}, we see that $FI\parallel BD$
and $FG\parallel AC$. Since $BD\perp AC$, this implies $IF\perp FG$.
But a parallelogram having one angle a right angle must be a rectangle.
\end{proof}

The circumcenter has center function $\cos A$.
This is of the form referenced by Theorem \ref{thm:orthoRect} when $k=0$.
We thus have the following corollary.

\begin{corollary}
\label{cor:orthoCircumcenter}
If the reference quadrilateral is orthodiagonal, the central quadrilateral is a rectangle
when the chosen center is the circumcenter
(Figure \ref{fig:dpOrthoCircumcenter}).
\end{corollary}

\begin{figure}[h!t]
\centering
\includegraphics[width=0.3\linewidth]{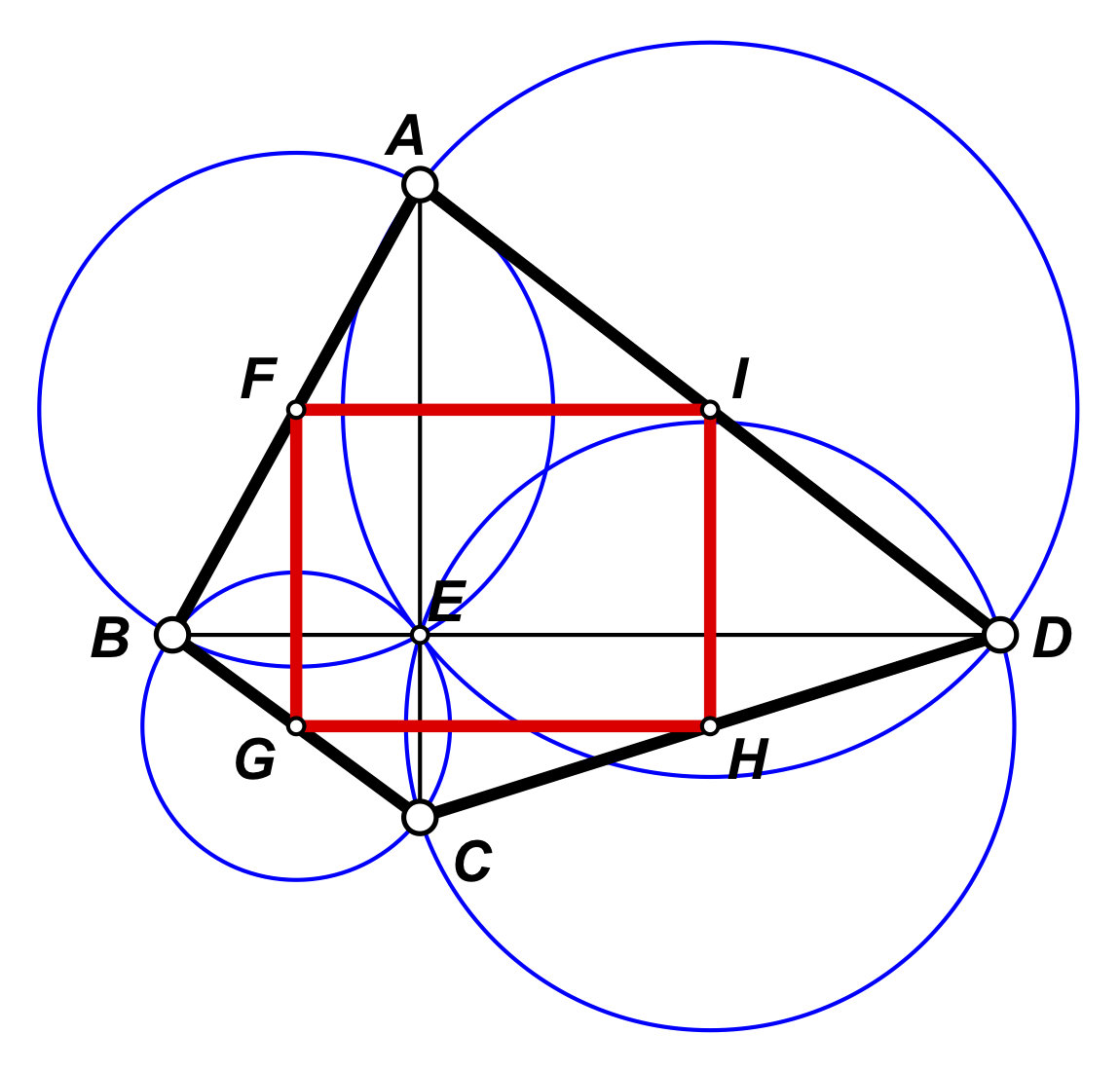}
\caption{Orthodiagonal quadrilateral: circumcenters $\implies$ rectangle}
\label{fig:dpOrthoCircumcenter}
\end{figure}

\begin{proof}
By Theorem \ref{thm:genCircumcenters}, $FGHI$ is a parallelogram.
From the proof of Theorem \ref{thm:genCircumcenters}, we see that $FI\perp AC$
and $FG\perp BD$. Since $BD\perp AC$, this implies $IF\perp FG$.
But a parallelogram having one angle a right angle must be a rectangle.
\end{proof}

We could also have noted that the circumcenter of a right triangle is the midpoint
of the hypotenuse and used that fact to prove that the sides of the central
quadrilateral are parallel to the diagonals of the reference quadrilateral.

The nine-point center has center function $\cos A+2\cos B\cos C$.
This is of the form referenced by Theorem \ref{thm:orthoRect} when $k=2$.
We thus have the following corollary.

\begin{corollary}
\label{thm:ortho9point}
If the reference quadrilateral is orthodiagonal, the central quadrilateral is a rectangle
when the chosen center is the nine-point center.
(Figure~\ref{fig:dpOrtho9}).
\end{corollary}

\begin{figure}[h!t]
\centering
\includegraphics[width=0.3\linewidth]{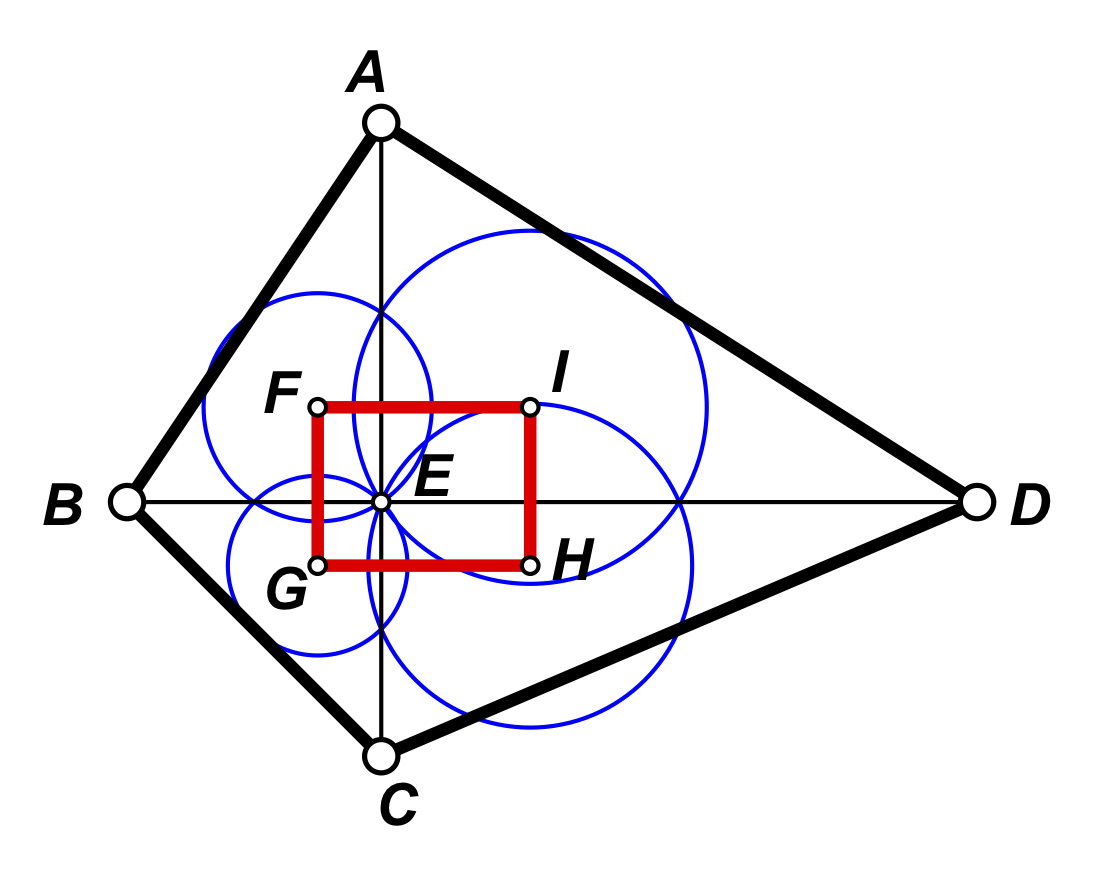}
\caption{Orthodiagonal quadrilateral: nine-point centers $\implies$ rectangle}
\label{fig:dpOrtho9}
\end{figure}

\begin{proof}
Since $\odot F$ is the nine-point circle of right triangle $AEB$, it passes through the
foot of the altitude from $B$ to $AE$, which is point $E$. It also passes through the midpoint
of $AE$. In the same way, $\odot I$ passes through $E$ and the midpoint
of $AE$. Thus, $AE$ is the common chord between these two circles which implies
that the line of centers, $FI$ is perpendicular to $AC$.
Similarly, $FG\perp BD$, $GH\perp AC$, and $HI\perp BD$.
Since $AC\perp BD$, this implies that $FG\perp GH$, $GH\perp HI$, $HI\perp IF$, and $IF\perp FG$,
making $FGHI$ a rectangle.
\end{proof}

\subsection{Hjelmslev Quadrilaterals}

\begin{theorem}
\label{thm:hj}
If the reference quadrilateral is a Hjelmslev quadrilateral, the central quadrilateral is orthodiagonal whenever the chosen center has center function of the form
$$\tan(2B)+\tan(2C)+k\tan(2A)$$
for some constant $k$.
\end{theorem}

\begin{theorem}
\label{thm:hjPrasolov}
If the reference quadrilateral is a Hjelmslev quadrilateral, then the central quadrilateral is a trapezoid when the chosen center is the Prasolov point ($X_{68}$).
See Figure \ref{fig:dpHjelmslev-trapezoid}.
\end{theorem}

\begin{figure}[h!t]
\centering
\includegraphics[width=0.4\linewidth]{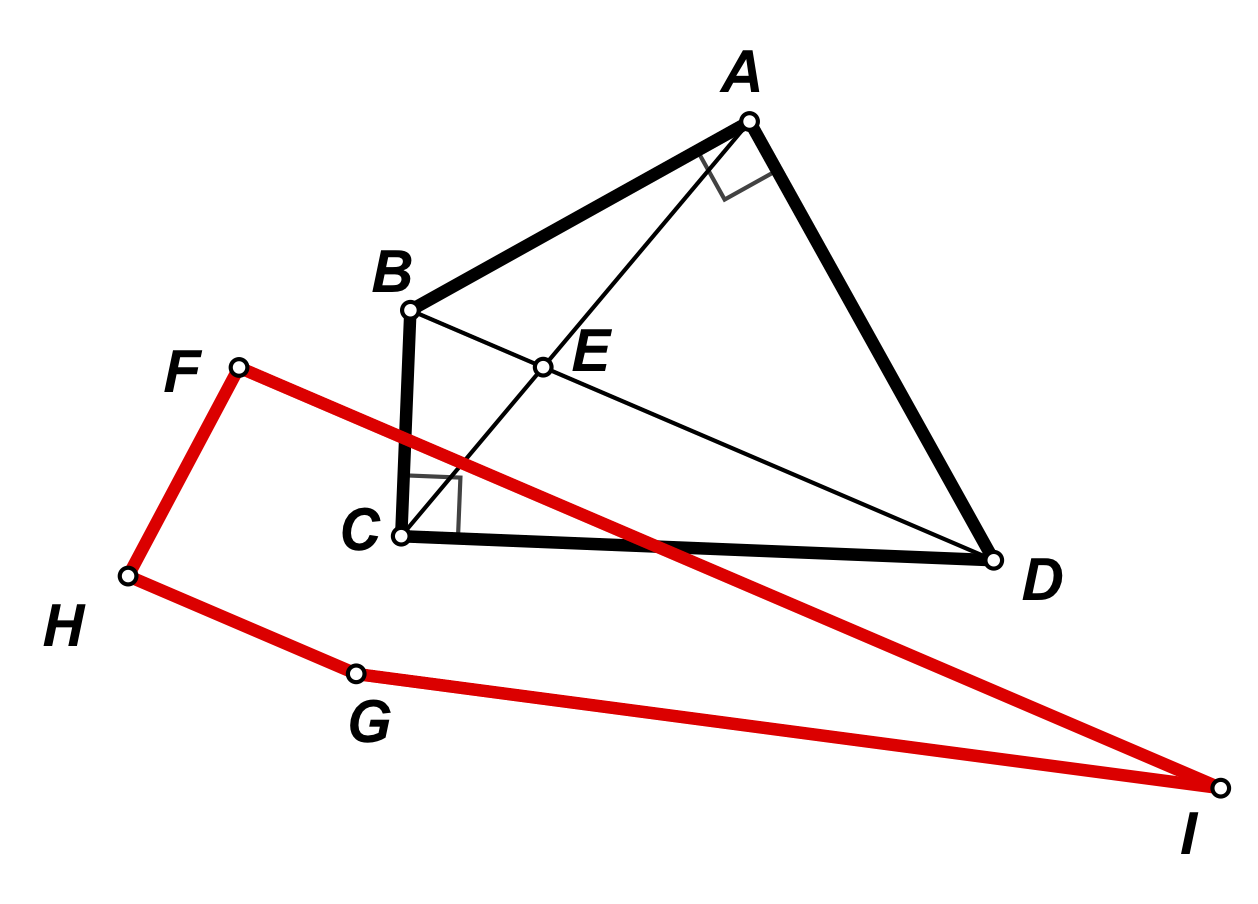}
\caption{Prasolov points $\implies$ $FI\parallel HG\parallel BD$}
\label{fig:dpHjelmslev-trapezoid}
\end{figure}

\subsection{Parallelograms}

\begin{lemma}[The Parallelogram Lemma]
\label{lemma:parallelogram}
Let $ABCD$ be a parallelogram and
let $E$ be the intersection of the diagonals.
Let $P$ be a triangle center of $\triangle BEC$ and
let $Q$ be the corresponding triangle center of $\triangle DEA$.
Then $PQ$ passes through $E$ and $PE=QE$.
(Figure \ref{fig:parallelogramLemma})
\end{lemma}

\begin{figure}[h!t]
\centering
\includegraphics[width=0.7\linewidth]{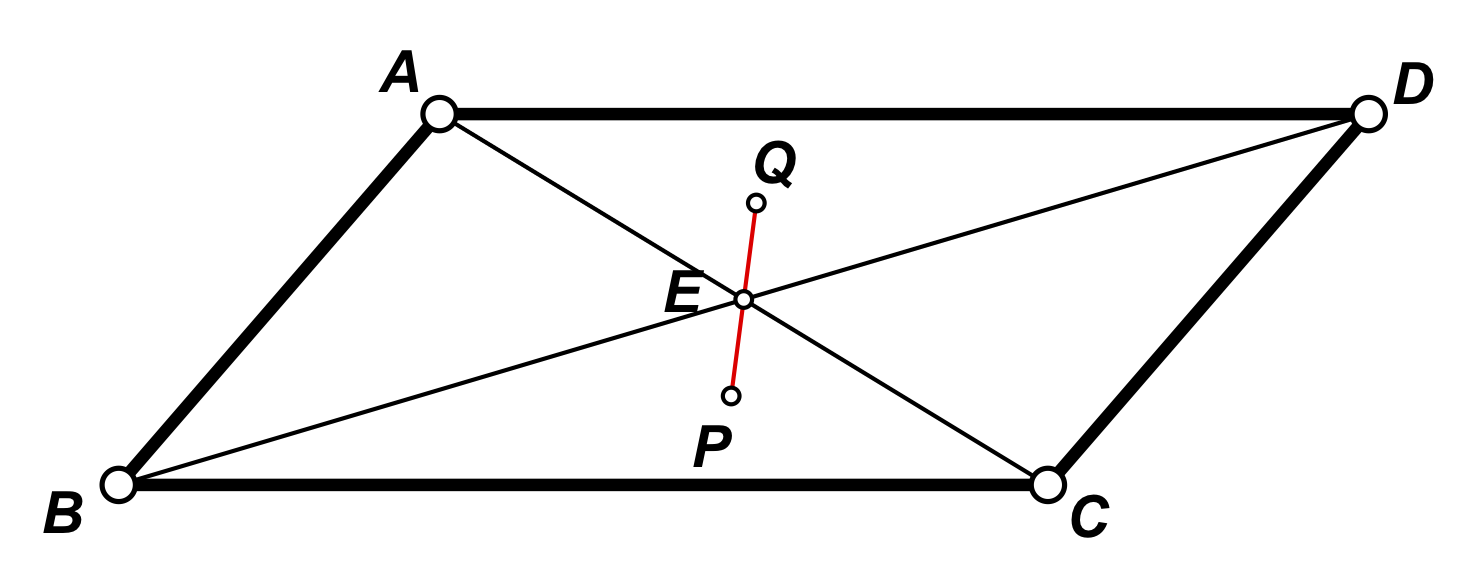}
\caption{}
\label{fig:parallelogramLemma}
\end{figure}

\begin{proof}
Note that triangles $BEC$ and $DEA$ are congruent.
Under the congruence transformation that maps $\triangle BEC$ into $\triangle DEA$, $P$ will get mapped into $Q$.
Since $E$ is the center of the congruence transformation, this means $PQ$ will pass through $E$. Congruence preserves lengths, so $PE=QE$.
\end{proof}

\begin{theorem}
\label{thm:arbCenterParallelogram}
For any triangle center,
if the reference quadrilateral is a parallelogram,
then the central quadrilateral is a parallelogram.
The sides of the central quadrilateral are parallel to the diagonals
of the reference quadrilateral. The two quadrilaterals have the same diagonal point.
(Figure \ref{fig:diagPtParallelogram})
\end{theorem}

\begin{figure}[h!t]
\centering
\includegraphics[width=0.7\linewidth]{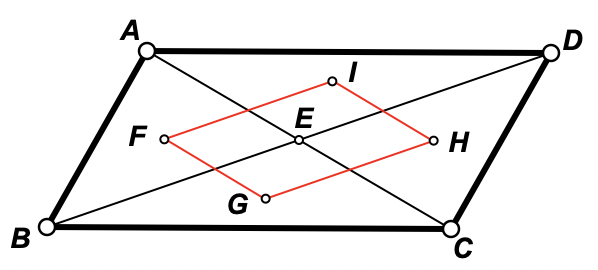}
\caption{parallelogram $\implies$ parallelogram}
\label{fig:diagPtParallelogram}
\end{figure}

\begin{proof}
By the Parallelogram Lemma, $E$ is the midpoint of $GI$.
Similarly, $E$ is the midpoint of $FH$.
Since the diagonals of quadrilateral $FGHI$ bisect each other, $FGHI$ must
be a parallelogram.
\end{proof}

\begin{theorem}
\label{thm:X46}
If the reference quadrilateral is a parallelogram, the central quadrilateral is a rhombus if the chosen center is the incenter ($X_1$) or one of the Vecten points ($X_{485}$ or $X_{486}$).
\end{theorem}

\subsection{Kites}

\begin{lemma}[The Isosceles Triangle Lemma]
\label{lemma:isosceles}
Let $ABC$ be an isosceles triangle (with $AB=AC$) and
let $D$ be the midpoint of $BC$.
Let $P$ be a triangle center of $\triangle ABD$ and
let $Q$ be the corresponding triangle center of $\triangle ACD$.
Then $PQ||BC$, $PQ\perp AD$, and $PQ$ is bisected by $AD$.
\end{lemma}

\begin{figure}[h!t]
\centering
\includegraphics[width=0.3\linewidth]{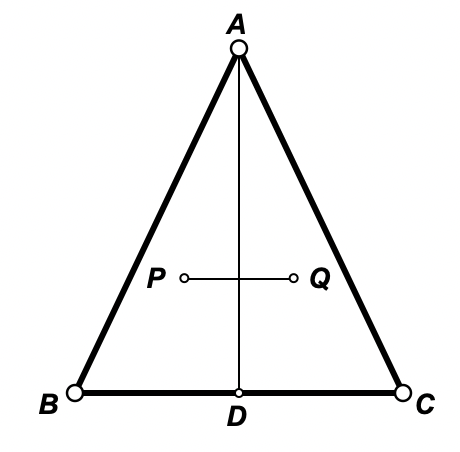}
\label{fig:isoscelesLemma}
\end{figure}

\begin{proof}
Note that $AD\perp BC$.
Under the congruence transformation that maps $\triangle ABD$ into $\triangle ACD$, $Q$ will be the reflection of $P$ about $AD$ (and $PQ$ is bisected by $AD$).
Thus $PQ\perp AD$ which implies $PQ||BC$.
\end{proof}

\begin{theorem}
\label{thm:arbCenterKite}
For any triangle center,
if the reference quadrilateral is a kite,
then the central quadrilateral is an isosceles trapezoid.
The parallel sides of the trapezoid are parallel to one of the diagonals
of the reference quadrilateral and are bisected by the other diagonal
(Figure \ref{fig:diagPtKite}).
\end{theorem}

\begin{figure}[h!t]
\centering
\includegraphics[width=0.4\linewidth]{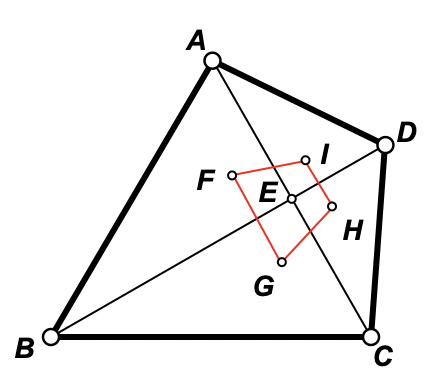}
\caption{kite $\implies$ isosceles trapezoid}
\label{fig:diagPtKite}
\end{figure}

\begin{proof}
Triangle  $DAC$ is isosceles, so by the Isosceles Triangle Lemma (Lemma \ref{lemma:isosceles}),
$HI||AC$. Similarly, $FG||AC$, so $FGHI$ is a trapezoid. Also by the Isosceles Triangle Lemma,
$ED$ bisects $HI$. From the fact that $H$ and $I$ are equidistant from $BD$,
and $F$ and $G$ are equidistant from $BD$, we can conclude that $GH=FI$ and the trapezoid is isosceles.
\end{proof}

\begin{theorem}
\label{thm:kiteVecten}
If the reference quadrilateral is a kite, the central quadrilateral is a square if the chosen center is the incenter ($X_1$) or one of the Vecten points ($X_{485}$ or $X_{486}$).
\end{theorem}

\textbf{Note.} There are many other centers for which the central quadrilateral is a square.

\subsection{Rhombi}

\begin{theorem}
\label{thm:arbCenterRhombus}
For any triangle center,
if the reference quadrilateral is a rhombus,
then the central quadrilateral is a rectangle.
The two quadrilaterals have the same diagonal point.
The sides of the rectangle are parallel to diagonals of the rhombus
and are bisected by them
(Figure \ref{fig:diagPtRhombus}).
\end{theorem}

\begin{figure}[h!t]
\centering
\includegraphics[width=0.4\linewidth]{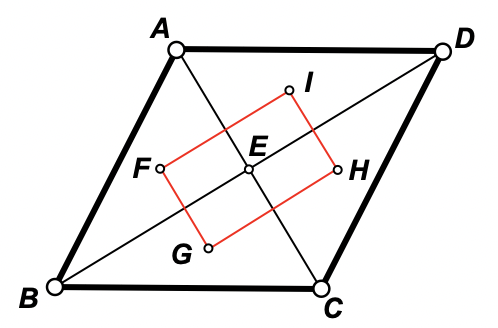}
\caption{rhombus $\implies$ rectangle}
\label{fig:diagPtRhombus}
\end{figure}

\begin{proof}
By Theorem \ref{thm:arbCenterKite}, the central quadrilateral must be an isosceles trapezoid.
By Theorem \ref{thm:arbCenterParallelogram}, the central quadrilateral must be a parallelogram.
But a quadrilateral that is both an isosceles trapezoid and a parallelogram must be a rectangle.
\end{proof}

\void{ %bogus; not sure where I came up with this.
\subsection{Trapezoids}

\begin{theorem}
\label{thm:arbCenterTrapezoid}
For any triangle center,\\
if the reference quadrilateral is a trapezoid,
then the central quadrilateral is an orthodiagonal quadrilateral.
The two quadrilaterals have the same diagonal point.
\end{theorem}

\begin{figure}[h!t]
\centering
\includegraphics[width=0.5\linewidth]{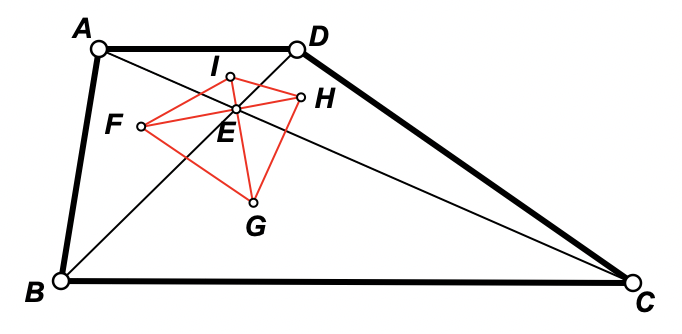}
\caption{trapezoid $\implies$ orthodiagonal}
\label{fig:diagPtTrapezoid}
\end{figure}

\begin{proof}
For a proof, see Appendix \ref{appendix:diagonalPointProofs} (pending).
\end{proof}
}

\subsection{Isosceles Trapezoids}

\void{
\begin{lemma}[The Trapezoid Lemma]
\label{lemma:trapezoid}
Let $ABCD$ be trapezoid (with $AD||BC$) and
let $E$ be the intersection of the diagonals.
Let $P$ be a triangle center of $\triangle BEC$ and
let $Q$ be the corresponding triangle center of $\triangle DEA$.
Then $PQ$ passes through $E$.
\end{lemma}

\begin{figure}[h!t]
\centering
\includegraphics[width=0.5\linewidth]{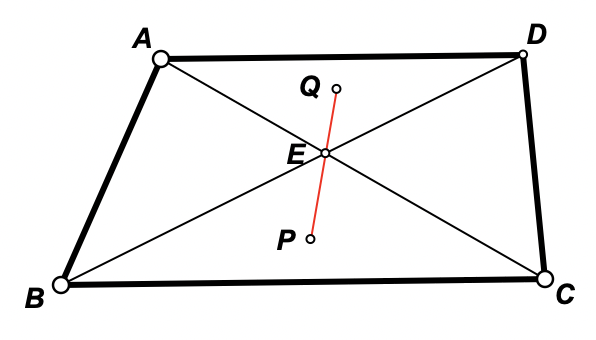}
\label{fig:trapezoidLemma}
\end{figure}

\begin{proof}
Note that triangles $BEC$ and $DEA$ are similar.
Under the similarity transformation that maps $\triangle BEC$ into $\triangle DEA$, $P$ will get mapped into $Q$.
Since $E$ is the center of the similarity, this means $PQ$ will pass through $E$.
\end{proof}

Do we need this lemma anywhere?
}

\begin{theorem}
\label{thm:isoTrap}
For any triangle center,
if the reference quadrilateral is an isosceles trapezoid,
then the central quadrilateral is a kite.
The two quadrilaterals have the same diagonal point.
One diagonal is parallel to the parallel sides of the trapezoid.
(Figure \ref{fig:diagPtIsoscelesTrapezoid})
\end{theorem}

\begin{figure}[h!t]
\centering
\includegraphics[width=0.4\linewidth]{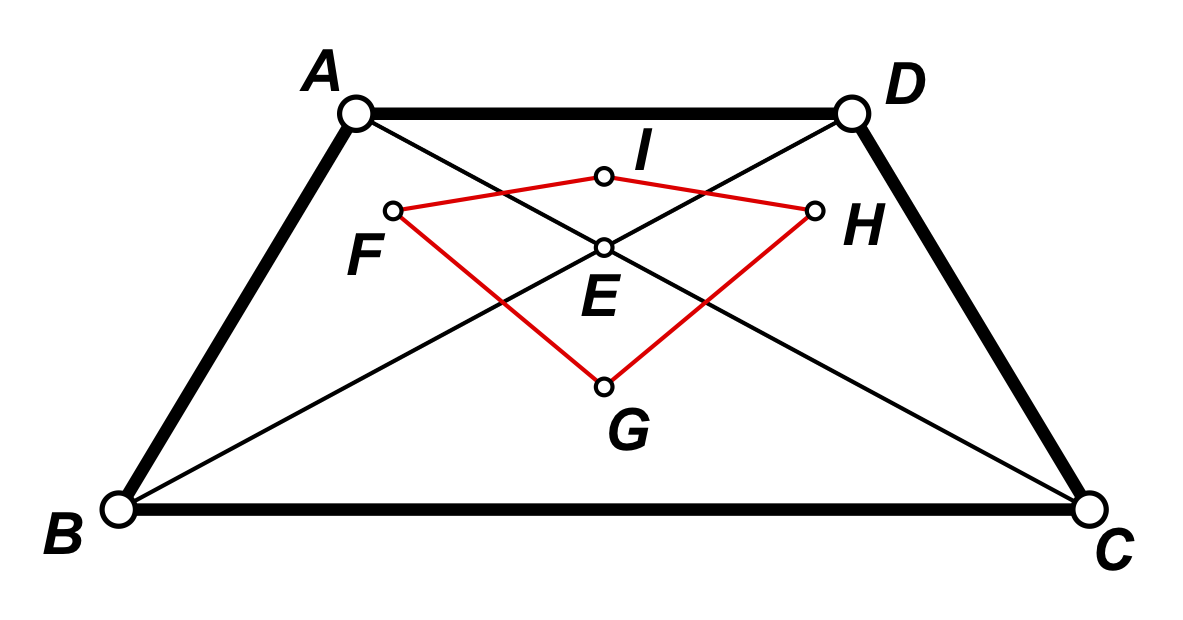}
\caption{isosceles trapezoid $\implies$ kite}
\label{fig:diagPtIsoscelesTrapezoid}
\end{figure}

\begin{proof}
Note that $\triangle AED$ and $\triangle BEC$ are isosceles triangles.
All triangle centers of an isosceles triangle lie on the altitude to the base.
We can conclude that $G$ and $I$ lie on the perpendicular bisector of $BC$
(which passes through $E$).
Now note that $\triangle AEB\cong \triangle DEC$.
One triangle is the mirror image of the other when reflected about the
perpendicular bisector of $BC$.
Under the reflection that maps $\triangle AEB$ into $\triangle DEC$, $F$ will get mapped into $H$. Thus, $GI$ is the perpendicular bisector of $FH$.
Hence, $FI=HI$ and $FG=HG$ which means that quadrilateral $FGHI$ is a kite.
\end{proof}

\subsection{Rectangles}

\begin{theorem}
\label{thm:arbCenterRectangle}
For any triangle center,
if the reference quadrilateral is a rectangle,
then the central quadrilateral is a rhombus.
The two quadrilaterals have the same diagonal point.
The sides of the rhombus are parallel to diagonals of the rectangle
and are bisected by them
(Figure \ref{fig:diagPtRectangle}).

\end{theorem}

\begin{figure}[h!t]
\centering
\includegraphics[width=0.3\linewidth]{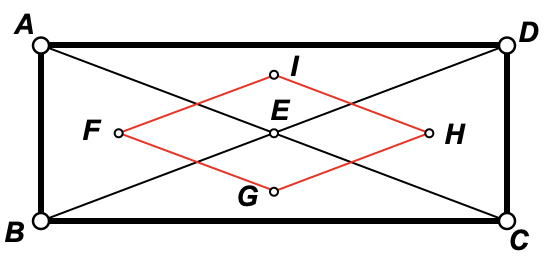}
\caption{rectangle $\implies$ rhombus}
\label{fig:diagPtRectangle}
\end{figure}

\begin{proof}
A rectangle can be considered to be an isosceles trapezoid in two different ways.
Thus, this result follows from Theorem \ref{thm:isoTrap}
since both quadrilaterals $FGHI$ and $GHIF$ are kites.
\end{proof}

\subsection{Squares}

\begin{theorem}
\label{thm:arbCenterSquare}
For any triangle center,
if the reference quadrilateral is a square,
then the central quadrilateral is a square
(Figure \ref{fig:diagPtSquare}).
\end{theorem}

\begin{figure}[h!t]
\centering
\includegraphics[width=0.15\linewidth]{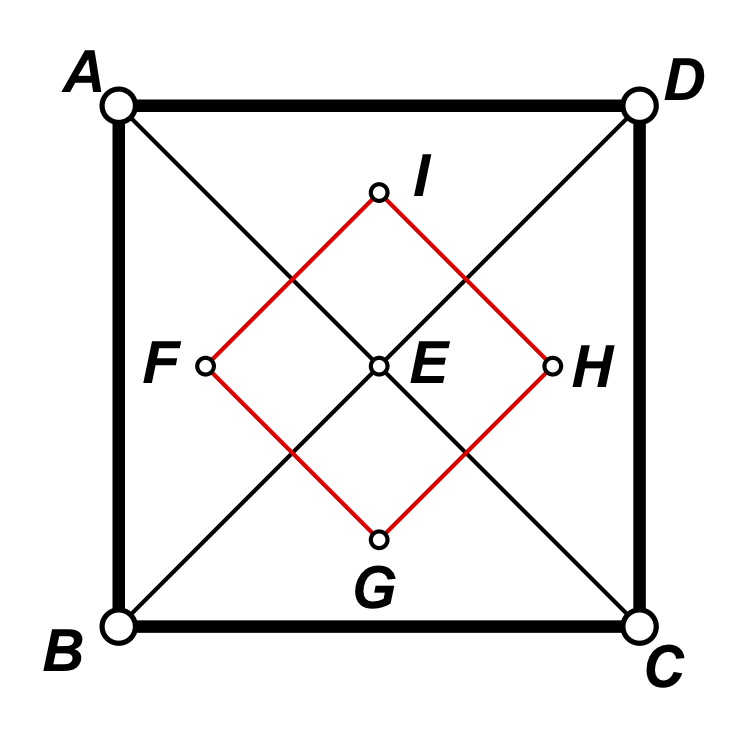}
\caption{square $\implies$ square}
\label{fig:diagPtSquare}
\end{figure}

\begin{proof}
By Theorem \ref{thm:arbCenterRectangle}, the central quadrilateral must be a rhombus.
By Theorem \ref{thm:arbCenterRhombus}, the central quadrilateral must be a rectangle.
But a quadrilateral that is both a rhombus and a rectangle must be a square.
\end{proof}

\begin{theorem}
\label{thm:X80Square}
Let $ABCD$ be a square with diagonal point $E$.
Let $F$ be the reflection of the incenter of $\triangle ABE$ about the Feuerbach point
of $\triangle ABE$ ($X_{80}$). Define $G$, $H$, and $I$ similarly using triangles
$\triangle BCE$, $\triangle CDE$, and $\triangle DAE$.
Then quadrilateral $FGHI$ is congruent to $ABCD$.
\end{theorem}

The same result holds for the Bevan point ($X_{40}$) and its isogonal conjugate ($X_{84}$).
See Figure \ref{fig:diagPtCongruentSquare}.
This is also an example where the vertices of the reference quadrilateral and the central
quadrilateral lie on the same circle.

\begin{figure}[h!t]
\centering
\includegraphics[width=0.15\linewidth]{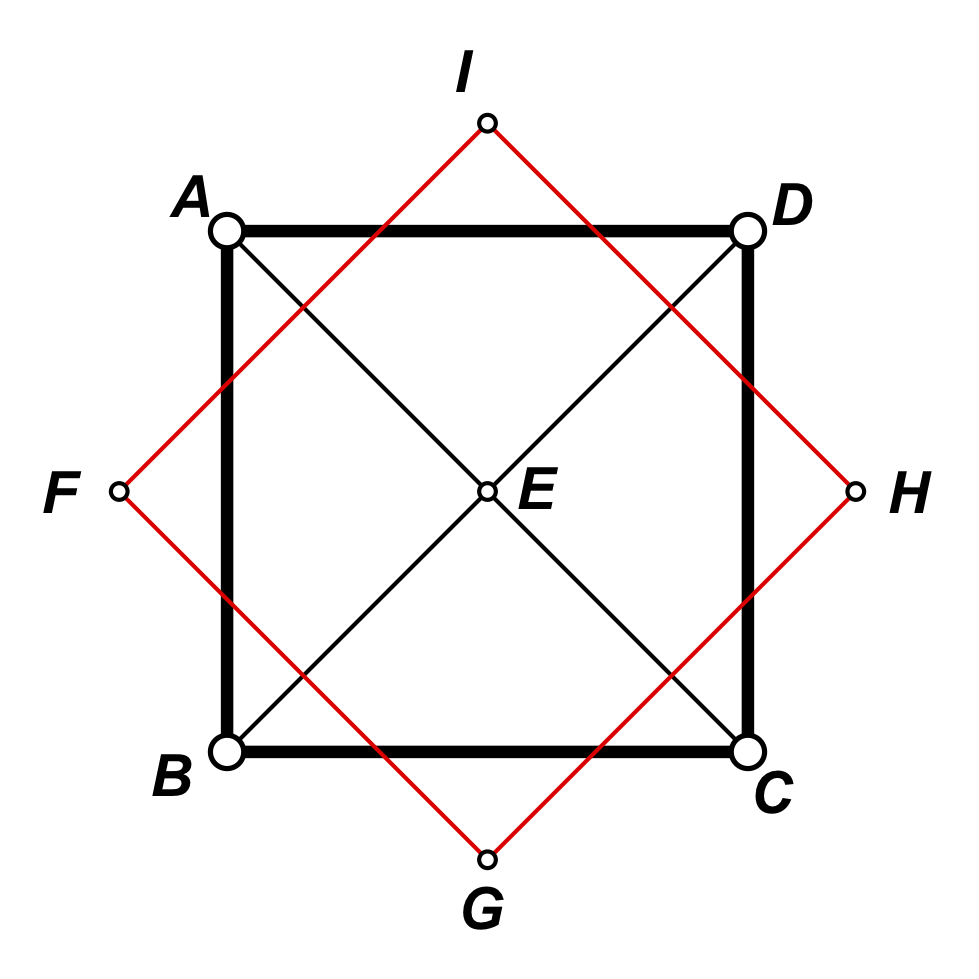}
\caption{square: Bevan points $\implies$ congruent square}
\label{fig:diagPtCongruentSquare}
\end{figure}

%**************************************
%    Half Triangles
%**************************************

\section{Results for Half Triangles}
\label{section:halfTriangles}

In this configuration, the reference quadrilateral is named $ABCD$.
Each diagonal of the quadrilateral divides
the quadrilateral into two triangles. Since there are two diagonals,
a total of four triangles are formed.
The four triangles (numbered 1 to 4) are shown
in Figure \ref{fig:sideTriangles}.

\begin{figure}[h!t]
\centering
\includegraphics[width=0.6\linewidth]{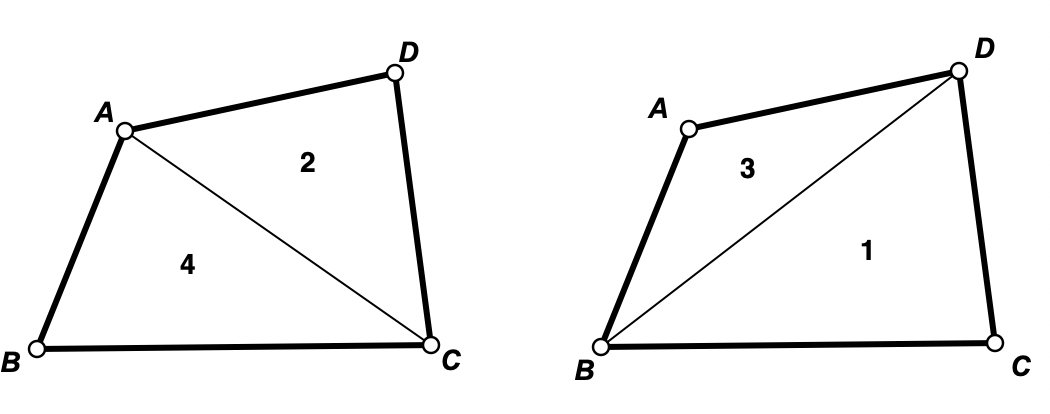}
\caption{Half Triangles}
\label{fig:sideTriangles}
\end{figure}

The triangles have been numbered so that triangle 1 is opposite vertex $A$,
triangle 2 is opposite vertex $B$, etc.
The four triangles are $\triangle BCD$, $\triangle ACD$, $\triangle ABD$, and $\triangle ABC$.
Triangle centers are selected in each triangle.
In order, their names are $E$, $F$, $G$, and $H$.

The raw data collected can be
found in Appendix \ref{appendix:halfTriangleData}.
Looking for patterns in the raw data, we found a number of theorems and
conjectures which are presented below.

\subsection{General Quadrilaterals}

\begin{conjecture}
\label{conjecture:halfGeneralCyclic}
If the reference quadrilateral is a general quadrilateral, $ABCD$, there is no center function
such that the central quadrilateral is a cyclic quadrilateral for all quadrilaterals $ABCD$.
\end{conjecture}

\begin{conjecture}
\label{conjecture:halfGeneralRectangle}
If the reference quadrilateral is a general quadrilateral, $ABCD$, there is no center function
such that the central quadrilateral is a rectangle for all quadrilaterals $ABCD$.
\end{conjecture}

\begin{theorem}
\label{thm:genSimilar}
If the reference quadrilateral is a general quadrilateral and the chosen center is $X_2$, then the central quadrilateral is similar to the reference quadrilateral. The ratio of similitude is 3
(Figure \ref{fig:halfGenSimilar}).
\end{theorem}

\begin{figure}[h!t]
\centering
\includegraphics[width=0.35\linewidth]{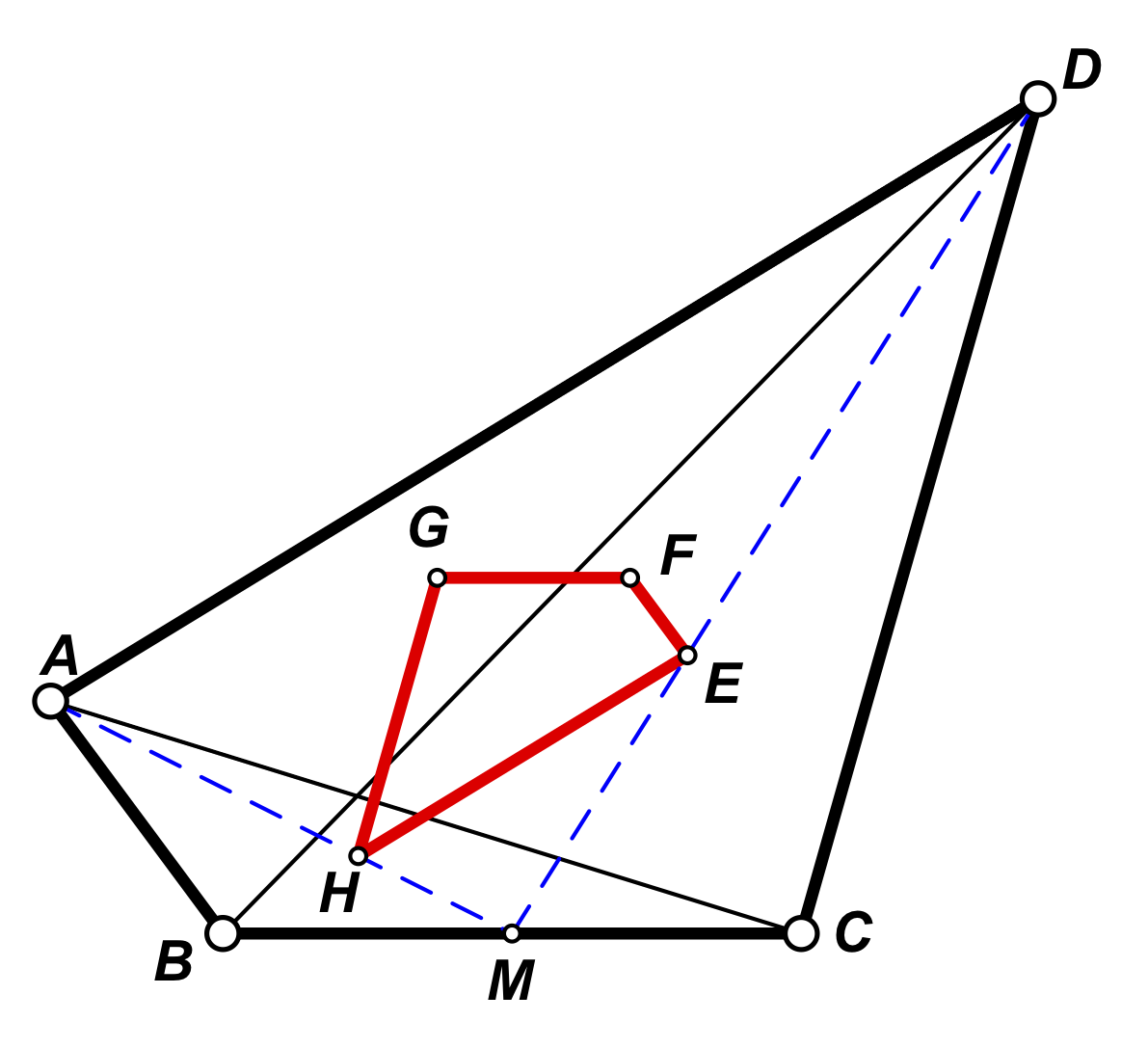}
\caption{centroids $\implies$ $EFGH\sim ABCD$}
\label{fig:halfGenSimilar}
\end{figure}

\begin{proof}
Let $M$ be the midpoint of $BC$. Then $AM$ and $DM$ are medians and $DM/EM=3$ and $AM/HM=3$.
Therefore, in $\triangle AMD$, we must have $HE\parallel AD$ and $AD/HE=3$.
Similarly, we have the same ratio and parallelism for $EF$, $FG$, and $GH$.
Thus, the sides of quadrilateral $ABCD$ are 3 times the corresponding sides of quadrilateral $EFGH$.
The parallelism implies that the angles of the two quadrilaterals are equal.
Since the sides are in proportion and the angles are equal, the two quadrilaterals must be similar.
\end{proof}

\begin{theorem}
\label{thm:halfGenArea}
If the reference quadrilateral is a general quadrilateral and the chosen center is $X_4$, then the central quadrilateral has the same area as the reference quadrilateral
(Figure \ref{fig:halfGenArea}).
\end{theorem}

\begin{figure}[h!t]
\centering
\includegraphics[width=0.4\linewidth]{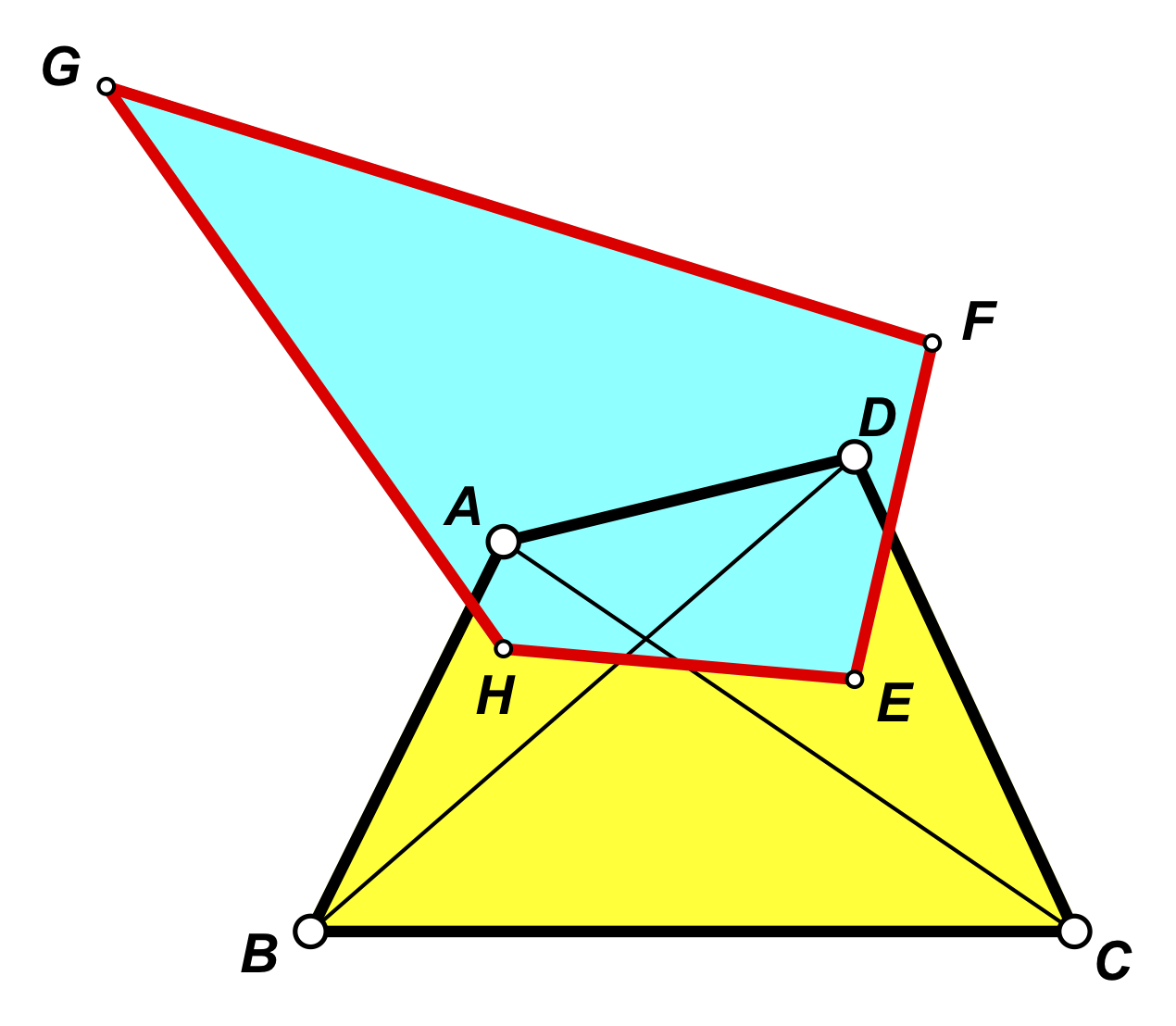}
\caption{orthocenters $\implies$ $[ABCD]=[EFGH]$}
\label{fig:halfGenArea}
\end{figure}

\begin{open}
Is there a purely geometric proof of Theorem \ref{thm:halfGenArea}?
(Figure~\ref{fig:halfGenArea})
\end{open}

\begin{conjecture}
\label{conjecture:genSimilar}
If the reference quadrilateral is a general quadrilateral, $ABCD$, then $X_2$ is the only center
such that the central quadrilateral is similar to the reference quadrilateral for all quadrilaterals $ABCD$.
\end{conjecture}

\begin{conjecture}
\label{conjecture:genArea}
If the reference quadrilateral is a general quadrilateral, $ABCD$, then $X_4$ is the only center
such that the central quadrilateral has the same area as the reference quadrilateral for all quadrilaterals $ABCD$.
\end{conjecture}

\subsection{Cyclic Quadrilaterals}

\begin{theorem}
\label{thm:halfCyclicJapanese}
If the reference quadrilateral is cyclic, then the central quadrilateral is a rectangle
when the chosen center is $X_1$.
\end{theorem}

\begin{proof}
This theorem is the same as Result 1 in the Introduction.
A proof can be found in \cite[p.~133]{Altshiller-Court}.
\end{proof}

\begin{theorem}
\label{thm:genCyclicRect}
If the reference quadrilateral is cyclic, then the central quadrilateral is a rectangle
when the center function is of the form
$$\cos B+\cos C+k\cos A-1$$
where $k$ is a constant.
\end{theorem}

The Bevan point of a triangle has center function $\cos B+\cos C-\cos A-1$.
This gives us the following corollary.

\begin{corollary}
\label{thm:halfCyclicBevan}
Let $ABCD$ be a cyclic quadrilateral.
Let $E$, $F$, $G$, and $H$ be the Bevan points of triangles $\triangle BCD$, $\triangle CDA$, 
$\triangle DAB$,  and $\triangle ABC$, respectively. Then $EFGH$ is a rectangle.
(Figure \ref{fig:halfCyclicBevan})
\end{corollary}

\begin{figure}[h!t]
\centering
\includegraphics[width=0.23\linewidth]{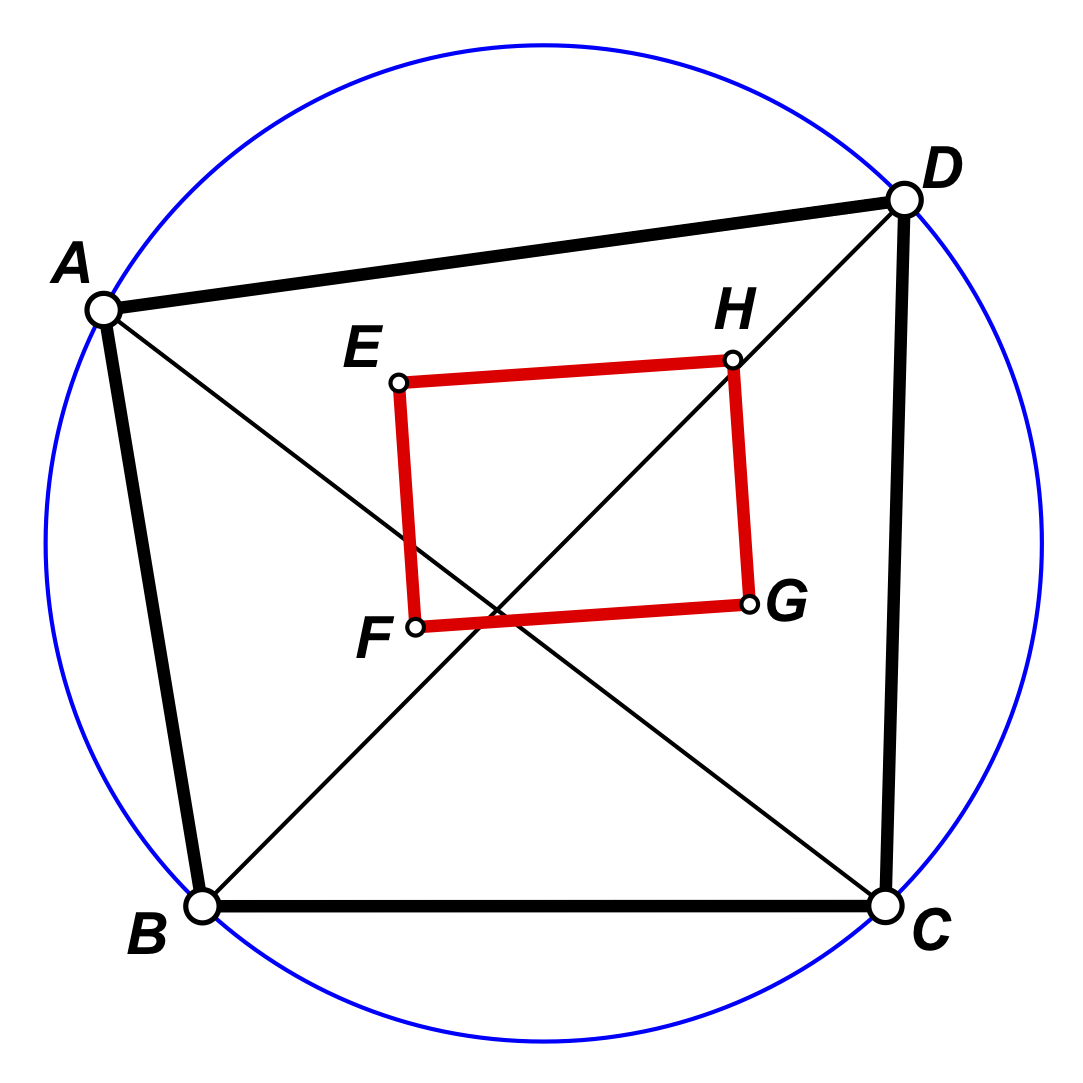}
\caption{Bevan points $\implies$ rectangle}
\label{fig:halfCyclicBevan}
\end{figure}

\begin{open}
Is there a purely geometric proof of Theorem \ref{thm:halfCyclicBevan}?
(Figure~\ref{fig:halfCyclicBevan})
\end{open}

\begin{theorem}
\label{thm:halfCyclic155}
If the reference quadrilateral is cyclic, then the central quadrilateral degenerates to a line segment
when the chosen center is $X_{155}$.
(Figure~\ref{fig:halfCyclic155})
\end{theorem}

\begin{figure}[h!t]
\centering
\includegraphics[width=0.4\linewidth]{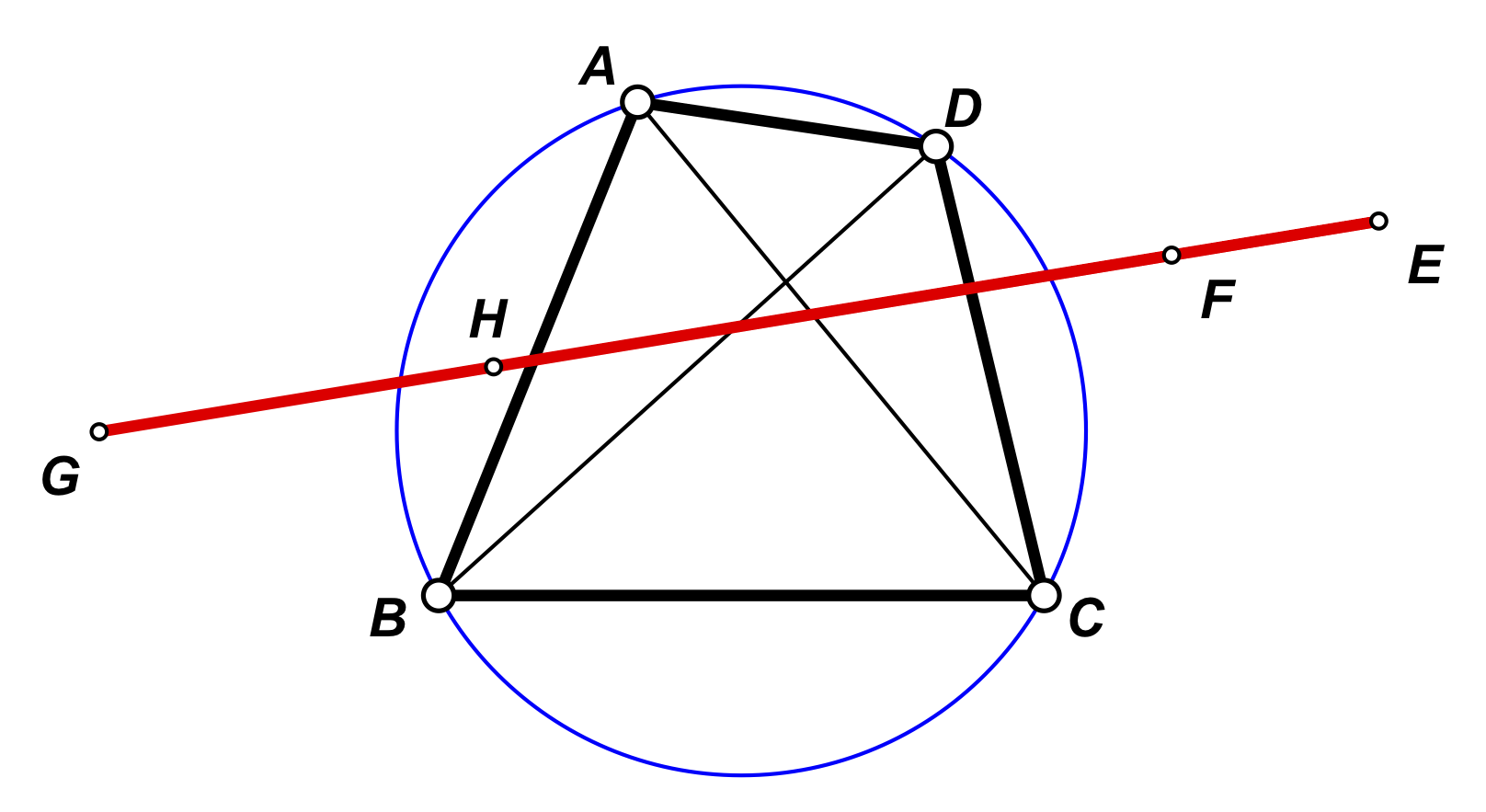}
\caption{$X_{155}$ points $\implies$ line}
\label{fig:halfCyclic155}
\end{figure}

\begin{theorem}
\label{thm:genCyclicTrig}
If the reference quadrilateral is cyclic, then the central quadrilateral is cyclic
when the center function is of the form
$$\cos B\cos C+k\cos A$$
where $k$ is a constant.
\end{theorem}

\begin{theorem}
\label{thm:genCyclicCommon}
If the reference quadrilateral is cyclic, then the central quadrilateral is cyclic for the following common centers:
the centroid, the orthocenter, the nine-point center, the first and second Fermat points, the first and second isodynamic points, the de Longchamps point, the Far-Out point, the Tarry point, and the Steiner point.
\end{theorem}

One case of Theorem \ref{thm:genCyclicCommon} is shown in Figure \ref{fig:halfCyclicX15}.

\begin{figure}[h!t]
\centering
\includegraphics[width=0.23\linewidth]{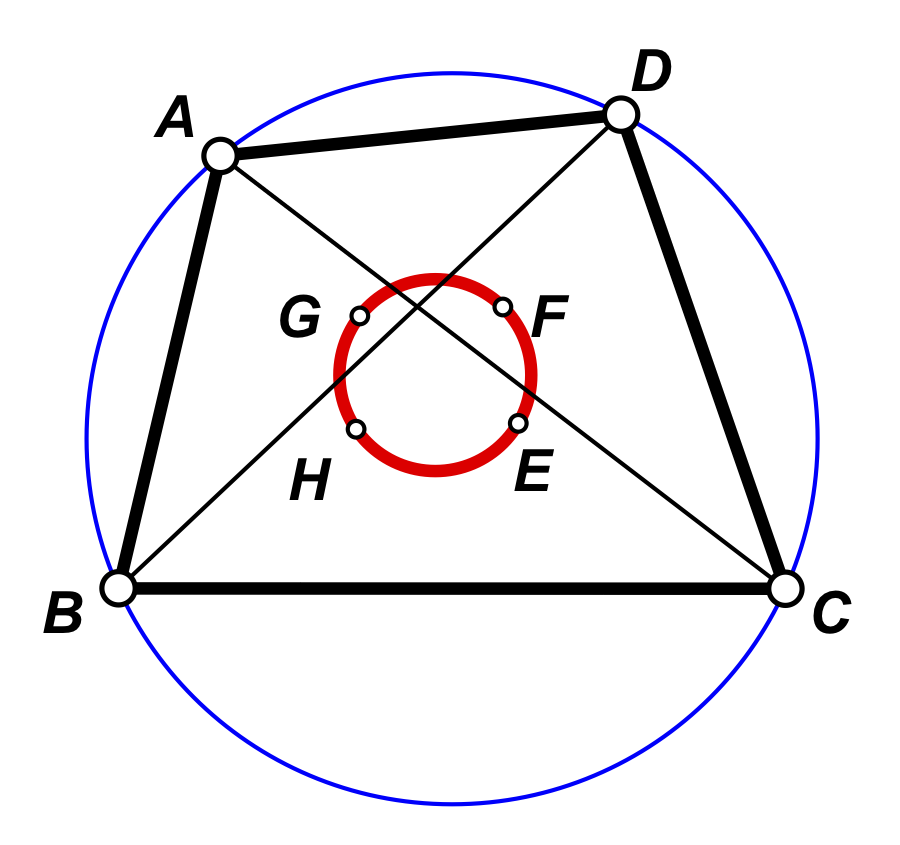}
\caption{1st isodynamic points $\implies$ circle}
\label{fig:halfCyclicX15}
\end{figure}

\textbf{Note.}
When the center is the orthocenter, the central quadrilateral is
congruent to the reference quadrilateral (Figure \ref{fig:halfCyclicX4}).
See \cite[p.~209]{Malcheski} for a proof.

\begin{figure}[h!t]
\centering
\includegraphics[width=0.2\linewidth]{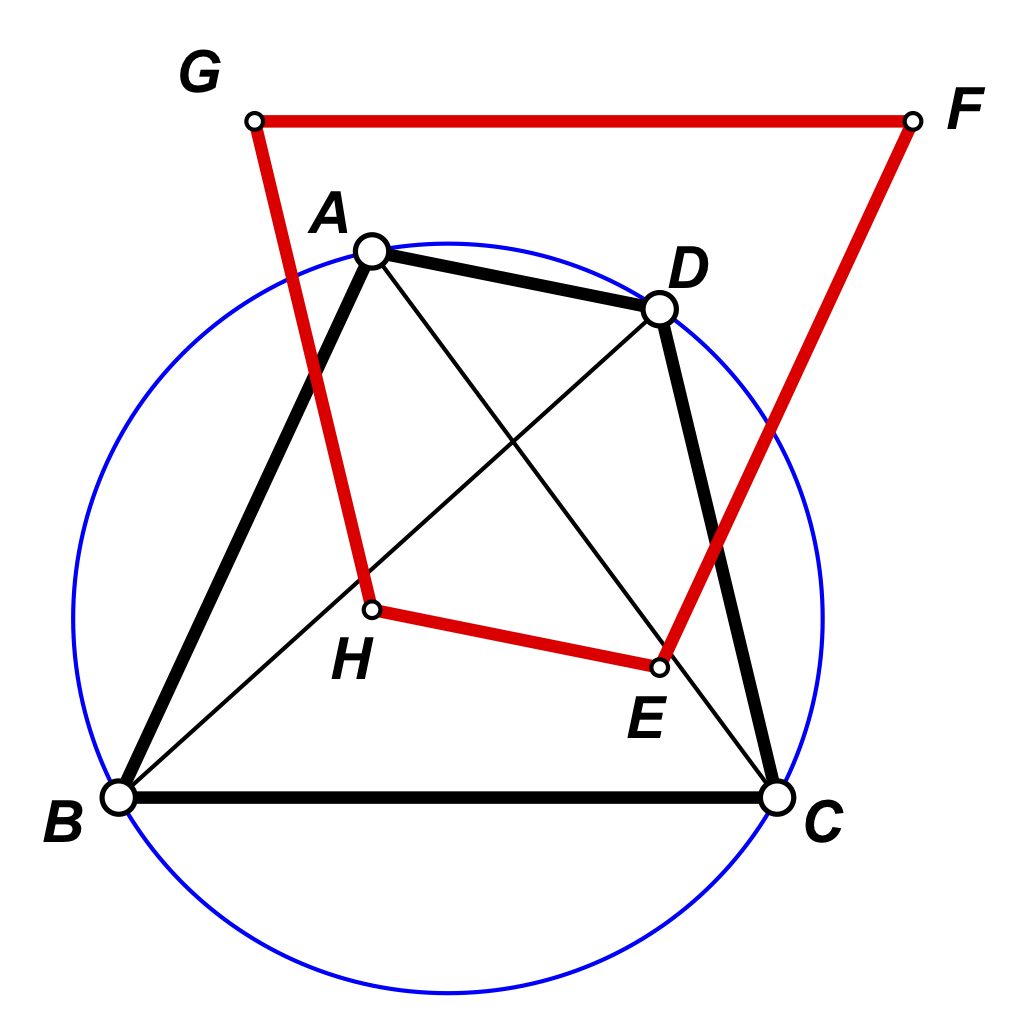}
\caption{orthocenters $\implies$ congruent quadrilaterals}
\label{fig:halfCyclicX4}
\end{figure}

\begin{open}
Are there purely geometric proofs for the results given in Theorem \ref{thm:genCyclicCommon}?
\end{open}

The centroid case follows from Theorem \ref{thm:genSimilar}.

Although Theorem \ref{thm:genCyclicCommon} can be proven using barycentric coordinates,
it is also possible to prove the results using complex numbers.
We give an example using $X_{20}$, the de Longchamps point.
We start with a Lemma from \cite[p.~162]{Malcheski}.

\begin{lemma}
The four complex numbers $p$, $q$, $r$, and $s$ are the consecutive vertices
of a cyclic quadrilateral (or are collinear) in the complex plane if and only~if
%the cross ratio
$$\frac{(p-s)(r-q)}{(p-q)(r-s)}\in \mathbb {R}.$$
\end{lemma}

\begin{theorem}
\label{thm:genCyclicDeLongchamps}
If the reference quadrilateral is cyclic, then the central quadrilateral is cyclic
if the chosen center is the de Longchamps point.
\end{theorem}

\begin{proof}

Assume that the circumcircle $O$ of the quadrilateral $ABCD$ is the origin of the complex plane. 
Let $a,b,c,d$ be the complex coordinates (affixes) of the vertices $A$, $B$, $C$, $D$. 
Since $ABCD$ is cyclic we have that the cross ratio
\begin{equation*}
   (a,b,c,d)=\frac{(a-b) (c-d)}{(a-d) (c-b)}
\end{equation*}
is a real number.
The complex coordinates $f,g,h,i$ of the orthocenters $F,G,H,I$ of the half triangles $\triangle BCD$, $\triangle ACD$, $\triangle ABD$, $\triangle ABC$ are (\cite[p.~122]{Malcheski})
\begin{align*}
	f&=b+c+d\\
	g&=a+c+d\\
	h&=a+b+d\\
	i&=a+b+c.
\end{align*}
Since $X_{20}$, the de Longchamps point, is the reflection of $X_4$ in $X_3=0$, we have
\begin{align*}
	f&=-(b+c+d)\\
	g&=-(a+c+d)\\
	h&=-(a+b+d)\\
	i&=-(a+b+c).
\end{align*}
A little calculation shows that
  \begin{equation*}
   (f,g,h,i)=(a,b,c,d) \in \mathbb {R},
  \end{equation*}
so $f,g,h,i$ are concyclic.	
\end{proof}

\begin{open}
Characterize those centers that yield cyclic central quadrilaterals when the reference quadrilateral is cyclic.
\end{open}

\subsection{EqualProdOpp Quadrilaterals}

\begin{theorem}
\label{thm:halfEqualProdOpp}
If the reference quadrilateral is equalProdOpp, then the central quadrilateral is also equalProdOpp for the following centers: $X_3$, $X_5$, $X_{15}$.
\end{theorem}

This theorem is illustrated in Figure \ref{fig:halfEqualProdOpp} using 1st isodynamic points ($X_{15}$).

\begin{figure}[h!t]
\centering
\includegraphics[width=0.3\linewidth]{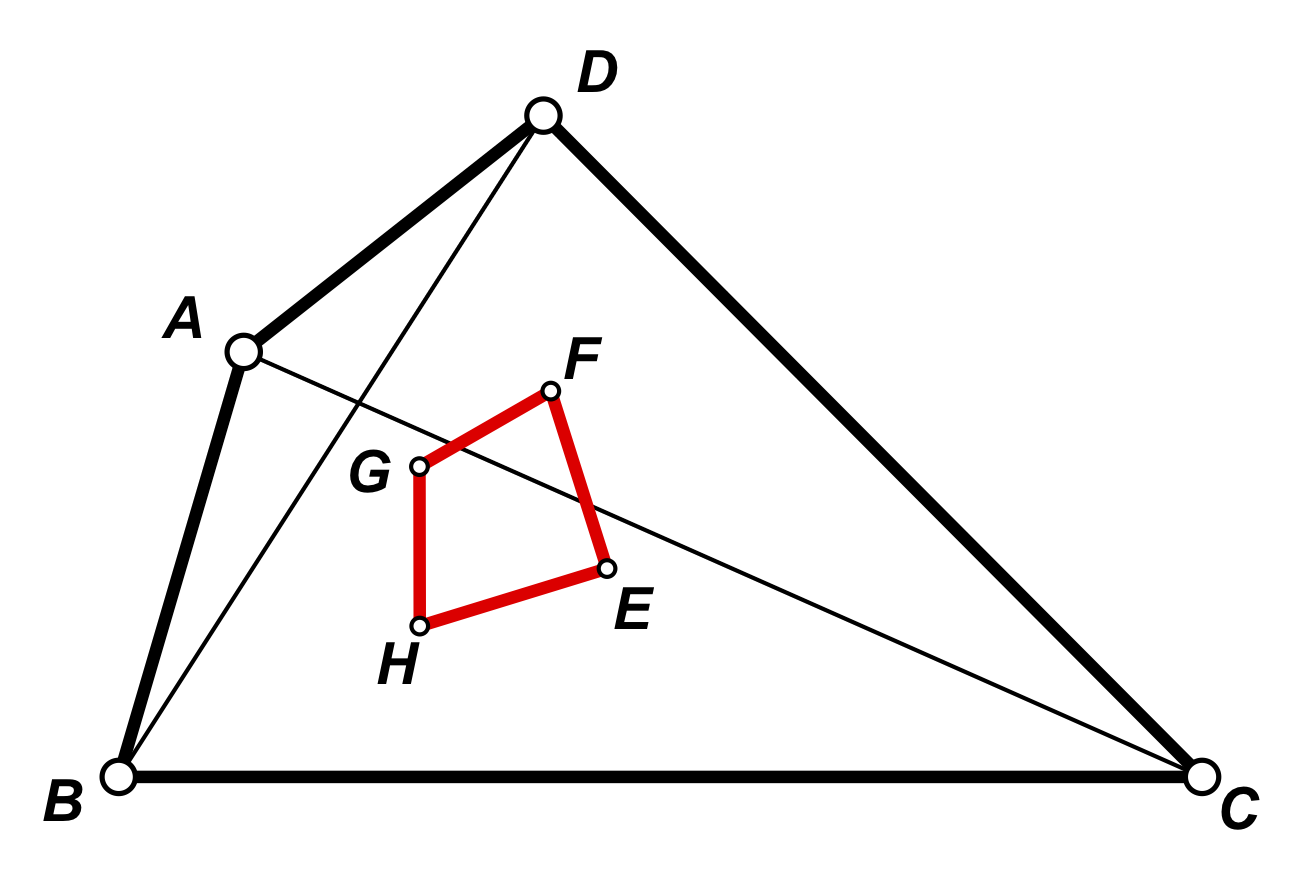}
\caption{$X_{15}$ centers, $AD\cdot BC=AB\cdot CD \implies EF\cdot GH=FG\cdot HE$}
\label{fig:halfEqualProdOpp}
\end{figure}

\begin{open}
Are there purely geometric proofs for the results given in Theorem \ref{thm:halfEqualProdOpp}?
\end{open}

\subsection{Orthodiagonal Quadrilaterals}

\begin{theorem}
\label{thm:halfOrtho}
If the reference quadrilateral is orthodiagonal, the central quadrilateral is orthodiagonal whenever the chosen center has center function of the form
$$\cos B\cos C+k\cos A$$
for some constant $k$.
In this case, the diagonals of $EFGH$ are parallel to the diagonals of $ABCD$.
\end{theorem}

\subsection{Trapezoids}

\begin{theorem}
\label{thm:halfTrap}
If the reference quadrilateral is a trapezoid, the central quadrilateral is a trapezoid whenever the chosen center has center function of the form
$$\cos B\cos C+k\cos A$$
for some constant $k$.
\end{theorem}

This theorem is illustrated in Figure \ref{fig:halfDeLongChamps} using de Longchamps points ($X_{20}$),
which have center function $\cos B\cos C-\cos A$.

\begin{figure}[h!t]
\centering
\includegraphics[width=0.25\linewidth]{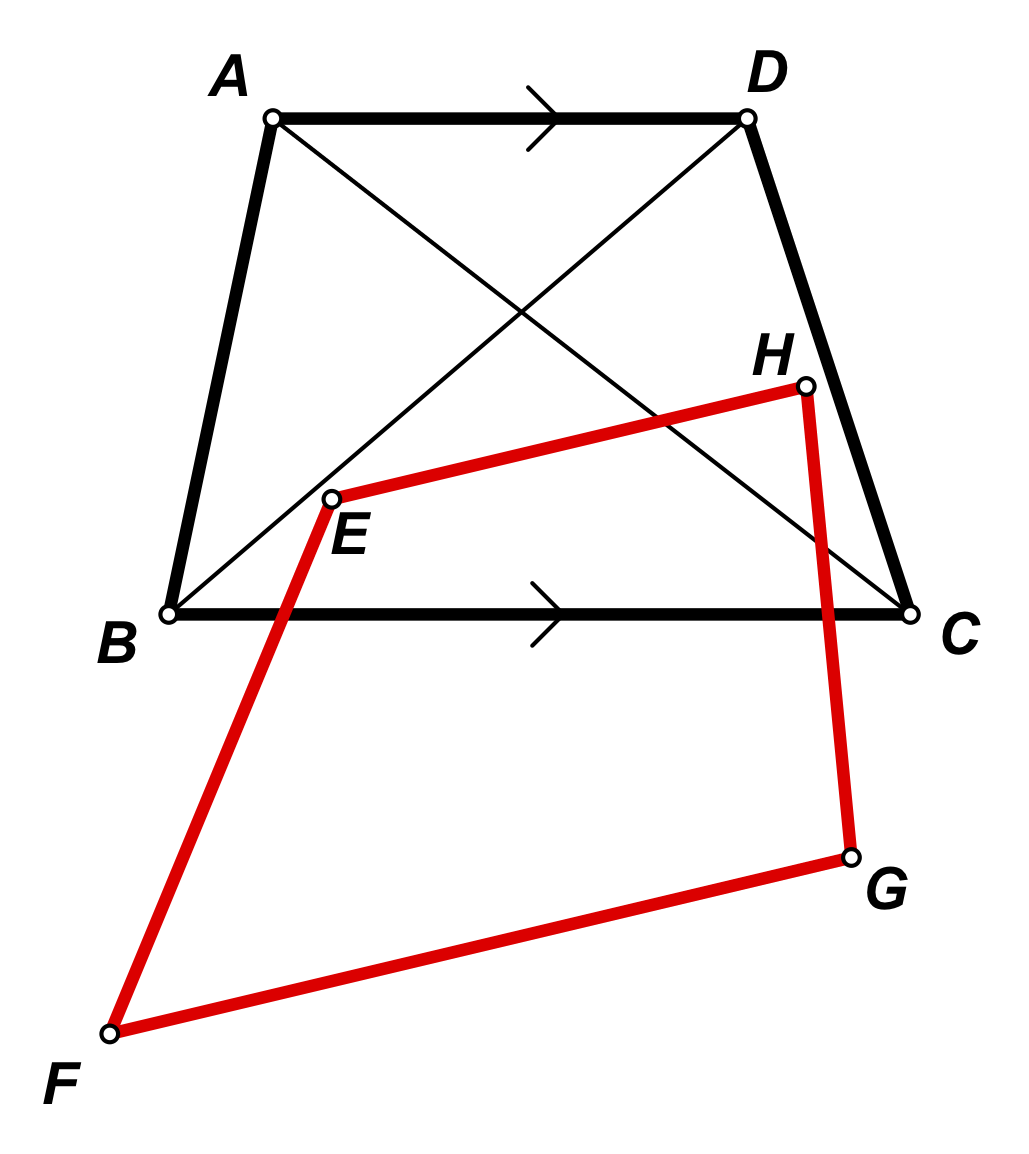}
\caption{de Longchamps points: $AD\parallel BC\implies EH\parallel FG$}
\label{fig:halfDeLongChamps}
\end{figure}

\subsection{Tangential Quadrilaterals}

\begin{theorem}
\label{thm:halfTangential}
If the reference quadrilateral is tangential, then the central quadrilateral is also tangential for the following centers: $X_3$, $X_5$.
\end{theorem}

This theorem is illustrated in Figure \ref{fig:halfTangential} using circumcenters.

\begin{figure}[h!t]
\centering
\includegraphics[width=0.3\linewidth]{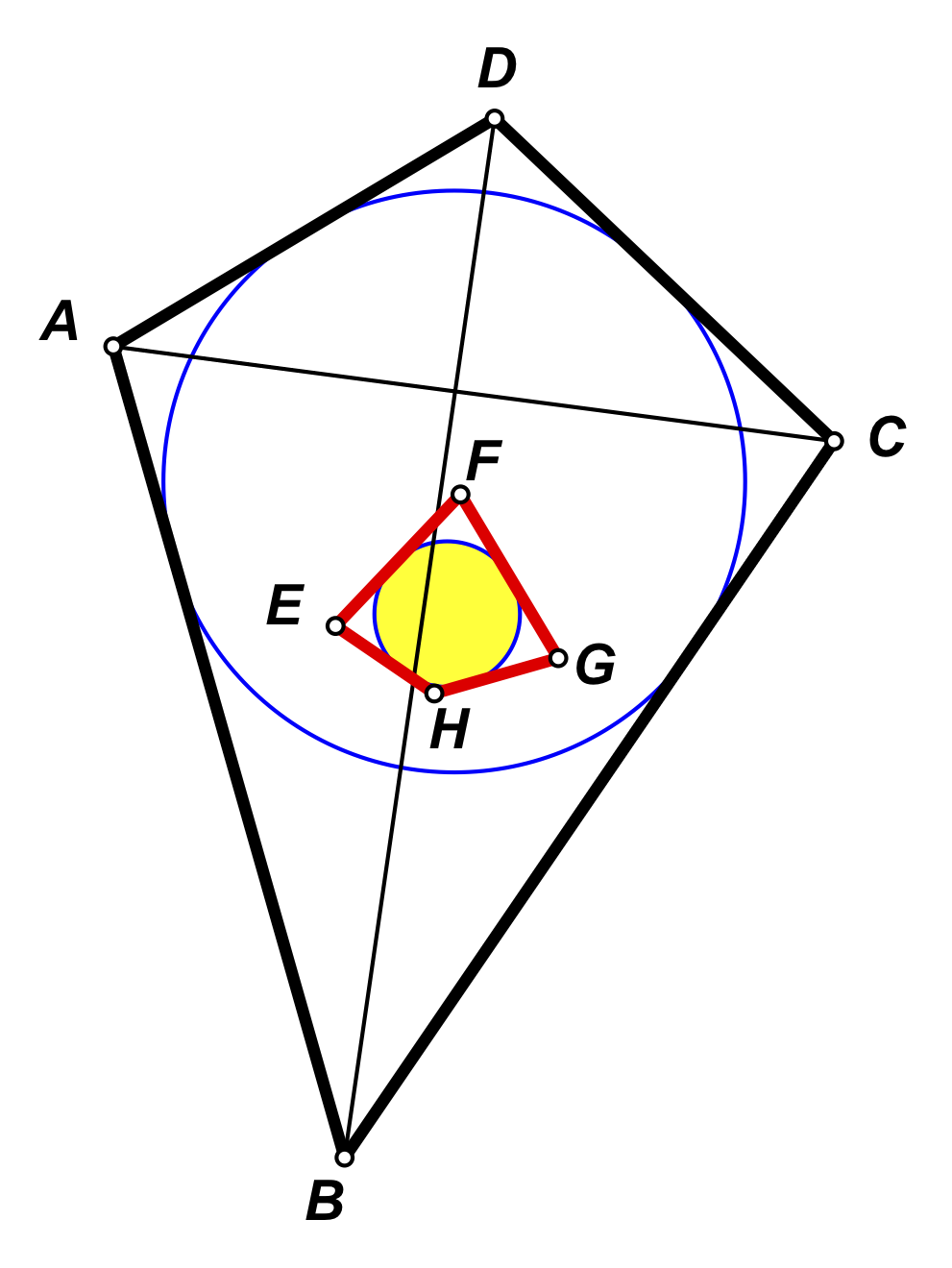}
\caption{Tangential quadrilateral: circumcenters $\implies$ tangential }
\label{fig:halfTangential}
\end{figure}
\begin{open}
Are there purely geometric proofs for the results given in Theorem \ref{thm:halfTangential}?
\end{open}

\subsection{Equidiagonal Orthodiagonal Quadrilaterals}

\begin{theorem}
\label{thm:halfEquiOrtho}
If the reference quadrilateral is equidiagonal and orthodiagonal, then the central quadrilateral is also equidiagonal and orthodiagonal for the following centers: $X_{51}$, $X_{373}$.
\end{theorem}

\begin{theorem}
\label{thm:halfEquiOrthoLine}
If the reference quadrilateral is equidiagonal and orthodiagonal, then the central quadrilateral degenerates to a line segment when the center is $X_{642}$.
\end{theorem}

\begin{theorem}
\label{thm:halfEquiOrthoOrtho}
If the reference quadrilateral is equidiagonal and orthodiagonal, then the central quadrilateral is orthodiagonal (but not equidiagonal) when the center is the inner Vecten point, $X_{486}$
(Figure \ref{fig:halfVecten}).
\end{theorem}

\begin{figure}[h!t]
\centering
\includegraphics[width=0.5\linewidth]{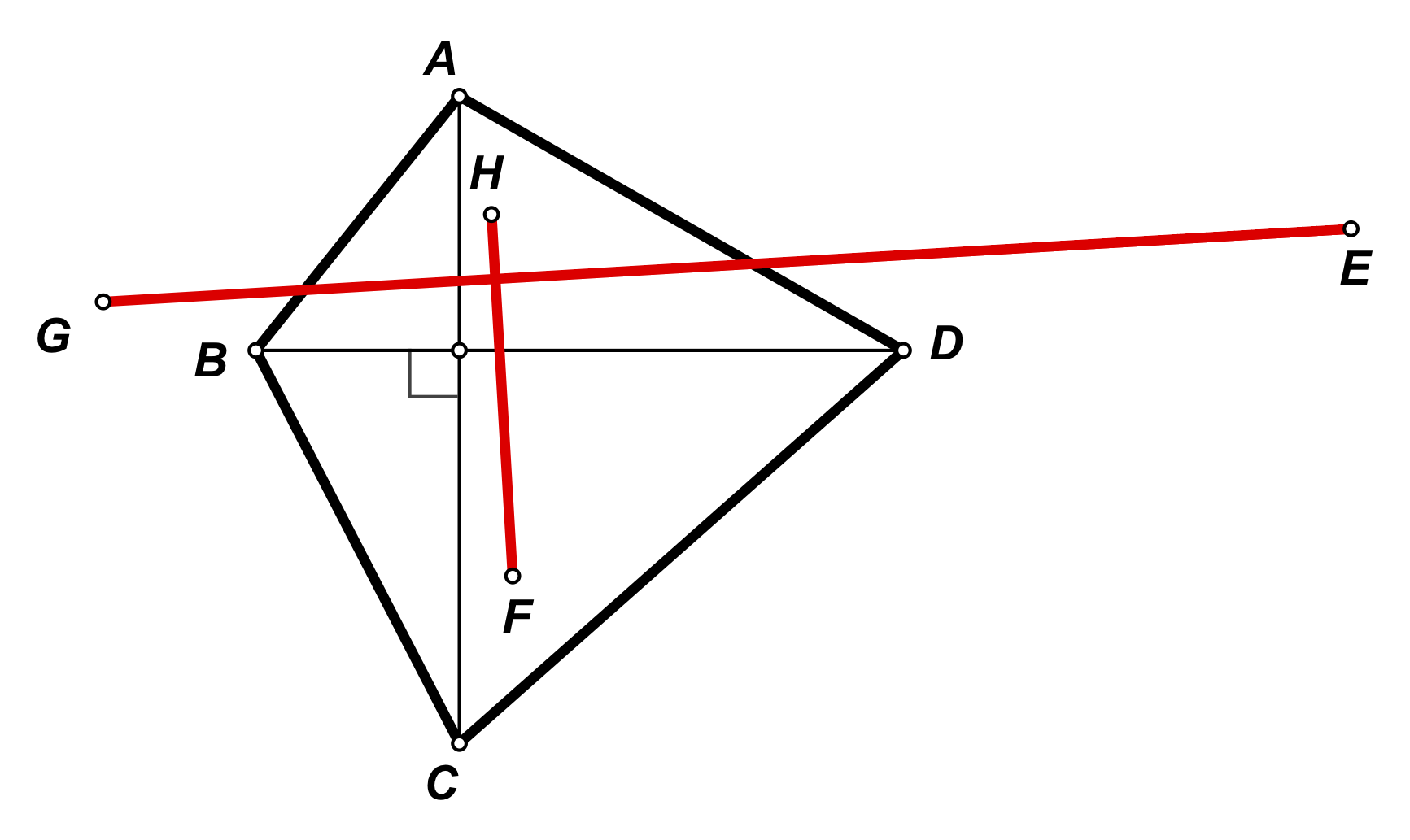}
\caption{inner Vecten pts., $AC=BD, AC\perp BD\implies FH\perp EG$}
\label{fig:halfVecten}
\end{figure}

\subsection{Isosceles Trapezoids}

\begin{theorem}
\label{thm:halfIsoTrap}
If the reference quadrilateral is an isosceles trapezoid, then the central quadrilateral is also an isosceles trapezoid
(Figure \ref{fig:halfTrap}).
\end{theorem}

\begin{figure}[h!t]
\centering
\includegraphics[width=0.45\linewidth]{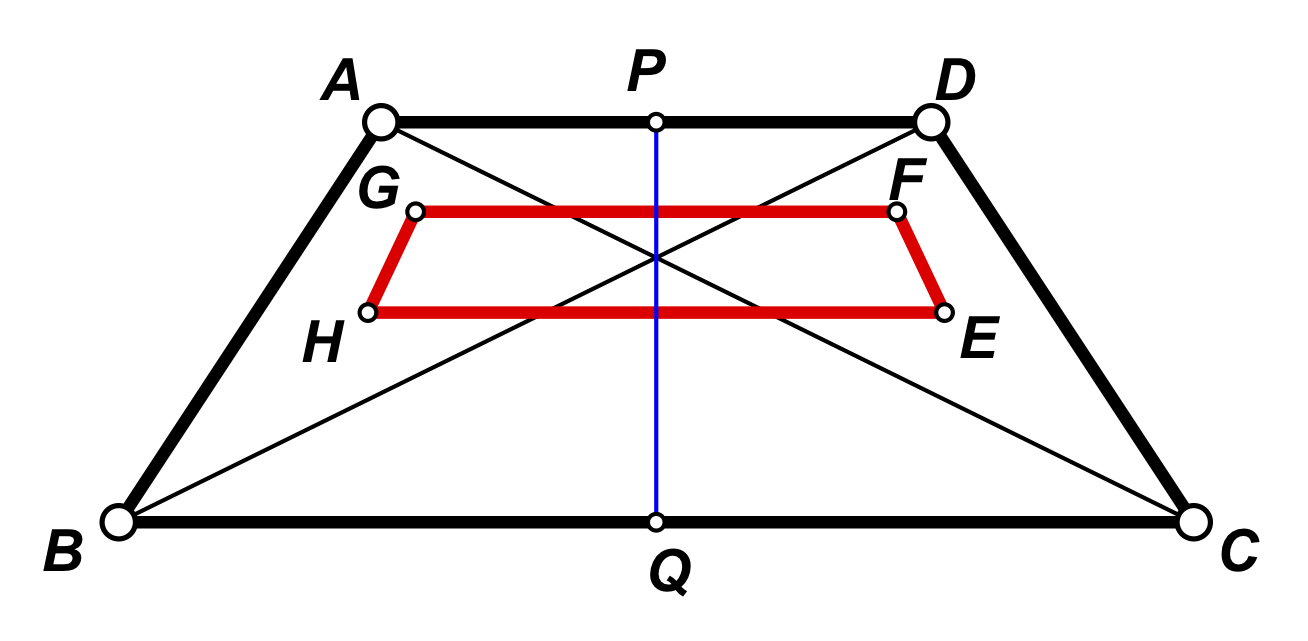}
\caption{isosceles trapezoid $\implies$ isosceles trapezoid}
\label{fig:halfTrap}
\end{figure}

\begin{proof}
Let $P$ and $Q$ be the midpoints of $AD$ and $BC$.
Note that triangles $\triangle ABC$ and $\triangle DCB$ are congruent.
One is the reflection of the other about $PQ$.
Under this reflection, $H$ maps to $E$, so $HE\perp PQ$ which means $HE\parallel BC$.
Similarly, $FG\parallel BC$. Thus, $EFGH$ is a trapezoid. Segment $EF$ maps to $HG$
under this reflection. Since reflection preserves distances, $EF=HG$.
Therefore, $EFGH$ is an isosceles trapezoid.
\end{proof}

\begin{theorem}
\label{thm:halfIsoTrapRect}
If the reference quadrilateral is an isosceles trapezoid, then the central quadrilateral is a rectangle for the following centers: $X_{40}$, $X_{165}$.
\end{theorem}

\subsection{Cyclic Orthodiagonal Quadrilaterals}

\begin{theorem}
\label{thm:halfCyclicOrthoLine}
If the reference quadrilateral is cyclic and orthodiagonal, then the central quadrilateral degenerates to a line segment when the center is $X_{63}$. (Figure \ref{fig:halfCyclicOrthoLine})
\end{theorem}

\begin{figure}[h!t]
\centering
\includegraphics[width=0.4\linewidth]{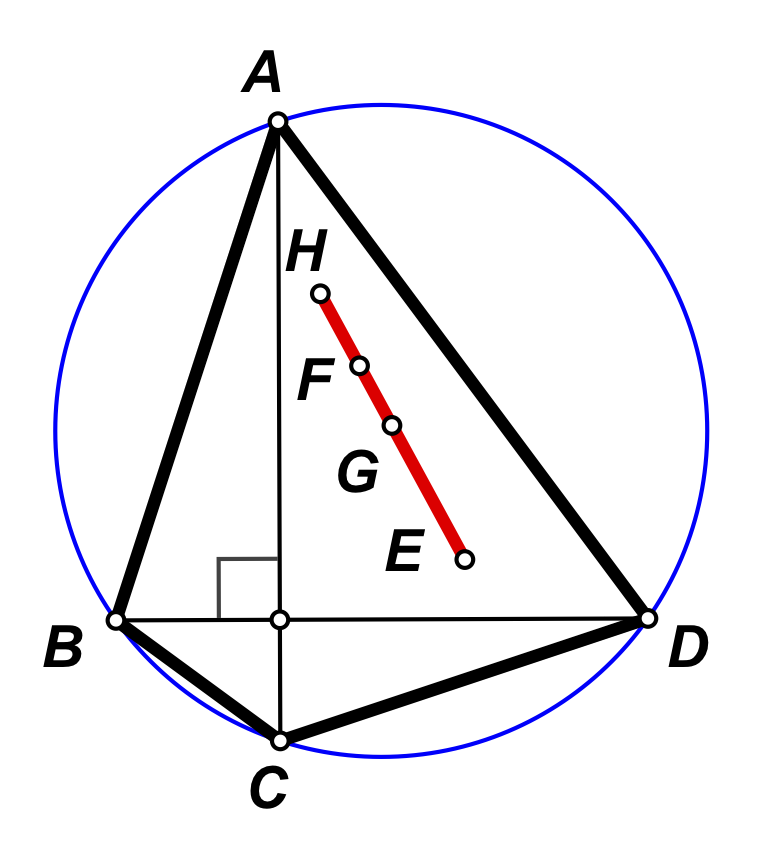}
\caption{$X_{63}$ points $\implies$ line}
\label{fig:halfCyclicOrthoLine}
\end{figure}

\subsection{Hjelmslev Quadrilaterals}

\begin{theorem}
\label{thm:halfHjPara}
If the reference quadrilateral is Hjelmslev, then the central quadrilateral is a parallelogram
when the center is $X_{53}$.
\end{theorem}

\begin{theorem}
\label{thm:halfHjLine}
If the reference quadrilateral is Hjelmslev, then the central quadrilateral degenerates to a
line segment when the center is $X_{97}$.
\end{theorem}

\subsection{Kites}

\begin{theorem}
\label{thm:halfKite}
If the reference quadrilateral is a kite, then the central quadrilateral is a (not necessarily convex) kite (Figure \ref{fig:halfKite}).
\end{theorem}

\begin{figure}[h!t]
\centering
\includegraphics[width=0.25\linewidth]{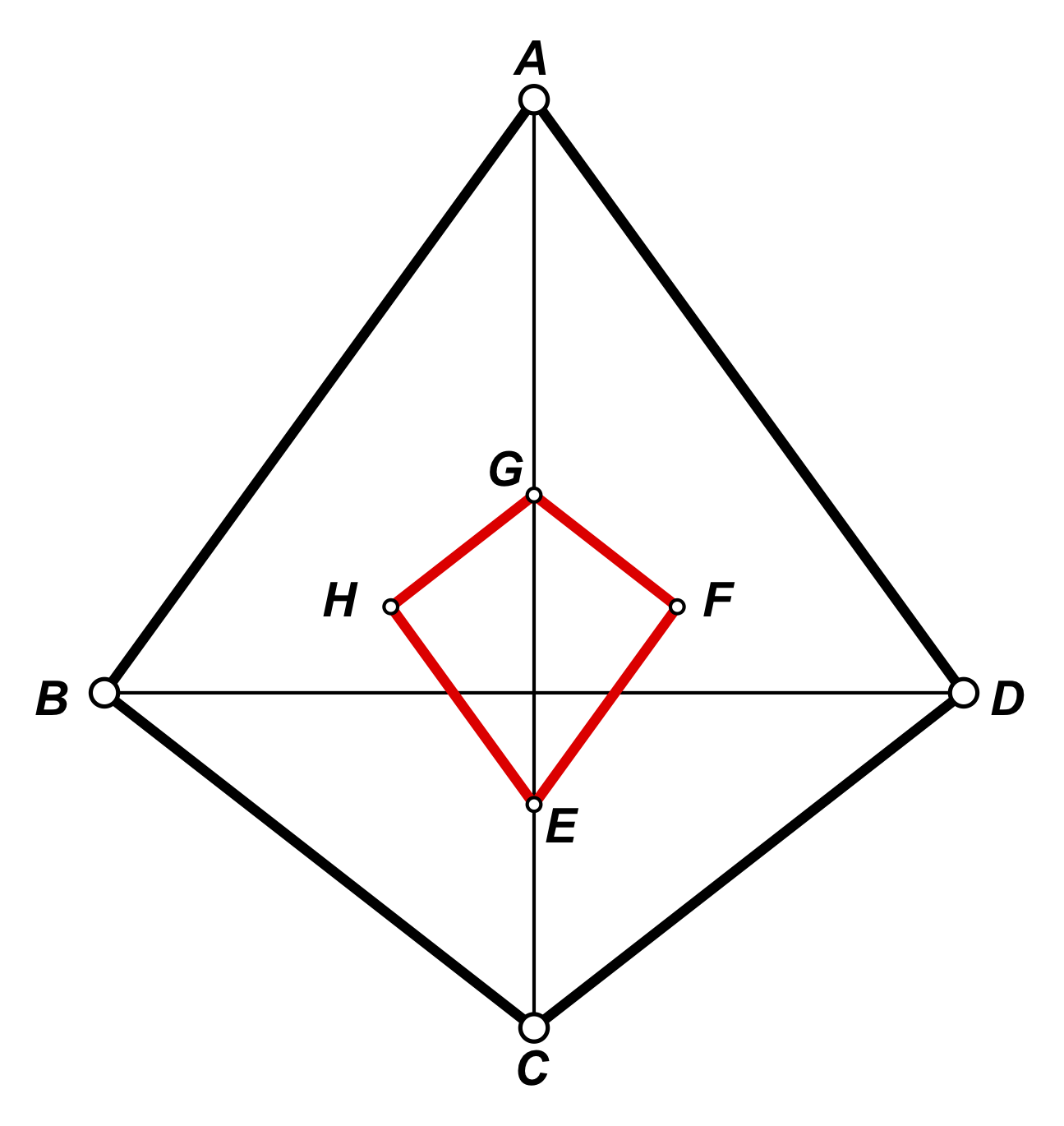}
\caption{kite $\implies$ kite}
\label{fig:halfKite}
\end{figure}

\begin{proof}
Label the kite so that $AB=AD$ and $CB=CD$ (Figure \ref{fig:halfKiteNagel}).
\begin{figure}[h!t]
\centering
\includegraphics[width=0.45\linewidth]{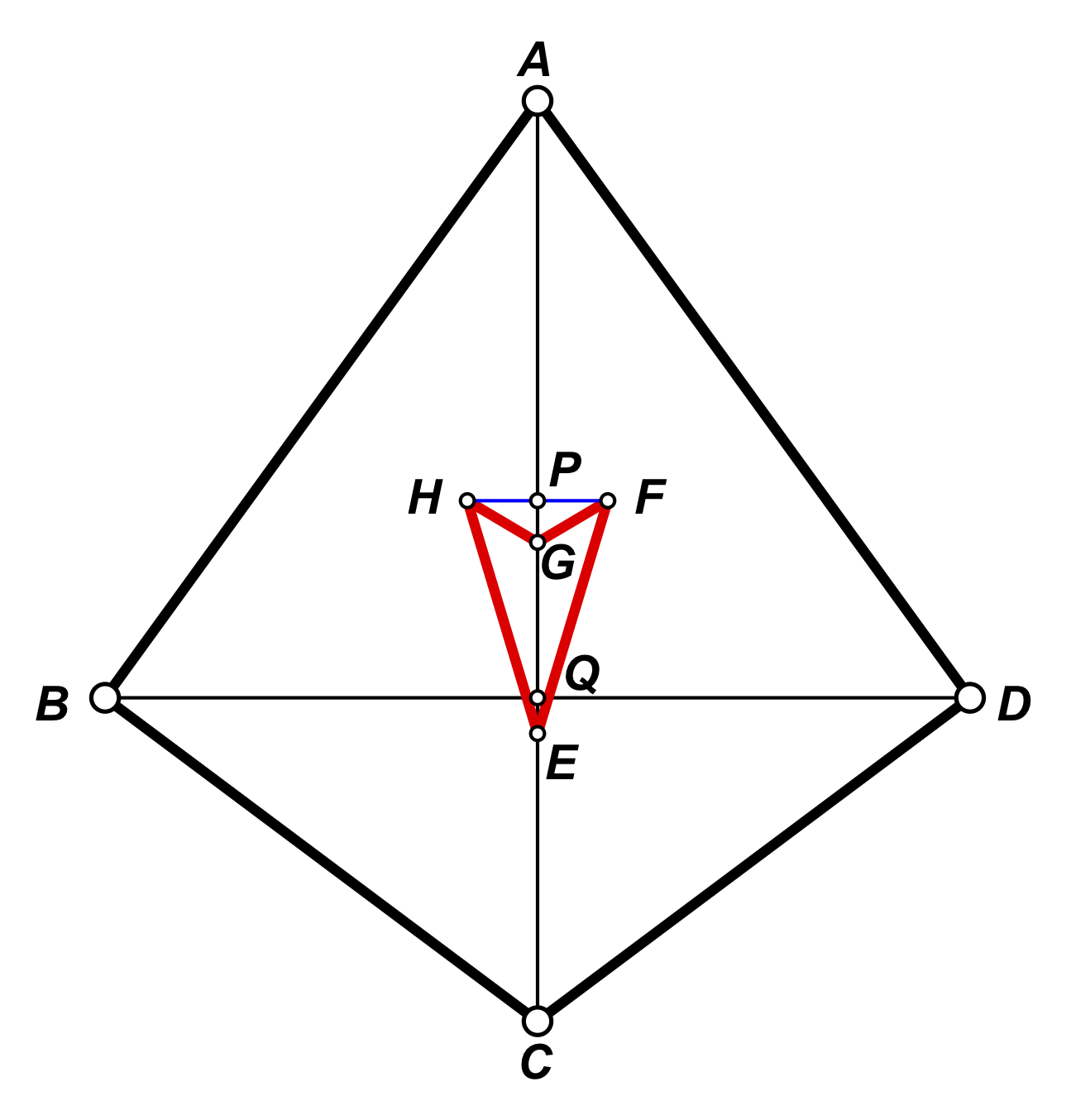}
\caption{central kite need not be convex}
\label{fig:halfKiteNagel}
\end{figure}

Let $AC$ meet $BD$ at $Q$.
The diagonals of a kite are perpendicular, so $AC\perp BD$.
A center of an isosceles triangle must lie on the altitude to the base, so centers $E$ and $G$
lie on $AC$.
Let $FH$ meet $AC$ at $P$.
Note that $P$ can lie anywhere on the line through $A$ and $C$, and is not necessarily
located as shown in Figure \ref{fig:halfKiteNagel}.
Note also that quadrilateral $EFGH$ could be convex as shown in Figure \ref{fig:halfKite}
or nonconvex as shown in Figure \ref{fig:halfKiteNagel}.

By the Isosceles Triangle Lemma (Lemma \ref{lemma:isosceles}),
$FH\perp AC$ and $PH=PF$. Thus $\triangle PEH\cong \triangle PEF$, so $EH=EF$.
Similarly, $\triangle PGH\cong \triangle PGF$, so $GH=GF$.
Thus, quadrilateral $EFGH$ is a kite.
\end{proof}

\subsection{Parallelograms}

\begin{theorem}
\label{thm:halfPara}
If the reference quadrilateral is a parallelogram, then the central quadrilateral is also a parallelogram (Figure \ref{fig:halfPara}).
\end{theorem}

\begin{figure}[h!t]
\centering
\includegraphics[width=0.4\linewidth]{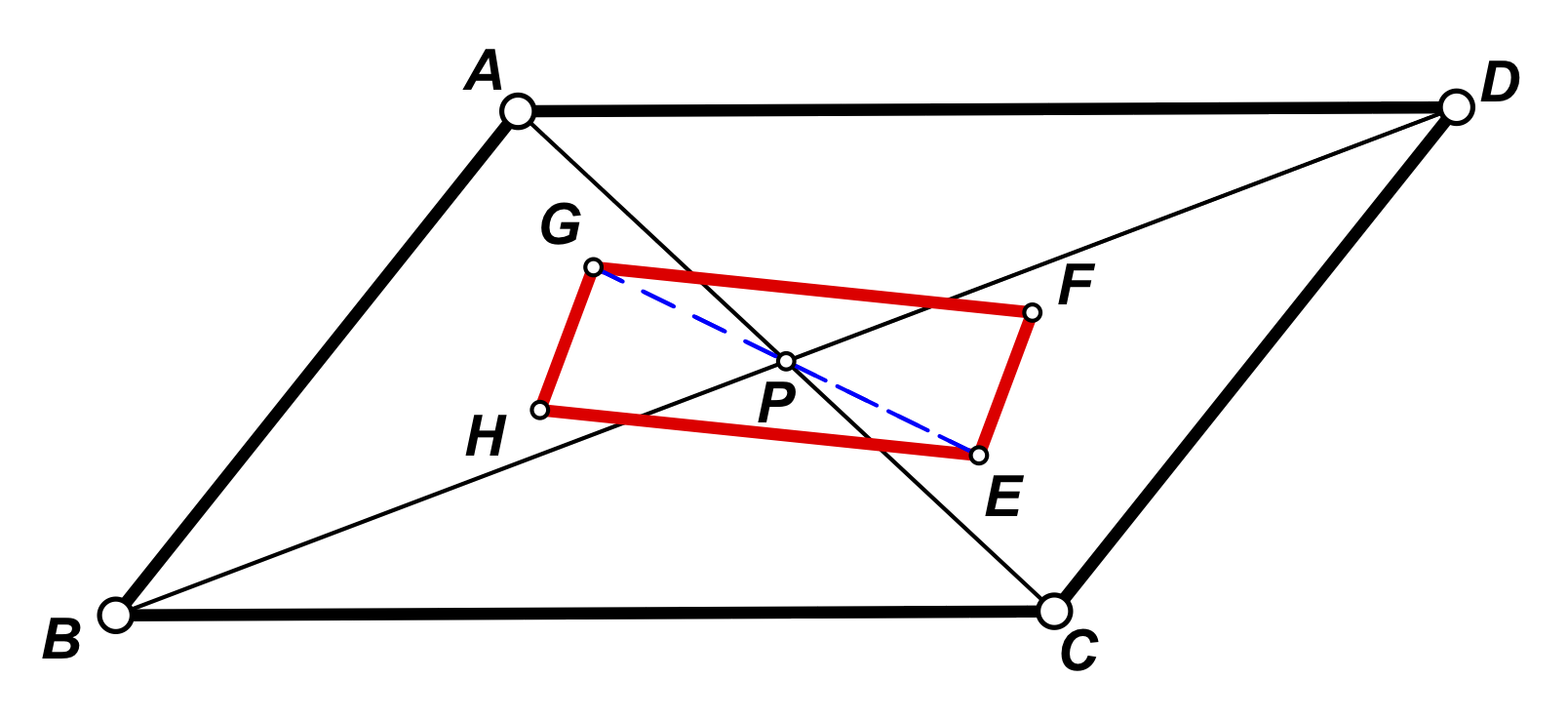}
\caption{parallelogram $\implies$ parallelogram}
\label{fig:halfPara}
\end{figure}

\begin{proof}
Let $P$ be the intersection of the diagonals of the reference quadrilateral $ABCD$.
Note that triangles $ABD$ and $CDB$ are congruent.
Under the congruence transformation that maps $\triangle ABD$ into $\triangle CDB$, $E$ will get mapped into $G$. Since
$P$ is the center of the congruence transformation, this means $EG$ will pass through
$P$. Congruence preserves lengths, so $EP=PG$ (Figure \ref{fig:halfPara}).
Similarly, $HP=PF$.
Since the diagonals of quadrilateral $EFGH$ bisect each other, this means that quadrilateral $EFGH$
is a parallelogram.
\end{proof}

\begin{theorem}
\label{thm:halfParaRhombus}
If the reference quadrilateral is a parallelogram, then the central quadrilateral is a rhombus
when the center is $X_{10}$, $X_{639}$, $X_{640}$, $X_{641}$, or $X_{642}$.
(Figure \ref{fig:halfParaRhombus} shows the case for $X_{10}$)
\end{theorem}

\begin{figure}[h!t]
\centering
\includegraphics[width=0.35\linewidth]{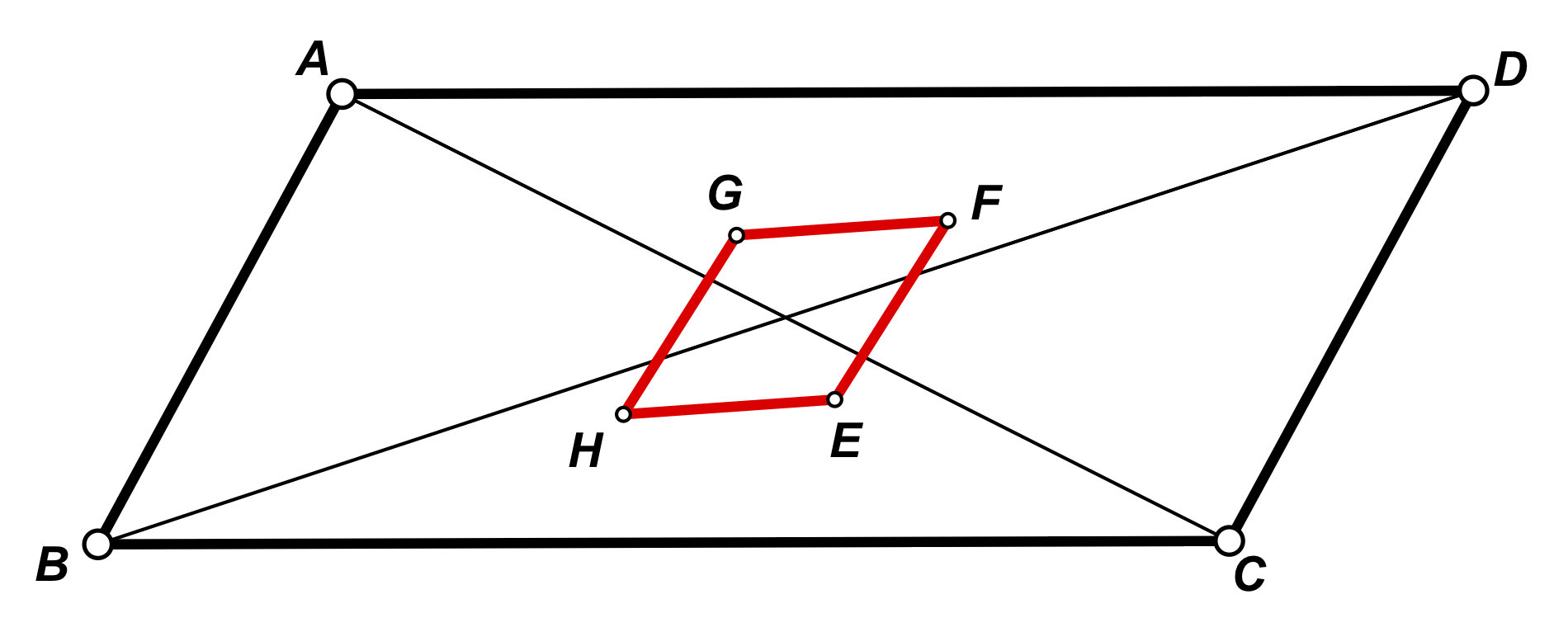}
\caption{parallelogram and $X_{10}$ points $\implies$ rhombus}
\label{fig:halfParaRhombus}
\end{figure}

\begin{open}
Is there a purely geometrical proof to the special case of Theorem \ref{thm:halfParaRhombus}
shown in Figure \ref{fig:halfParaRhombus} where the chosen center is the Spieker center?
% ($X_{10}$)?
\end{open}

\subsection{Rhombi}

\begin{theorem}
\label{thm:halfRhombus}
If the reference quadrilateral is a rhombus, then the central quadrilateral is also a rhombus
(Figure \ref{fig:halfRhombus}).
\end{theorem}

\begin{figure}[h!t]
\centering
\includegraphics[width=0.4\linewidth]{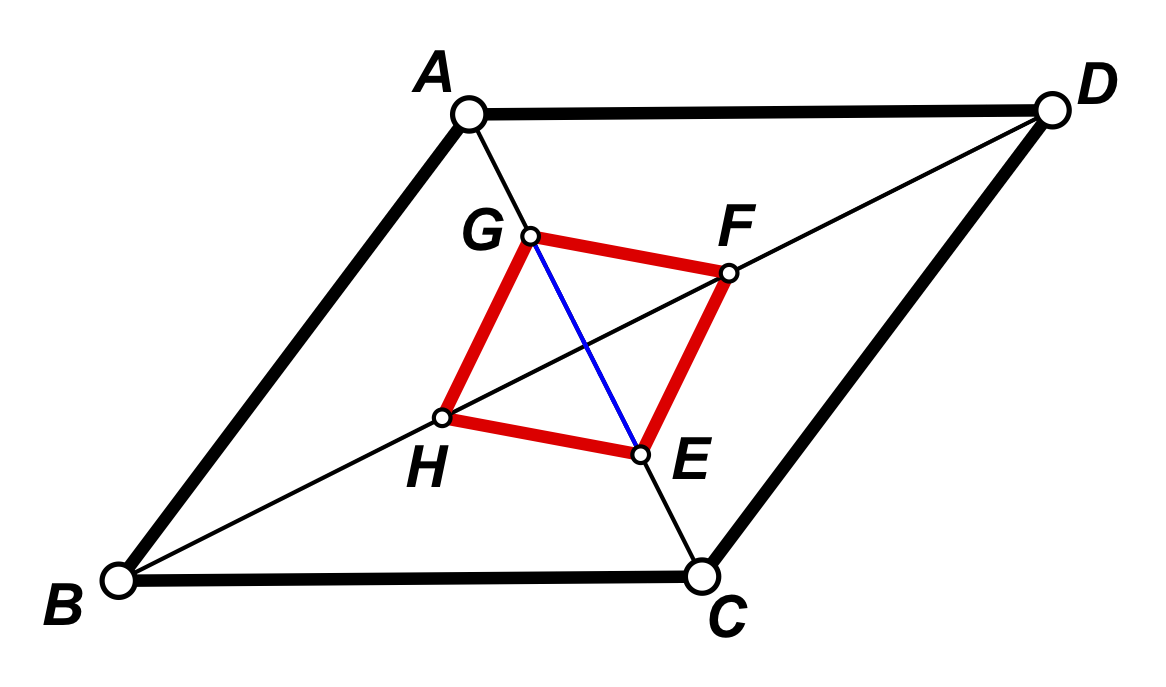}
\caption{rhombus $\implies$ rhombus}
\label{fig:halfRhombus}
\end{figure}

\begin{proof}
A center of an isosceles triangle must lie on the altitude to the base.
Thus, $E$ and $G$ lie on $AC$ and $F$ and $H$ lie on $BD$.
Thus, the diagonals of $EFGH$ lie along the diagonals of $ABCD$.
The diagonals of $ABCD$ are perpendicular, so the diagonals of $EFGH$ are also
perpendicular. This, combined with Theorem \ref{thm:halfPara}, implies that $EFGH$ is a rhombus.
\end{proof}

\begin{theorem}
\label{thm:halfRhombusSquare}
If the reference quadrilateral is a rhombus, then the central quadrilateral is a square
when the center is the Weill point ($X_{354}$).
\end{theorem}

\subsection{Rectangles}

\begin{theorem}
\label{thm:halfRect}
If the reference quadrilateral is a rectangle, then the central quadrilateral is a rectangle
(Figure \ref{fig:halfRect}).
\end{theorem}

\begin{figure}[h!t]
\centering
\includegraphics[width=0.3\linewidth]{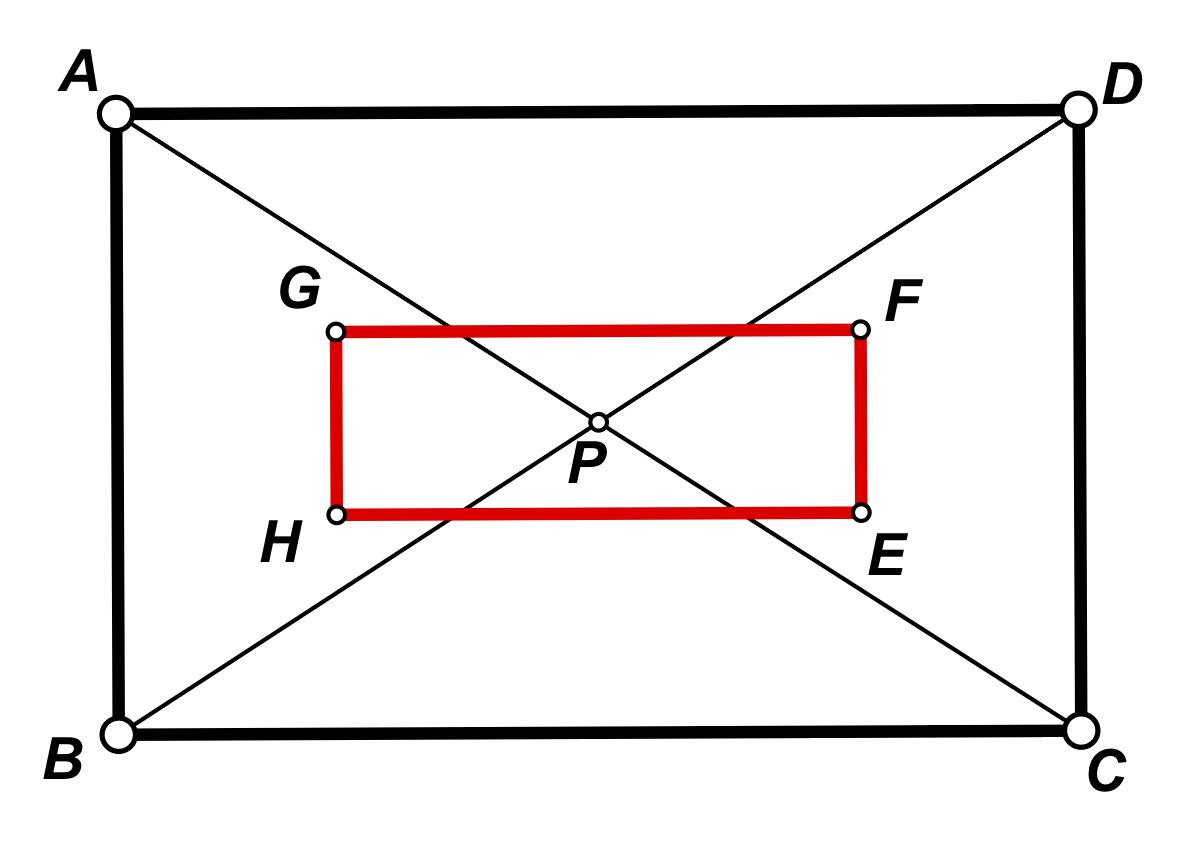}
\caption{rectangle $\implies$ rectangle}
\label{fig:halfRect}
\end{figure}

\begin{proof}
Note that triangles $\triangle ABC$ and $\triangle DCB$ are congruent.
Under the reflection that maps $\triangle ABC$ into $\triangle DCB$, $H$ will get mapped into $E$
and $HE$ will be parallel to $BC$.
Similarly, $GF\parallel AD$, $GH\parallel AB$, and $FE\parallel DC$.
Thus, $EFGH$ is a rectangle.
\end{proof}

\subsection{Squares}

\begin{theorem}
\label{thm:halfSquare}
If the reference quadrilateral is a square, then the central quadrilateral is a square
(Figure \ref{fig:halfSquare}).
\end{theorem}

\begin{figure}[h!t]
\centering
\includegraphics[width=0.25\linewidth]{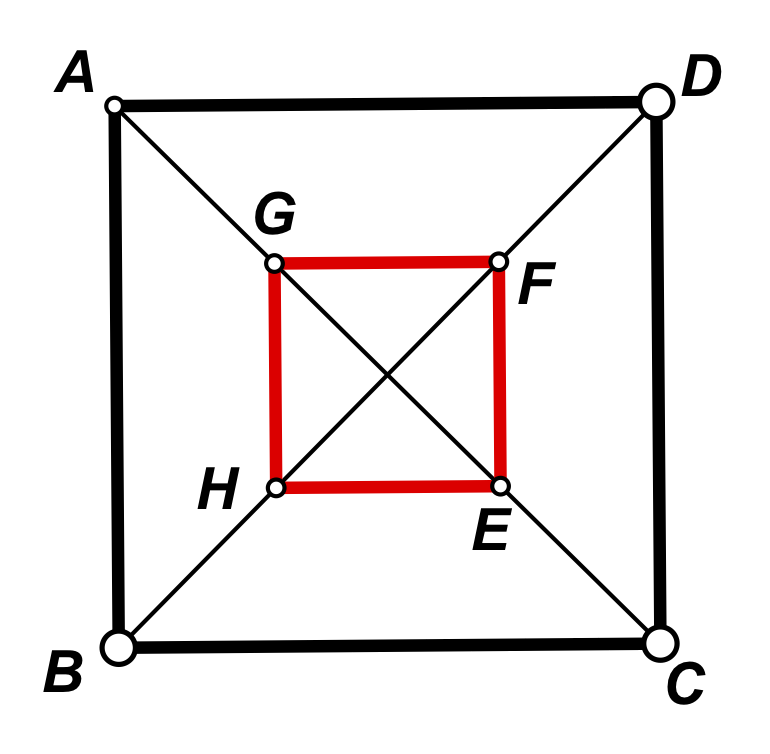}
\caption{square $\implies$ square}
\label{fig:halfSquare}
\end{figure}

\begin{proof}
By Theorem \ref{thm:halfPara}, $EFGH$ must be a parallelogram.
By Theorem \ref{thm:halfKite}, $EFGH$ must be a kite.
But any parallelogram that is also a kite must be a square.
\end{proof}

\section{Areas for future research}
\label{section:research}

There are many avenues for future investigation.

\textbf{Generalizing the center}.

Instead of placing a triangle center in each component triangle, there may be other
types of points that can be used. For example, consider the following result
which comes from \cite{RG9236}. See also \cite{Josefsson2017} for generalizations.

\begin{theorem}
\label{thm:9236}
Let $ABCD$ be an equidiagonal quadrilateral with diagonal point $E$.
Locate points $F$, $G$, $H$, and $I$ inside triangles $\triangle ABE$, $\triangle BCE$, $\triangle CDE$, and
$\triangle DAE$, so that triangles $\triangle ABF$, $\triangle BCG$, $\triangle CDH$, and $\triangle DAI$
are similar isosceles triangles. Then $FGHI$ is an orthodiagonal quadrilateral.
(Figure \ref{fig:dpSimilarIsoscelesTriangles})
\end{theorem}

\begin{figure}[h!t]
\centering
\includegraphics[width=0.34\linewidth]{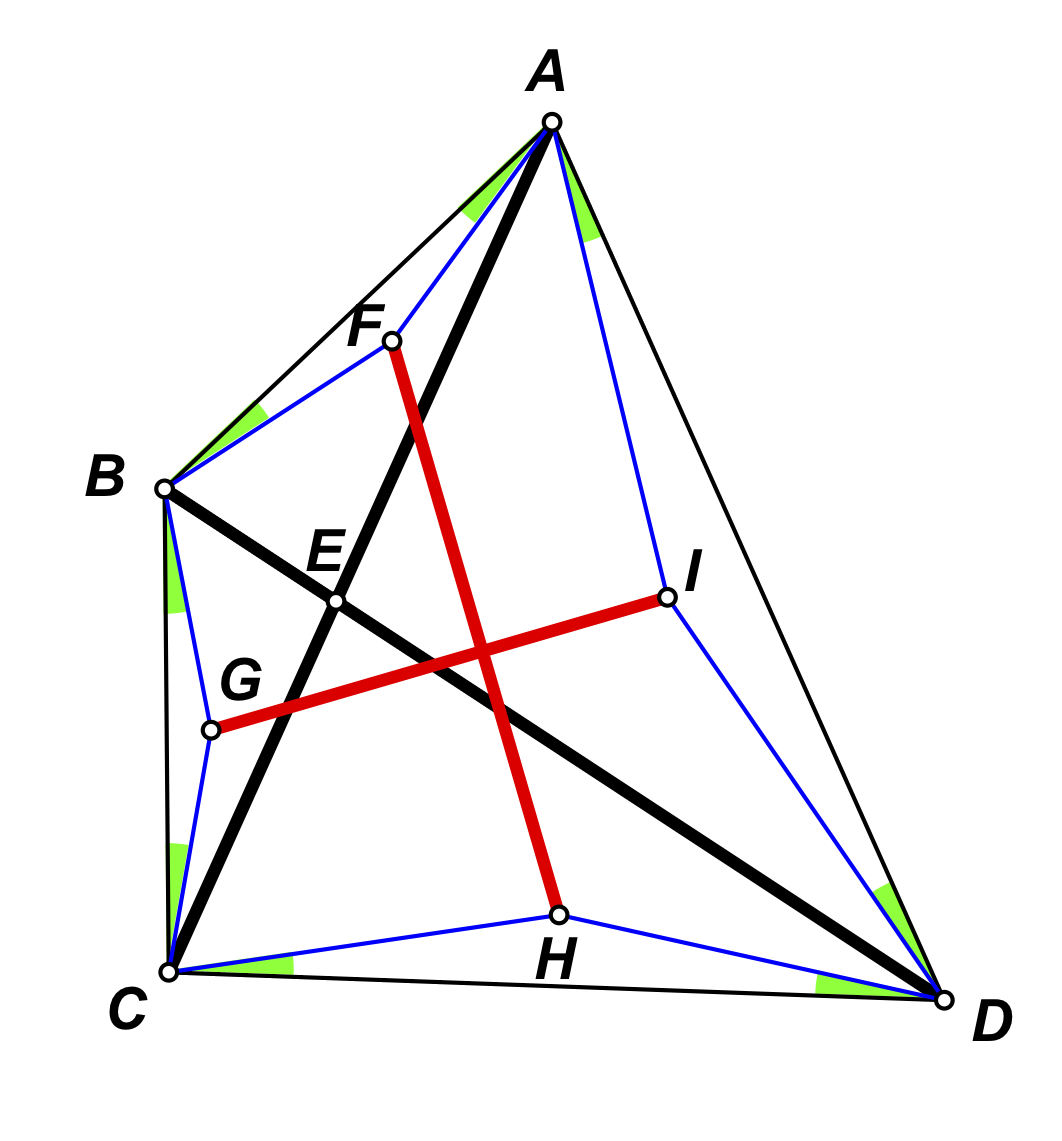}
\caption{$AC=BD \implies FH\perp GI$}
\label{fig:dpSimilarIsoscelesTriangles}
\end{figure}

\textbf{Using different component triangles}.

There are many ways of forming four component triangles from a given quadrilateral besides
using quarter triangles or half triangles. For example, if the quadrilateral is $ABCD$, we could
pick a point $P$ inside the quadrilateral and then consider the four component triangles
$\triangle ABP$, $\triangle BCP$, $\triangle CDP$, and $\triangle DAP$.

For example consider the following result due to Dao \cite{Dao}.

\begin{theorem}
\label{thm:DAO}
Let $ABCD$ be a tangential quadrilateral.
Let $E$ be a point inside the quadrilateral such that $AE=CE$ and $BE=DE$.
Let $F$, $G$, $H$, and $I$ be the incenters of triangles $\triangle ABE$, $\triangle BCE$, $\triangle CDE$, and
$\triangle DAE$, respectively. Then $FGHI$ is a cyclic quadrilateral.
(Figure \ref{fig:DaoQuasiCircumcenter})
\end{theorem}

\begin{figure}[h!t]
\centering
\includegraphics[width=0.3\linewidth]{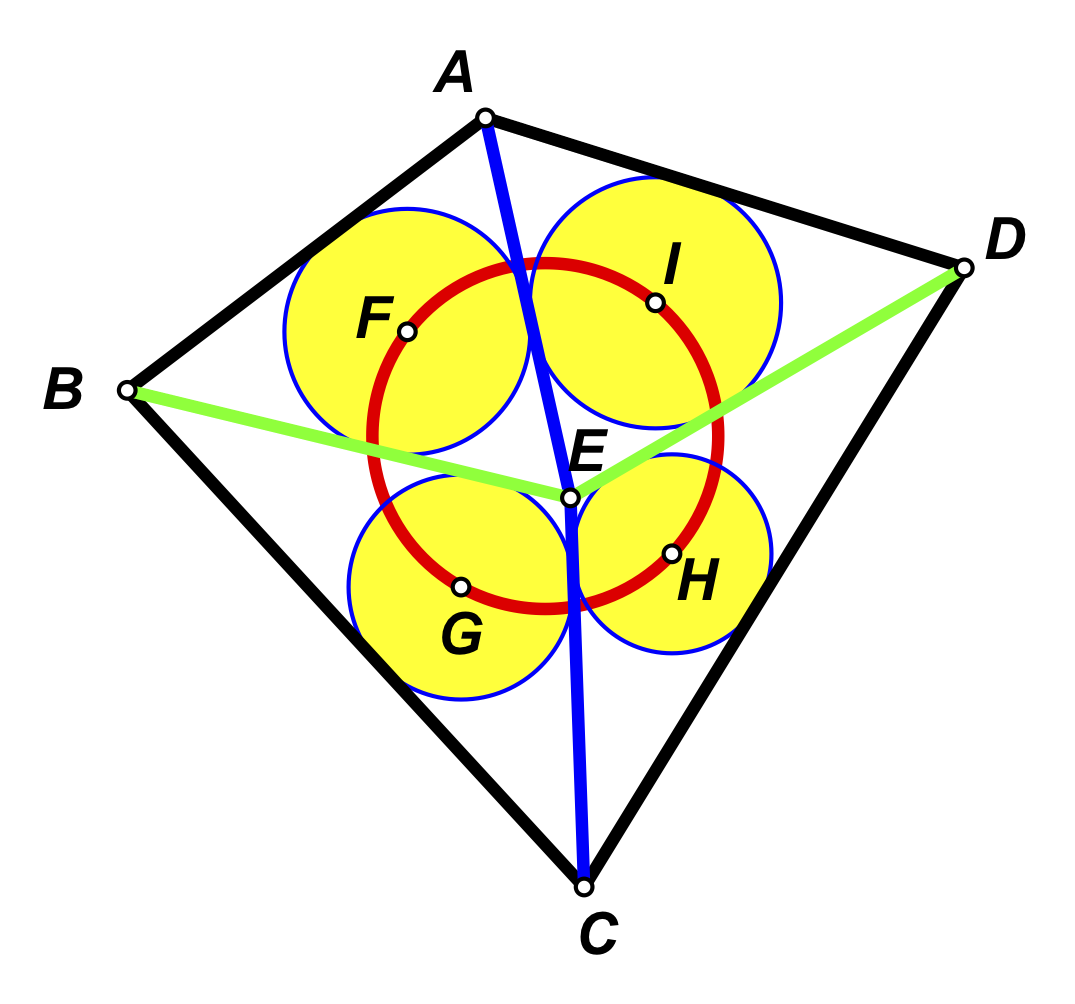}
\caption{$ABCD$ tangential, $AE$=$CE$, $BE$=$DE$ $\implies FGHI$ concyclic}
\label{fig:DaoQuasiCircumcenter}
\end{figure}

There are many other interesting points associated with a quadrilateral.
For example, we could specify that $E$ is the centroid, Poncelet point, or Steiner point
of the quadrilateral. If the quadrilateral is cyclic, we could specify that $E$
is the anticenter.

We could even ask if there are any properties when $E$ is an arbitrary point inside the quadrilateral. For example, we have the following result found by computer.

\begin{theorem}
\label{thm:arbRhombus}
Let $ABCD$ be an equidiagonal quadrilateral.
Let $E$ be an arbitrary point inside the quadrilateral.
Let $F$, $G$, $H$, and $I$ be the centroids of triangles $\triangle ABE$, $\triangle BCE$, $\triangle CDE$, and
$\triangle DAE$, respectively. Then $FGHI$ is a rhombus.
(Figure \ref{fig:arbPtRhombus})
\end{theorem}

\begin{figure}[h!t]
\centering
\includegraphics[width=0.3\linewidth]{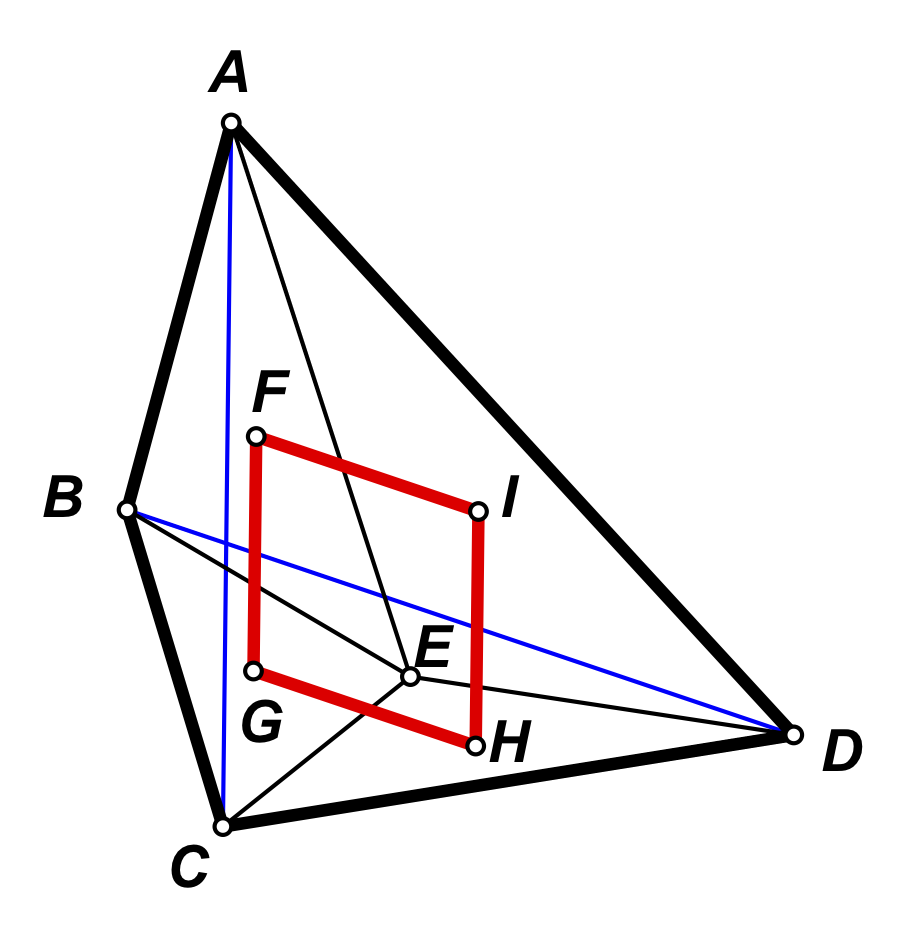}
\caption{Equidiagonal quadrilateral, centroids $\implies$ rhombus}
\label{fig:arbPtRhombus}
\end{figure}

Another possibility is to use the four triangles determined by the sides of the reference quadrilateral
taken three at a time.

\textbf{Using different centers}.

Instead of placing the same center in each of the four component triangles,
we might try placing different centers in each triangle, or different centers
in each pair of triangles. For example, the following result was found by computer.

\begin{theorem}
Let $ABCD$ be a tangential quadrilateral.
Let $E$ be the incenter of $\triangle BCD$.
Let $F$ be the centroid of $\triangle ACD$.
Let $G$ be the incenter of $\triangle ABD$.
Let $H$ be the centroid of $\triangle ABC$.
Then $EFGH$ is an orthodiagonal quadrilateral (Figure \ref{fig:twoCenters}).
\end{theorem}

\begin{figure}[h!t]
\centering
\includegraphics[width=0.4\linewidth]{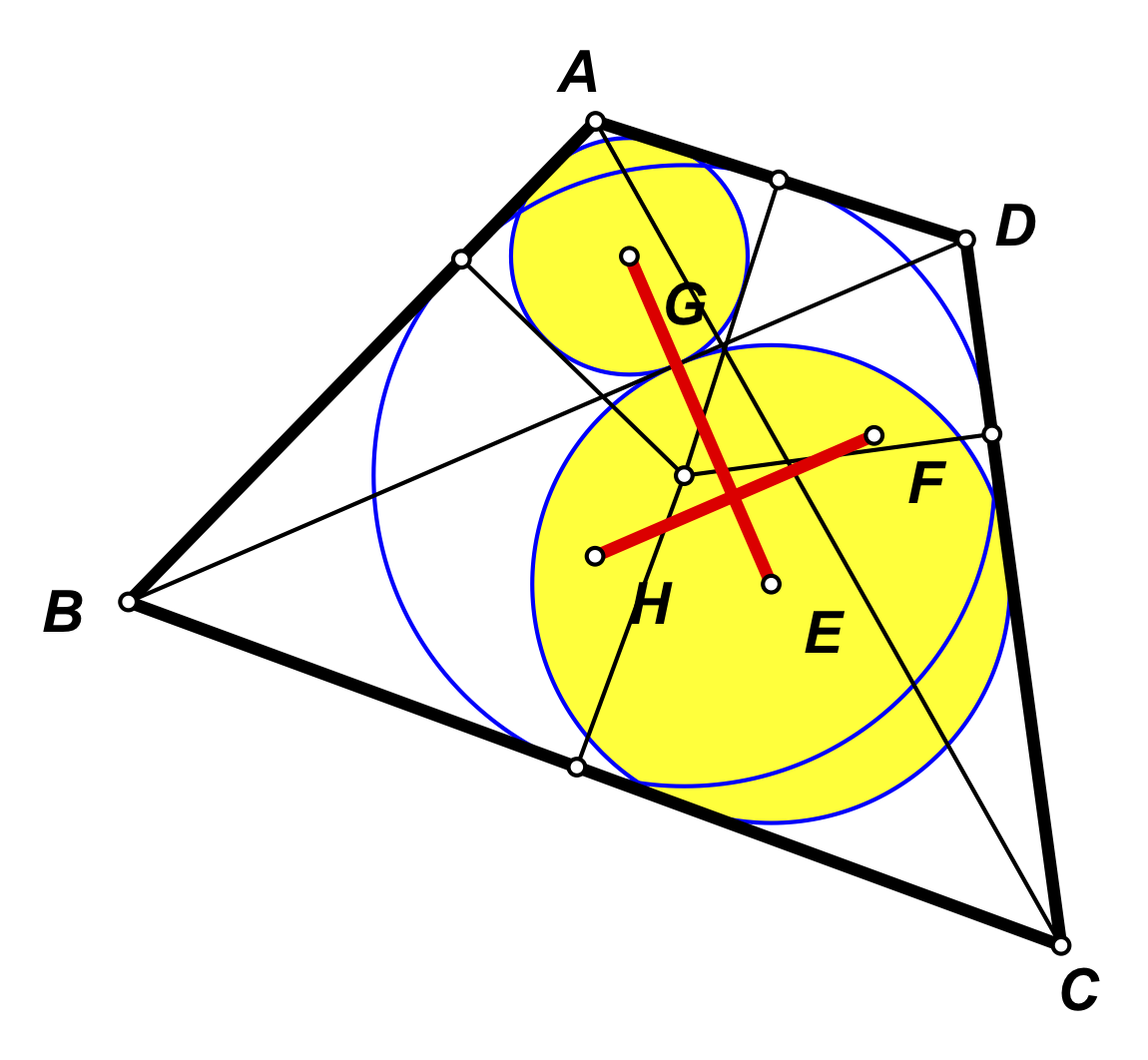}
\caption{Tangential quad., incenters/centroids $\implies EG\perp FH$}
\label{fig:twoCenters}
\end{figure}

\textbf{Checking for other properties}.

After we have placed points inside four component triangles, we can ask
other questions about the central quadrilateral besides asking about its shape.
We can compare the central quadrilateral with the reference quadrilateral and ask
questions like ``are they similar?'', ``are they congruent?'', ``do they have the
same centroid?'', ``are they homothetic?'', ``are they in perspective?'',
``do they have the same area?'', etc.

Here is an example found by computer.

\begin{theorem}
Let $ABCD$ be a rhombus with diagonal point $E$.
Let $F$, $G$, $H$ and $I$ be the Bevan point ($X_{40}$) of triangles
$\triangle ABE$, $\triangle BCE$, $\triangle CDE$, and $\triangle DAE$, respectively.
Then quadrilaterals $ABCD$ and $FGHI$ have the same area and perimeter
(Figure \ref{fig:dpBevanRhombus}).
\end{theorem}

\begin{figure}[h!t]
\centering
\includegraphics[width=0.4\linewidth]{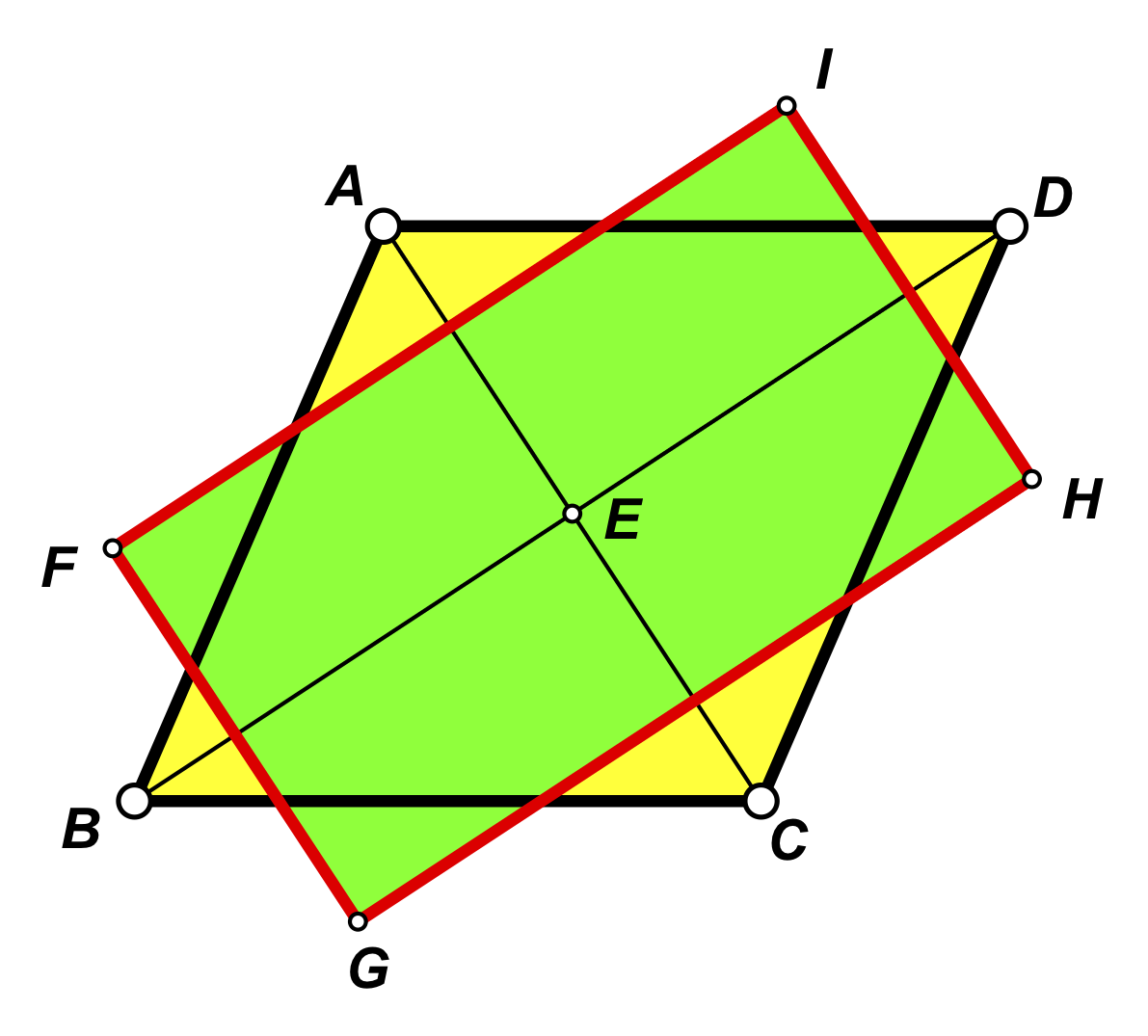}
\caption{Rhombus: Bevan points $\implies$ same area and perimeter}
\label{fig:dpBevanRhombus}
\end{figure}

We could also ask if a certain center of the reference quadrilateral coincides
with a center of the central quadrilateral. Here, ``center'' refers to some well-known
point associated with a quadrilateral, such as the centroid, Poncelet point, Steiner point,
anticenter, etc.

There are lots of interesting results still waiting to be discovered!

\newpage
\appendix

\void{
\section{Types of Triangle Centers Studied}
\label{appendix:triangleCenters}

The types of centers we might consider are:

\begin{enumerate}
\item Kimberling center $X_n$, for $1\leq n\leq 100$ (excluding points at infinity)
\item Arbitrary center (center function $f(a,g(b,c))$)
\item Major triangle center (center function $f(A)$)
\item The $r$-power points (center function $f(a,b,c)=a^r$)
\item The power centers (center function $f(a,b,c)=a^r g[b,c]$)
\item Triangle centers with center function $f(a,b,c)=a^r(b+c)$
\item Triangle centers with center function $f(a,b,c)=a^r(b^2+c^2)$
\item Triangle centers with center function $f(a,b,c)=a^r(b+c-a)$
\item Triangle centers with center function $f(a,b,c)=a^r(b+c-2a)$
\item Triangle centers with center function $f(a,b,c)=a^r(b^2+c^2-a^2)$
\item Triangle centers with center function $f(a,b,c)=a^r(b^3+c^3)$
\item Triangle centers with center function $f(a,b,c)=a^r(b^2+c^2+b c)$
\item Triangle centers with center function $f(a,b,c)=2a^r+b^r+c^r$
\item Triangle centers with center function $f(a,b,c)=(b^r+c^r)/a$
\item Triangle centers with center function $f(a,b,c)=(b^r+c^r-a^r)/a$
\item Triangle centers with center function $f(a,b,c)=(b^r+c^r+2a^r)/a$
\end{enumerate}

}

%**************************************
%    Quarter Triangles Raw Data
%**************************************

\section{Raw Data for Quarter Triangles}
\label{appendix:diagonalPointData}

\begin{theorem}%new
For a general quadrilateral,\\
(a) The central quadrilateral is orthodiagonal if $n=1$.\\
(b) The central quadrilateral is a parallelogram if $n=$2-5, 20, 140, 376, 381-382, 546-550, 631-632.
% question about 48?

No additional results were obtained for the following types of quadrilaterals:
\begin{itemize}
\item APquad%new
\item bicentric%new
%\item bilateral
%\item bisect-diagonal quadrilateral
\item equalProdAdj%new
\item equalProdOpp%new
%\item equiRecipSum
\item exbicentric%new
\item extangential%new
\item{harmonic}%new
%\item isosceles
\item Pythagorean%new
%\item right trapezoid
\item tangential trapezoid%new
\item trapezoid%new
%\item triangular
%\item trilateral
%\item Watt
%\item cyclic APquad
%\item cyclic bilateral
%\item cyclic equiRecipSum
%\item diametric
\end{itemize}
\end{theorem}

\begin{theorem}%new
For a cyclic quadrilateral,\\
(excluding results for a general quadrilateral)\\
(a) The central quadrilateral is an isosceles trapezoid if $n=110$.\\
(b) The central quadrilateral is orthodiagonal if $n=$485-486.
\end{theorem}

\begin{theorem}%new
For an equidiagonal quadrilateral,\\
(excluding results for a general quadrilateral)\\
(a) The central quadrilateral is orthodiagonal if $n=$8, 10, 40, 165, 355, 551, 946, 962.\\
%orig run had 145
(b) The central quadrilateral is a rhombus if $n=$2-5, 20, 140, 376, 381-382, 546-550, 631-632.
%orig run had 48
\end{theorem}

%\newpage
\begin{theorem}%new
For an orthodiagonal quadrilateral,\\
(excluding results for a general quadrilateral)\\
(a) The central quadrilateral is orthodiagonal if $n=$46-48, 73, 90-91, 163, 223, 226, 282, 284, 336, 380, 388, 485-486, 493-494, 497, 563, 579-581, 610, 652, 656, 820, 822, 836, 920-921, 944-951.\\
% orig run had 371-372, 
(b) The central quadrilateral is equidiagonal if $n=$151.\\
% n=151 confirmed with GGB
(c) The central quadrilateral is cyclic if $n=$6, 51, 54, 67, 70, 74,
125, 130, 184-185, 217, 287-288, 296, 389, 578, 686, 973-974.\\
%n=130 was checked with GGB
(d) The central quadrilateral is equidiagonal, orthodiagonal if $n=$102, 109, 117, 124.\\
(e) The central quadrilateral is a rectangle if $n=$2, 3, 5, 20 , 22-23, 95, 97,
122-123, 127, 131, 140, 175, 216, 233, 253, 268, 280, 339, 347, 376, 381-382, 401-402, 408,
417-418, 426, 440, 441, 454, 464-466, 546-550, 577, 631-632, 828, 852, 856, 858, 925.\\
(f) The central quadrilateral degenerates to a point for all centers that occur at the vertex
of the right angle in a right triangle.
\end{theorem}

\begin{theorem}%new
\label{thm:diagonalPoint-isoTrap}
For an isosceles trapezoid,\\
(excluding results for an equidiagonal quadrilateral)\\
(excluding results for a cyclic quadrilateral)\\
(excluding results for trapezoid)\\
%(excluding results for a Watt quadrilateral)\\
%bogus (a) The central quadrilateral is a Pythagorean, tangential, extangential, equalProdOpp quadrilateral if $n=679$.\\
(a) The central quadrilateral is a kite if $n=$1, 6, 8-12, 14,15, 17-19, 21-22, 29, 33, 35, 37-42,45, 47-48
and many more.
\end{theorem}

\begin{theorem}%new
For an orthodiagonal trapezoid,\\
(excluding results for a trapezoid)\\
(excluding results for an orthodiagonal quadrilateral)\\
(a) The central quadrilateral is orthodiagonal if $n=1$, 48, 73, 223, 226, 282, 284, 336,
371, 380, 388, 485, 493, 497, 581, 610, 820, 836, 944-951.\\
(b) The central quadrilateral is an equidiagonal quadrilateral if $n=151$.\\
% also bisect-diagonal
(c) The central quadrilateral is Hjelmslev if $n=$6, 51, 54, 67, 70, 74, 125, 130, 184-185,
217, 287-288, 296, 389, 578, 686, 973-974.\\
(d) The central quadrilateral is an equidiagonal, orthodiagonal quadrilateral if $n=$102, 124.
%[why not 109 and 117?]\\
\end{theorem}

\begin{theorem}%new
For a parallelogram,\\
%(excluding results for a Watt quadrilateral)\\
(excluding results for a trapezoid)\\
(excluding results for Pythagorean)\\
(excluding results for extangential)\\
%bogus (a) The central quadrilateral is a rhombus for $n=$1, 485-486.\\
%bogus (b) The central quadrilateral is extangential Pythagorean for $n=$48.\\
%bogus (c) The central quadrilateral is extangential trapezoid for $n=$307, 456.\\
(a) The central quadrilateral is a parallelogram for all n.
\end{theorem}

\begin{theorem}%new
For a tangential quadrilateral,\\
(excluding results for a general quadrilateral)\\
(a) The central quadrilateral is cyclic orthodiagonal if $n=1$.
\end{theorem}

\void{
\begin{theorem}
For a tangential trapezoid (tested up to 1000),\\
(excluding results for a trapezoid)\\
(excluding results for a tangential quadrilateral)\\
(a) The central quadrilateral is a bisect-diagonal quadrilateral if $n=$40.
\end{theorem}
}

\void{
\begin{theorem}%new
For a bicentric quadrilateral, (tested up to 1000)\\
(excluding results for a cyclic quadrilateral)\\
(excluding results for a tangential quadrilateral)\\
%dup (a) The central quadrilateral is cyclic if $n=$1.\\
%bogus (b) The central quadrilateral is orthodiagonal if $n=$46.
\end{theorem}

\begin{theorem}%new
For an exbicentric quadrilateral,\\
(excluding results for a cyclic quadrilateral)\\
(excluding results for a extangential quadrilateral)\\
%dup(a) The central quadrilateral is an isosceles trapezoid if $n=$110.
\end{theorem}

}

\begin{theorem}%new
For a Hjelmslev quadrilateral,\\
(excluding results for a cyclic quadrilateral)\\
(excluding results for a Pythagorean quadrilateral)\\
(a) The central quadrilateral is orthodiagonal if $n=$48, 91, 563.\\ 
(b) The central quadrilateral is a trapezoid if $n=$68, 155, 317, 577. 
\end{theorem}

\begin{theorem}%new
For a kite,\\
(excluding results for an orthodiagonal quadrilateral)\\
(excluding results for a tangential quadrilateral)\\
(excluding results for an extangential quadrilateral)\\
(excluding results for a Pythagorean quadrilateral)\\
(excluding results for an equalProdOpp quadrilateral)\\
(excluding results for an equalProdAdj quadrilateral)\\
(a) The central quadrilateral is an equidiagonal orthodiagonal trapezoid if $n=$1, 46-48, 73,
90-91, 102, 109, 117, 124, 163, 223, 226, 282, 284, 336, 380, 388, 485-486, 493-494, 497,
563-564, 579-581, 610, 652, 656, 820, 822, 836, 920-921, 944-951.\\
(b) The central quadrilateral is an isosceles trapezoid if $n$=6-18, and many more.
(c) The central quadrilateral is a rectangle if $n=$2, 3, 5, 20, 22-23, 95, 97, 122-123, 127,
131, 140, 175, 216, 233, 253, 268, 280, 339, 347, 376, 381-382, 401-402, 408, 417-418,
426, 440-441, 454, 464-466, 546-550, 577, 631-632, 828, 852, 856, 858, 925.
\end{theorem}

\void{
Note that an equidiagonal orthodiagonal trapezoid must be isosceles, so, in this case, the central quadrilateral is always an isosceles trapezoid.
}

\void{
\begin{theorem}
For a right kite (tested up to 1000),\\
(excluding results for a kite)\\
(excluding results for a Hjelmslev quadrilateral)\\
(excluding results for a bicentric quadrilateral)\\
(excluding results for an extangential quadrilateral)\\
(a) The central quadrilateral is an isosceles trapezoid if $n$=7-18, and many more.\\
(b) The central quadrilateral is a bicentric trapezoid if $n$=6, 51, 54, 65, 67, 70, 74,
125, 130, 151,184-185, 217, 287-288, 290, 296,... .\\
(c) The central quadrilateral is a rectangle FHIG if $n$=52, 110, 113, 129, 193,
247, 651, 677, 850.\\
(d) The central quadrilateral is an equidiagonal orthodiagonal trapezoid if $n$=1, 46-48,
73, 90-91, 102, 109, 117, 124, 163, 223, 226, 282, 294, 336, 371-372, 380, 388,... .\\
(e) The central quadrilateral is a parallelogram if $n$=52, 110, 113,
129, 193, 247, 651, 677, 850.\\
(f) The central quadrilateral is a rectangle FGHI if $n$=2-3, 5, 20-23, 61-62, 95, 97, 122-123, 127, 131, 140, 175, 199, 216, 233, 237, 253, 268, 280,... .\\
(g) The central quadrilateral is a line segment if $n$=96, 157, 230-231,
571, 647, 650, 676.\\
(h) The central quadrilateral is triangular if $n$=14, 16, 18, 36, 44, 46, 47, 49, 50,... .
\end{theorem}
}

\begin{theorem}%new
For a rhombus,\\
(excluding results for a parallelogram)\\
(excluding results for a tangential trapezoid)\\
(excluding results for a kite)\\
(excluding results for an orthogonal trapezoid)\\
(a) The central quadrilateral is a square if $n=$1, 46-48, 73, 90-91,
102, 109, 117, 124, 163, 223, 226, 282, 284, 336, 380, 388, 485-486, 493-494, 497,
563, 564, 579-581, 610, 652, 656, 820, 822, 836, 920-921, 944-951.\\
(b) The central quadrilateral is a rectangle if $n$=2-3, 5-18, 20-23, 31-32, 35-52, 54-63, 65, 67, 69-89, 94-101, and many more.
\end{theorem}

\begin{theorem}%new
For an equidiagonal orthodiagonal quadrilateral,\\
(excluding results for an equidiagonal quadrilateral)\\
(excluding results for an orthodiagonal quadrilateral)\\
(a) The central quadrilateral is orthodiagonal if $n=$1, 8, 10, 40, 46-48, 63,
73, 84, 90-91, 145, 163, 165, 189, 197, 223, 226-227, 255, 271, 282-284,
307, 336, 355, 380, 388, 485, 488, 493, 497, 551, 581,
610, 637, 639, 641, 820, 836, 944-951, 958, 962, 993.\\
(b) The central quadrilateral is a square if $n=$2, 3, 5, 20,22-23, 95, 97, 122-123, 127, 131, 140, 175, 216, 233, 253, 268, 280, 339, 347, 376, 381-382,
401-402, 408, 417-418, 426, 440-441, 454, 464-466, 546-550,
577, 631-632, 828, 852, 856, 858, 925.\\
(c) The central quadrilateral is cyclic if $n=$6, 51, 54, 67, 70, 74, 125, 130,
184-185, 217, 287-288, 296, 389, 578, 686, 973-974.\\
(d) The central quadrilateral is equidiagonal, orthodiagonal if $n=$102, 124, 151.
\end{theorem}

\begin{theorem}%new
For an equidiagonal orthodiagonal kite,\\
(excluding results for an equidiagonal orthodiagonal quadrilateral)\\
(excluding results for a kite)\\
(a) The central quadrilateral is an equidiagonal, orthodiagonal trapezoid if $n=$1, 8, 10, 40, 46-48,
63, 73, 84, 90-91, 102, 124, 145, 151, 165, 189, 197, 223, 226-227, 255, 271, 282-284,
307, 336, 355, 380, 388, 485, 488, 493, 497, 551, 581,
610, 637, 639, 641, 820, 836, 944-951, 958, 962, 993.\\
(b) The central quadrilateral is a square if $n=$2, 3, 5, 20, 22-23, 95, 97, 122-123, 127, 131, 140,
175, 216, 233, 253, 268, 280, 339, 347, 376, 381-382, 401-402, 408, 417-418,
426, 440-441, 454, 464-466, 546-550, 577, 631-632, 828, 852, 856, 858, 925.\\
(c) The central quadrilateral is an isosceles trapezoid if $n=$6-7, 9, 11-18, 21,
31-32, and many more values (almost all n).\\
(d) The central quadrilateral is a rectangle if $n=$224, 487, 914.\\
(e) The central quadrilateral is a line segment  if $n=$109, 117, 663.
\end{theorem}

\begin{theorem}%new
For an equidiagonal orthodiagonal trapezoid,\\
(excluding results for an equidiagonal orthodiagonal quadrilateral)\\
(excluding results for an orthogonal trapezoid)\\
(excluding results for an isosceles trapezoid)\\
(a) The central quadrilateral is an equiOrthoKite if $n=$102, 124, 151, 638, 640.\\
(b) The central quadrilateral is a kite if $n=$1, 7-17, 21, 31, 35-43, 44-45, 48,
55, 57-58, 60-63, 65, 69, 71-73, 75-83, 85, 86, 88-89, 94, 98, 103-106, and many more.\\
(c) The central quadrilateral is bicentric, exbicentric, harmonic, and Hjelmslev if $n=$6, 51, 54, 67, 70, 74, 125,
130, 184-185, 217, 287-288, 296, 389, 578, 973-974.\\
(d) The central quadrilateral is a square if $n=$2, 3, 5, 20, 22-23, 95, 97, 122-123, 127, 140,
175, 216, 233, 253, 268, 280, 339, 347, 376, 381-382,
401-402, 408, 417-418, 426, 440-441, 464-466, 546-550, 577, 631-632, 828, 852, 856, 858.
\end{theorem}

\void{
\begin{theorem}%new
For a harmonic quadrilateral,\\
(excluding results for a cyclic quadrilateral)\\
(excluding results for an equalProdOpp quadrilateral)\\
%bogus (a) The central quadrilateral is a parallelogram if $n=$61-62. 
\end{theorem}
}

\void{
\begin{theorem}
For a diamond (tested up to 20),\\
(excluding results for a rhombus)\\
(a) The central quadrilateral is a square if the center if $n=1$.\\
(b) The central quadrilateral is a rectangle if $n$=2-3, 5-15, 17, 20.
\end{theorem}
}

\begin{theorem}
For a rectangle,\\
(excluding results for a harmonic quadrilateral)\\
(excluding results for a Hjelmslev quadrilateral\\
(excluding results for an isosceles trapezoid)\\
(excluding results for a parallelogram)\\
(a) The central quadrilateral is a rhombus for all triangle centers. 
\end{theorem}

\begin{theorem}
For a square,\\
(excluding results for an equidiagonal orthogonal trapezoid)\\
(excluding results for an equidiagonal orthogonal kite)\\
(excluding results for a cyclic orthodiagonal quadrilateral)\\
(excluding results for a bicentric trapezoid)\\
(excluding results for a rhombus)\\
(excluding results for a rectangle)\\
(excluding results for an exbicentric quadrilateral)\\
%(a) The central quadrilateral is cyclic for $n=$68.\\
%(b) The central quadrilateral is harmonic for $n=$921.\\
%(c) The central quadrilateral is Hjelmslev for $n=$90, 155, 247, 642.\\
(a) The central quadrilateral is a square for all triangle centers. 
\end{theorem}

%**************************************
%    Half Triangles Raw Data
%**************************************

\section{Raw Data for Half Triangles}
\label{appendix:halfTriangleData}

\begin{theorem}
For a general quadrilateral,\\
no results were found.
\end{theorem}

\begin{theorem}
For a cyclic quadrilateral,\\
(excluding results for a general quadrilateral)\\
(a) The central quadrilateral is a rectangle for $n=$1, 40, 165.\\
(b) The central quadrilateral is a line segment for $n=$155.\\
(c) $EFGH$ is cyclic for $n=$2, 4, 5, 13-16, 20, 23, 26, 36, 74, 80, 98-100\\
and a large number of other $n$.\\
Missing: 6-12, 17-19, 21, 22, 24, 25, 27-29...
\end{theorem}

\begin{theorem}
For an equalProdOpp quadrilateral,\\
(excluding results for a general quadrilateral)\\
(a) The central quadrilateral is equalProd for $n=$2, 3, 5, 15.
\end{theorem}

\begin{theorem}
For an equalProdAdj quadrilateral,\\
(excluding results for a general quadrilateral)\\
no new results were found.
\end{theorem}

\begin{theorem}
For an equidiagonal quadrilateral,\\
(excluding results for a general quadrilateral)\\
(a) The central quadrilateral is equidiagonal for $n=$2.
\end{theorem}

\begin{theorem}
For an extangential quadrilateral,\\
(excluding results for general quadrilateral)\\
(a) The central quadrilateral is extangential for $n=$2, 3, 5.
\end{theorem}

\begin{theorem}
For an orthodiagonal quadrilateral,\\
(excluding results for a general quadrilateral)\\
(a) The central quadrilateral is orthodiagonal for $n=$2-5, 20, 51,
140, 376, 381, 382, 546-550, 631-632.
\end{theorem}

\begin{theorem}
For a Pythagorean quadrilateral,\\
(excluding results for a general quadrilateral)\\
(a) The central quadrilateral is Pythagorean for $n=$2.
\end{theorem}

\begin{theorem}
For a trapezoid,\\
(excluding results for a general quadrilateral)\\
(a) The central quadrilateral is a trapezoid for $n=$2-5, 20, 140, 376, 381, 382,
546-550, 631-632.
% what about 51?
% suppa had rhombus for n=3
\end{theorem}

\begin{theorem}
For a tangential quadrilateral,\\
(excluding results for a general quadrilateral)\\
(a) The central quadrilateral is tangential for $n=$2, 3, 5.
\end{theorem}

\begin{theorem}
For an APquad,\\
(excluding results for an extangential quadrilateral)\\
(a) The central quadrilateral is an APquad for $n=$2.
\end{theorem}

\begin{theorem}
For an equidiagonal orthodiagonal quadrilateral,\\
(excluding results for an equidiagonal quadrilateral)\\
(excluding results for an orthodiagonal quadrilateral)\\
(a) The central quadrilateral is an equidiagonal orthodiagonal quadrilateral for $n=$2, 51, 373.\\
(b) The central quadrilateral is a line segment for $n=$642.\\
(c) The central quadrilateral is orthodiagonal for $n=$486.
\end{theorem}

\begin{theorem}
For an orthodiagonal trapezoid,\\
(excluding results for a trapezoid)\\
(excluding results for an orthodiagonal quadrilateral)\\
(a) The central quadrilateral is orthodiagonal for $n=$2-5, 20, 140, 376, 381, 382,
546-550, 631-632.\\
(b) The central quadrilateral is orthodiagonal for $n=$51-53, 143.
\end{theorem}

\begin{theorem}
For an isosceles trapezoid,\\
(excluding results for an equidiagonal quadrilateral)\\
(excluding results for a cyclic quadrilateral)\\
(excluding results for trapezoid)\\
(a) The central quadrilateral is a rectangle for $n=$1, 40, 165.\\
(b) The central quadrilateral appears to be an isosceles trapezoid for just about
all other $n$.
\end{theorem}

\begin{theorem}
For a cyclic orthodiagonal quadrilateral,\\
(excluding results for a cyclic quadrilateral)\\
(excluding results for an orthodiagonal quadrilateral)\\
(a) The central quadrilateral is an isosceles trapezoid for $n=$26.\\
(b) The central quadrilateral is an isosceles trapezoid for $n=$53, 143, 373.\\
(c) The central quadrilateral is a line segment for $n=$63.\\
(d) The central quadrilateral is a trapezoid for $n=$577.
\end{theorem}

\begin{theorem}
For a tangential trapezoid,\\
(excluding results for a tangential quadrilateral)\\
(excluding results for a trapezoid)\\
(a) The central quadrilateral is a tangential trapezoid for $n=$2, 3, 5.
\end{theorem}

\begin{theorem}
For a tangential trapezoid, ?????\\
(excluding results for a general quadrilateral)\\
(a) The central quadrilateral is orthodiagonal for $n=$2-5, 20, 140, 376, 381, 382,
546-550, 631-632.\\
(b) The central quadrilateral is cyclic for $n=$2-5, 20, 140, 376, 381, 382,
% what about 51?
\end{theorem}

\begin{theorem}
For a kite,\\
(excluding results for an orthodiagonal quadrilateral)\\
(excluding results for a tangential quadrilateral)\\
(excluding results for an extangential quadrilateral)\\
(excluding results for a Pythagorean quadrilateral)\\
(excluding results for an equalProdOpp quadrilateral)\\
(excluding results for an equalProdAdj quadrilateral)\\
(a) The central quadrilateral is a kite for all $n$.
% non-convex for $n=$8, 11, 14, 16, 18, 22, 23, 29, 33, 36, 44, 47, 49...
\end{theorem}

\begin{theorem}
For a bicentric quadrilateral,\\
(excluding results for a tangential quadrilateral)\\
(excluding results for cyclic quadrilateral)\\
(a) The central quadrilateral is bicentric for $n=$2, 4, 5, 20, 140, 376, 381, 382,
546-550, 631, 632.
\end{theorem}

\begin{theorem}
For a harmonic quadrilateral,\\
(excluding results for a cyclic quadrilateral)\\
(excluding results for an equalProdOpp quadrilateral)\\
(a) The central quadrilateral is harmonic for $n=$2, 4, 5, 15, 16, 20, 23, 125, 140, 186,
265, 376, 381, 382, 546-550, 631, 632.
\end{theorem}

\begin{theorem}
For a parallelogram,\\
(excluding results for Pythagorean)\\
(excluding results for extangential)\\
(a) The central quadrilateral is a rhombus for $n=$10, 639-642.\\
(b) The central quadrilateral appears to be a parallelogram for all other $n$.
\end{theorem}

\begin{theorem}
For a Hjelmslev quadrilateral,\\
(excluding results for a cyclic quadrilateral)\\
(excluding results for a Pythagorean quadrilateral)\\
(a) The central quadrilateral is Hjelmslev for $n=$2, 4, 5, 20, 102, 109, 140, 376,
381, 382, 546-550, 631, 632, 930.\\
% Suppa did not find n=4
(b) The central quadrilateral is a parallelogram for $n=$53.\\
(c) The central quadrilateral is a line segment for $n=$97.
\end{theorem}

\begin{theorem}
For an exbicentric quadrilateral,\\
(excluding results for an extangential quadrilateral)\\
(excluding results for a cyclic quadrilateral)\\
(a) The central quadrilateral is exBicentric for $n=$2, 4, 5, 20, 140, 376,
381, 382, 546-550, 631, 632.
\end{theorem}

\begin{theorem}
For a bicentric trapezoid,\\
(excluding results for an isosceles quadrilateral)\\
(excluding results for an tangential quadrilateral)\\
(excluding results for a bicentric quadrilateral)\\
(a) The central quadrilateral is a bicentric trapezoid for $n=$2, 4, 5, 20, 140,
376, 381, 382, 546-550, 631, 632.\\
(b) The central quadrilateral appears to be an isosceles trapezoid for just about
all other $n$.
\end{theorem}

\begin{theorem}
For a rhombus,\\
(excluding results for a parallelogram)\\
(excluding results for a tangential trapezoid)\\
(excluding results for a kite)\\
(excluding results for an orthogonal trapezoid)\\
(a) The central quadrilateral is a rhombus for all $n$.\\
(b) The central quadrilateral is a square for $n=$354.
\end{theorem}

\begin{theorem}
For an equidiagonal orthodiagonal kite,\\
(excluding results for an equidiagonal orthodiagonal quadrilateral)\\
(excluding results for a kite)\\
(a) The central quadrilateral is a bicentric exbicentric harmonic Hjelmslev quadrilateral for $n=$620.\\
(b) The central quadrilateral is an equidiagonal orthodiagonal kite for $n=$2, 51, 373, 638.\\
(c) The central quadrilateral is a line segment for $n=$615.
\end{theorem}

\begin{theorem}
For an equidiagonal orthodiagonal trapezoid,\\
(excluding results for an equidiagonal orthodiagonal quadrilateral)\\
(excluding results for an orthogonal trapezoid)\\
(excluding results for an isosceles trapezoid)\\
(a) The central quadrilateral is a bicentric trapezoid for $n=$157.\\
(b) The central quadrilateral is cyclic for $n=$70, 565, 657, 770.\\
(c) The central quadrilateral is an equidiagonal orthodiagonal trapezoid for
$n=$2, 4, 5, 20, 51-53, 68, 91, 96, 135, 137, 140, 143, 373, 376, 381, 382, 389,
486, 546-550, 568, 571, 631, 632, 847, 925.\\
(d) The central quadrilateral is a rectangle for $n=$1, 26, 40, 131, 165, 577, 578.\\
(e) The central quadrilateral is a  trapezoid for $n=$70, 565, 657, 770.\\
(f) The central quadrilateral is a line segment for $n=$63, 136, 494, 642.\\
(g) The central quadrilateral appears to be an isosceles trapezoid for most other $n$.
\end{theorem}

\begin{theorem}
For a rectangle,\\
(excluding results for a harmonic quadrilateral)\\
(excluding results for a Hjelmslev quadrilateral\\
(excluding results for an isosceles trapezoid)\\
(excluding results for a parallelogram)\\
(a) The central quadrilateral is a point for $n=$3, 97, 122, 123, 127, 131, 216, 268, 339,
408, 417, 418, 426, 440, 441, 454, 464-466, 577, 828, 852, 856.\\
(b) The central quadrilateral is a square for $n=$10, 117, 124, 197, 227, 355,
639-642, 958, 993.\\
(c) The central quadrilateral appears to be rectangle for most other $n$.
\end{theorem}

\begin{theorem}
For a square,\\
(excluding results for an equidiagonal orthogonal trapezoid,
an equidiagonal orthogonal kite,
a cyclic orthodiagonal quadrilateral,
a bicentric trapezoid,
an exbicentric quadrilateral,
a rhombus,
and a rectangle)\\
(a) The central quadrilateral is a point for $n=$3, 11, 48, 49, 63, 69, 71-73, 77, 78, 97,
115, 116, 122-125, 127, 130, 137, 184, 185,
201, 212, 216, 217, 219, 222, 228, 244-246, 248, 255, 265, 268,
271, 283, 287, 293, 295, 296, 304-307, 326, 328, 332, 336-339, 
343, 345, 348, 394, 408, 417, 418, 426, 440, 441, 464-466, 
488, 499, 577, 591, 603, 606, 615, 640, 682, 748, 820, 828, 836, 
852, 856, 865-868, 895, 974.\\
(b) The central quadrilateral appears to be a square all other $n$.
\end{theorem}

\newpage
\goodbreak

%\printindex

\begin{thebibliography}{99}

\bibitem{Altshiller-Court}
Nathan Altshiller-Court,
\textit{College Geometry},
2nd edition.
Barnes \& Noble, Inc. New York: 1952.

\bibitem{Dao}
Dao Thank Oai, \textit{Four incenters lie on a circle-Does this theorem have a name?},
in the mathoverflow Stack Exchange, Nov. 2021.\\
\texttt{https://mathoverflow.net/questions/408027/four-incenters-lie-on-a-circle\-does-this-theorem-have-a-name}

\bibitem{Grinberg}
Darij Grinberg, \textit{Circumscribed quadrilaterals revisited}. 2021.\\
\url{http://www.cip.ifi.lmu.de/~grinberg/CircumRev.pdf}.

\bibitem{Grozdev}
Sava Grozdev and Deko Dekov, \textit{Barycentric Coordinates: Formula Sheet},
International Journal of Computer Discovered Mathematics, \textbf{1}(2016)75--82.\\
\url{http://www.journal-1.eu/2016-2/Grozdev-Dekov-Barycentric-Coordinates-pp.75-82.pdf}

\bibitem{Josefsson2011}
Martin Josefsson, \textit{More Characterizations of Tangential Quadrilaterals}, Forum Geometricorum \textbf{11}(2011)65--82.\\
\url{https://forumgeom.fau.edu/FG2009volume9/FG200910.pdf}

\bibitem{Josefsson2017}
Martin Josefsson, \textit{Generalisation of a quadrilateral duality theorem}. The Mathematical Gazette 101(2017)208--213.

\bibitem{KimberlingA}
Clark Kimberling,
\textit{Central Points and Central Lines in the Plane of a Triangle},
Mathematics Magazine, 67(1994)163--187.

\bibitem{KimberlingB}
Clark Kimberling,
\textit{Triangle Centers and Central Triangles},
Congressus Numerantium, 129(1998)1--295.

\bibitem{ETC}
Clark Kimberling,
\textit{Encyclopedia of Triangle Centers}, 2022.\\
\url{http://faculty.evansville.edu/ck6/encyclopedia/ETC.html}

\bibitem{Malcheski}
Risto Malcheski, Sava Grozdev, and Katerina Anevska, \textit{Geometry of Complex Numbers}.
Union of mathematicians od Macedonia, Skopje: 2017.

\bibitem{Pop}
Ovidiu T. Pop, Nicu\c{s}or Minculete, Mih\'aly Bencze,
\textit{An introduction to quadrilateral geometry},
Editura Didactic\u{a} \c{s}i Pedagogic\u{a}.\\
\url{https://www.edituradp.ro/carte/an-introduction-to-quadrilateral-geometry--i1030}.

\bibitem{RG9236}
Stanley Rabinowitz, \textit{Problem 9236}. Romantics of Geometry Facebook group. Dec. 29, 2021.
\url{https://www.facebook.com/groups/parmenides52/posts/4702670449846624/}

\void{
\bibitem{MathWorld-EulerLine}
Eric W.~Weisstein, \textit{Euler Line}. From MathWorld--A Wolfram Web Resource.
\url{https://mathworld.wolfram.com/EulerLine.html}
}

\bibitem{Yiu2012}
Paul Yiu, \textit{Introduction to the Geometry of the Triangle}, Florida Atlantic University lecture notes, December 2012.\\
\url{http://math.fau.edu/Yiu/YIUIntroductionToTriangleGeometry121226.pdf}


% Whitworth: https://hdl.handle.net/2027/coo.31924059323034

\end{thebibliography}
\end{document}